\documentclass[11pt,twoside,a4paper]{article}
\usepackage{color,amscd,amsmath,amssymb,amsthm,latexsym,stmaryrd,authblk,mathabx,shuffle,hyperref,tikz,tkz-tab} 
\usepackage[latin1]{inputenc}
\usepackage[T1]{fontenc}   
\usepackage[english]{babel}
\providecommand{\keywords}[1]{\textbf{\textit{Keywords.}} #1}
\providecommand{\AMSclass}[1]{\textbf{\textit{AMS classification.}} #1}

\newcommand{\rond}[1]{*++[o][F-]{#1}}

\newcommand{\bun}{\begin{tikzpicture}[line cap=round,line join=round,>=triangle 45,x=.2cm,y=.2cm]
\clip(-0.2,0.) rectangle (0.2,1.);
\draw [line width=.5pt] (0.,0.)-- (0.,1.);
\end{tikzpicture}}

\newcommand{\bdeux}{
\begin{tikzpicture}[line cap=round,line join=round,>=triangle 45,x=0.2cm,y=0.2cm]
\clip(-1.2,0.) rectangle (1.2,2.);
\draw [line width=.5pt] (0.,0.)-- (0.,1.);
\draw [line width=.5pt] (0.,1.)-- (-1.,2.);
\draw [line width=.5pt] (0.,1.)-- (1.,2.);
\end{tikzpicture}}

\newcommand{\btroisun}{
\begin{tikzpicture}[line cap=round,line join=round,>=triangle 45,x=0.2cm,y=0.2cm]
\clip(-2.2,0.) rectangle (1.2,3.);
\draw [line width=.5pt] (0.,0.)-- (0.,1.);
\draw [line width=.5pt] (0.,1.)-- (-1.,2.);
\draw [line width=.5pt] (0.,1.)-- (1.,2.);
\draw [line width=.5pt] (-1.,2.)-- (0.,3.);
\draw [line width=.5pt] (-1.,2.)-- (-2.,3.);
\end{tikzpicture}}
\newcommand{\btroisdeux}{
\begin{tikzpicture}[line cap=round,line join=round,>=triangle 45,x=0.2cm,y=0.2cm]
\clip(-1.2,0.) rectangle (2.2,3.);
\draw [line width=.5pt] (0.,0.)-- (0.,1.);
\draw [line width=.5pt] (0.,1.)-- (-1.,2.);
\draw [line width=.5pt] (0.,1.)-- (1.,2.);
\draw [line width=.5pt] (1.,2.)-- (2.,3.);
\draw [line width=.5pt] (1.,2.)-- (0.,3.);
\end{tikzpicture}}

\newcommand{\bquatreun}{
\begin{tikzpicture}[line cap=round,line join=round,>=triangle 45,x=0.2cm,y=0.2cm]
\clip(-3.2,0.) rectangle (1.2,4.);
\draw [line width=.5pt] (0.,0.)-- (0.,1.);
\draw [line width=.5pt] (0.,1.)-- (-1.,2.);
\draw [line width=.5pt] (0.,1.)-- (1.,2.);
\draw [line width=.5pt] (-1.,2.)-- (0.,3.);
\draw [line width=.5pt] (-1.,2.)-- (-2.,3.);
\draw [line width=.5pt] (-3.,4.)-- (-2.,3.);
\draw [line width=.5pt] (-2.,3.)-- (-1.,4.);
\end{tikzpicture}}
\newcommand{\bquatredeux}{
\begin{tikzpicture}[line cap=round,line join=round,>=triangle 45,x=0.2cm,y=0.2cm]
\clip(-2.2,0.) rectangle (1.2,4.);
\draw [line width=.5pt] (0.,0.)-- (0.,1.);
\draw [line width=.5pt] (0.,1.)-- (-1.,2.);
\draw [line width=.5pt] (0.,1.)-- (1.,2.);
\draw [line width=.5pt] (-1.,2.)-- (-2.,3.);
\draw [line width=.5pt] (-1.,2.)-- (0.,3.);
\draw [line width=.5pt] (-1.,4.)-- (0.,3.);
\draw [line width=.5pt] (0.,3.)-- (1.,4.);
\end{tikzpicture}}
\newcommand{\bquatretrois}{
\begin{tikzpicture}[line cap=round,line join=round,>=triangle 45,x=0.2cm,y=0.2cm]
\clip(-1.2,0.) rectangle (2.2,4.);
\draw [line width=0.4pt] (0.,0.)-- (0.,1.);
\draw [line width=0.4pt] (0.,1.)-- (-1.,2.);
\draw [line width=0.4pt] (0.,1.)-- (1.,2.);
\draw [line width=0.4pt] (1.,2.)-- (2.,3.);
\draw [line width=0.4pt] (1.,2.)-- (0.,3.);
\draw [line width=0.4pt] (-1.,4.)-- (0.,3.);
\draw [line width=0.4pt] (0.,3.)-- (1.,4.);
\end{tikzpicture}}
\newcommand{\bquatrequatre}{
\begin{tikzpicture}[line cap=round,line join=round,>=triangle 45,x=0.2cm,y=0.2cm]
\clip(-1.2,0.) rectangle (3.2,4.);
\draw [line width=.5pt] (0.,0.)-- (0.,1.);
\draw [line width=.5pt] (0.,1.)-- (-1.,2.);
\draw [line width=.5pt] (0.,1.)-- (1.,2.);
\draw [line width=.5pt] (1.,2.)-- (0.,3.);
\draw [line width=.5pt] (1.,2.)-- (2.,3.);
\draw [line width=.5pt] (1.,4.)-- (2.,3.);
\draw [line width=.5pt] (2.,3.)-- (3.,4.);
\end{tikzpicture}}
\newcommand{\bquatrecinq}{
\begin{tikzpicture}[line cap=round,line join=round,>=triangle 45,x=0.2cm,y=0.2cm]
\clip(-2.2,0.) rectangle (2.2,3.);
\draw [line width=.5pt] (0.,0.)-- (0.,1.);
\draw [line width=.5pt] (0.,1.)-- (-2.,3.);
\draw [line width=.5pt] (0.,1.)-- (1.,2.);
\draw [line width=.5pt] (1.,2.)-- (2.,3.);
\draw [line width=.5pt] (1.,2.)-- (0.5,3.);
\draw [line width=.5pt] (-0.5,3.)-- (-1.,2.);
\end{tikzpicture}}

\newcommand{\bdtroisun}{
\begin{tikzpicture}[line cap=round,line join=round,>=triangle 45,x=0.2cm,y=0.2cm]
\clip(-3,0.) rectangle (1.2,4.);
\draw [line width=.5pt] (0.,0.)-- (0.,1.);
\draw [line width=.5pt] (0.,1.)-- (-1.,2.);
\draw [line width=.5pt] (0.,1.)-- (1.,2.);
\draw [line width=.5pt] (-1.,2.)-- (0.,3.);
\draw [line width=.5pt] (-1.,2.)-- (-2.,3.);
\draw(-1.,1.2) node {\tiny 1};
\end{tikzpicture}}
\newcommand{\bdtroisdeux}{
\begin{tikzpicture}[line cap=round,line join=round,>=triangle 45,x=0.2cm,y=0.2cm]
\clip(-2.4,0.) rectangle (2.2,4.);
\draw [line width=.5pt] (0.,0.)-- (0.,1.);
\draw [line width=.5pt] (0.,1.)-- (-1.,2.);
\draw [line width=.5pt] (0.,1.)-- (1.,2.);
\draw [line width=.5pt] (1.,2.)-- (2.,3.);
\draw [line width=.5pt] (1.,2.)-- (0.,3.);
\draw(1.,1.2) node {\tiny 1};
\end{tikzpicture}}
\newcommand{\bdquatreun}{
\begin{tikzpicture}[line cap=round,line join=round,>=triangle 45,x=0.2cm,y=0.2cm]
\clip(-3.2,0.) rectangle (1.2,4.);
\draw [line width=.5pt] (0.,0.)-- (0.,1.);
\draw [line width=.5pt] (0.,1.)-- (-1.,2.);
\draw [line width=.5pt] (0.,1.)-- (1.,2.);
\draw [line width=.5pt] (-1.,2.)-- (0.,3.);
\draw [line width=.5pt] (-1.,2.)-- (-2.,3.);
\draw [line width=.5pt] (-3.,4.)-- (-2.,3.);
\draw [line width=.5pt] (-2.,3.)-- (-1.,4.);
\draw(-1.,1.2) node {\tiny 1};
\draw(-2.2,2.2) node {\tiny 2};
\end{tikzpicture}}
\newcommand{\bdquatredeux}{
\begin{tikzpicture}[line cap=round,line join=round,>=triangle 45,x=0.2cm,y=0.2cm]
\clip(-2.2,0.) rectangle (1.2,4.);
\draw [line width=.5pt] (0.,0.)-- (0.,1.);
\draw [line width=.5pt] (0.,1.)-- (-1.,2.);
\draw [line width=.5pt] (0.,1.)-- (1.,2.);
\draw [line width=.5pt] (-1.,2.)-- (-2.,3.);
\draw [line width=.5pt] (-1.,2.)-- (0.,3.);
\draw [line width=.5pt] (-1.,4.)-- (0.,3.);
\draw [line width=.5pt] (0.,3.)-- (1.,4.);
\draw(-1.,1.2) node {\tiny 1};
\draw(-0.1,2.2) node {\tiny 2};
\end{tikzpicture}}
\newcommand{\bdquatretrois}{
\begin{tikzpicture}[line cap=round,line join=round,>=triangle 45,x=0.2cm,y=0.2cm]
\clip(-1.2,0.) rectangle (2.2,4.);
\draw [line width=0.4pt] (0.,0.)-- (0.,1.);
\draw [line width=0.4pt] (0.,1.)-- (-1.,2.);
\draw [line width=0.4pt] (0.,1.)-- (1.,2.);
\draw [line width=0.4pt] (1.,2.)-- (2.,3.);
\draw [line width=0.4pt] (1.,2.)-- (0.,3.);
\draw [line width=0.4pt] (-1.,4.)-- (0.,3.);
\draw [line width=0.4pt] (0.,3.)-- (1.,4.);
\draw(1.,1.2) node {\tiny 1};
\draw(0.1,2.2) node {\tiny 2};
\end{tikzpicture}}
\newcommand{\bdquatrequatre}{
\begin{tikzpicture}[line cap=round,line join=round,>=triangle 45,x=0.2cm,y=0.2cm]
\clip(-1.2,0.) rectangle (3.2,4.);
\draw [line width=.5pt] (0.,0.)-- (0.,1.);
\draw [line width=.5pt] (0.,1.)-- (-1.,2.);
\draw [line width=.5pt] (0.,1.)-- (1.,2.);
\draw [line width=.5pt] (1.,2.)-- (0.,3.);
\draw [line width=.5pt] (1.,2.)-- (2.,3.);
\draw [line width=.5pt] (1.,4.)-- (2.,3.);
\draw [line width=.5pt] (2.,3.)-- (3.,4.);
\draw(1.,1.2) node {\tiny 1};
\draw(2.2,2.2) node {\tiny 2};
\end{tikzpicture}}
\newcommand{\bdquatrecinq}{
\begin{tikzpicture}[line cap=round,line join=round,>=triangle 45,x=0.2cm,y=0.2cm]
\clip(-2.2,0.) rectangle (2.2,3.);
\draw [line width=.5pt] (0.,0.)-- (0.,1.);
\draw [line width=.5pt] (0.,1.)-- (-2.,3.);
\draw [line width=.5pt] (0.,1.)-- (1.,2.);
\draw [line width=.5pt] (1.,2.)-- (2.,3.);
\draw [line width=.5pt] (1.,2.)-- (0.5,3.);
\draw [line width=.5pt] (-0.5,3.)-- (-1.,2.);
\draw(-1.,1.2) node {\tiny 1};
\draw(1.,1.2) node {\tiny 2};
\end{tikzpicture}}

\newcommand{\bdcinqun}{
\begin{tikzpicture}[line cap=round,line join=round,>=triangle 45,x=0.2cm,y=0.2cm]
\clip(-3.7,0.) rectangle (1.2,5.);
\draw [line width=.5pt] (0.,0.)-- (0.,1.);
\draw [line width=.5pt] (0.,1.)-- (-1.,2.);
\draw [line width=.5pt] (0.,1.)-- (1.,2.);
\draw [line width=.5pt] (-1.,2.)-- (0.,3.);
\draw [line width=.5pt] (-1.,2.)-- (-2.,3.);
\draw [line width=.5pt] (-3.,4.)-- (-2.,3.);
\draw [line width=.5pt] (-2.,3.)-- (-1.,4.);
\draw [line width=.5pt] (-4.,5.)-- (-3.,4.);
\draw [line width=.5pt] (-3.,4.)-- (-2.,5.);
\draw(-1.,1.2) node {\tiny 1};
\draw(-2.2,2.2) node {\tiny 2};
\draw(-3.2,3.3) node {\tiny 3};
\end{tikzpicture}}
\newcommand{\bdcinqdeux}{
\begin{tikzpicture}[line cap=round,line join=round,>=triangle 45,x=0.2cm,y=0.2cm]
\clip(-3.,0.) rectangle (1.2,5.);
\draw [line width=.5pt] (0.,0.)-- (0.,1.);
\draw [line width=.5pt] (0.,1.)-- (-1.,2.);
\draw [line width=.5pt] (0.,1.)-- (1.,2.);
\draw [line width=.5pt] (-1.,2.)-- (0.,3.);
\draw [line width=.5pt] (-1.,2.)-- (-2.,3.);
\draw [line width=.5pt] (-3.,4.)-- (-2.,3.);
\draw [line width=.5pt] (-2.,3.)-- (-1.,4.);
\draw [line width=.5pt] (-2.,5.)-- (-1.,4.);
\draw [line width=.5pt] (-1.,4.)-- (0.,5.);
\draw(-1.,1.2) node {\tiny 1};
\draw(-2.2,2.2) node {\tiny 2};
\draw(-0.7,3.3) node {\tiny 3};
\end{tikzpicture}}
\newcommand{\bdcinqtrois}{
\begin{tikzpicture}[line cap=round,line join=round,>=triangle 45,x=0.2cm,y=0.2cm]
\clip(-2.2,0.) rectangle (1.2,5.);
\draw [line width=.5pt] (0.,0.)-- (0.,1.);
\draw [line width=.5pt] (0.,1.)-- (-1.,2.);
\draw [line width=.5pt] (0.,1.)-- (1.,2.);
\draw [line width=.5pt] (-1.,2.)-- (-2.,3.);
\draw [line width=.5pt] (-1.,2.)-- (0.,3.);
\draw [line width=.5pt] (-1.,4.)-- (0.,3.);
\draw [line width=.5pt] (0.,3.)-- (1.,4.);
\draw [line width=.5pt] (-2.,5.)-- (-1.,4.);
\draw [line width=.5pt] (-1.,4.)-- (0.,5.);
\draw(-1.,1.2) node {\tiny 1};
\draw(-0.1,2.2) node {\tiny 2};
\draw(-1.1,3.2) node {\tiny 3};
\end{tikzpicture}}
\newcommand{\bdcinqquatre}{
\begin{tikzpicture}[line cap=round,line join=round,>=triangle 45,x=0.2cm,y=0.2cm]
\clip(-2.2,0.) rectangle (1.7,5.);
\draw [line width=.5pt] (0.,0.)-- (0.,1.);
\draw [line width=.5pt] (0.,1.)-- (-1.,2.);
\draw [line width=.5pt] (0.,1.)-- (1.,2.);
\draw [line width=.5pt] (-1.,2.)-- (-2.,3.);
\draw [line width=.5pt] (-1.,2.)-- (0.,3.);
\draw [line width=.5pt] (-1.,4.)-- (0.,3.);
\draw [line width=.5pt] (0.,3.)-- (1.,4.);
\draw [line width=.5pt] (0.,5.)-- (1.,4.);
\draw [line width=.5pt] (1.,4.)-- (2.,5.);
\draw(-1.,1.2) node {\tiny 1};
\draw(-0.1,2.2) node {\tiny 2};
\draw(0.9,3.2) node {\tiny 3};
\end{tikzpicture}}

\newcommand{\bdcinqcinq}{
\begin{tikzpicture}[line cap=round,line join=round,>=triangle 45,x=0.2cm,y=0.2cm]
\clip(-2.7,0.) rectangle (2.2,4.);
\draw [line width=.5pt] (0.,0.)-- (0.,1.);
\draw [line width=.5pt] (0.,1.)-- (-1.,2.);
\draw [line width=.5pt] (0.,1.)-- (1.,2.);
\draw [line width=.5pt] (-1.,2.)-- (-3.,4.);
\draw [line width=.5pt] (-1.,2.)-- (0.,3.);
\draw [line width=.5pt] (0.,3.)-- (1.,4.);
\draw [line width=.5pt] (0.,3.)-- (-0.5,4.);
\draw [line width=.5pt] (-1.5,4.)-- (-2.,3.);
\draw(-1.,1.2) node {\tiny 1};
\draw(-2.,2.2) node {\tiny 2};
\draw(0.,2.2) node {\tiny 3};
\end{tikzpicture}}

\newcommand{\Y}[2]{
\bigvee_{#1,#2}}

\input{xy}
\xyoption{all}

\title{Typed binary trees and generalized dendrifom algebras}
\date{}
\author{Lo\"\i c Foissy}
\affil{\small{Fédération de Recherche Mathématique du Nord Pas de Calais FR 2956,\\
Laboratoire de Mathématiques Pures et Appliquées Joseph Liouville,\\
Université du Littoral Côte d'Opale\thanks{Centre Universitaire de la Mi-Voix, 50, 
rue Ferdinand Buisson, CS 80699,  62228 Calais Cedex, France}.\\ Email: \texttt{foissy@univ-littoral.fr}}}

\theoremstyle{plain}
\newtheorem{theo}{Theorem}
\newtheorem{lemma}[theo]{Lemma}
\newtheorem{cor}[theo]{Corollary}
\newtheorem{prop}[theo]{Proposition}
\newtheorem{defi}[theo]{Definition}

\theoremstyle{remark}
\newtheorem{remark}{Remark}
\newtheorem{notation}{Notations}
\newtheorem{example}{Example}

\newcommand{\calT}{\mathcal{T}}
\newcommand{\calP}{\mathcal{P}}
\newcommand{\K}{\mathbb{K}}
\newcommand{\Z}{\mathbb{Z}}
\newcommand{\adm}{\mathrm{Adm}}
\newcommand{\DS}{\mathbf{DS}}
\newcommand{\EDS}{\mathbf{EDS}}
\newcommand{\Sh}{\mathrm{Sh}}
\newcommand{\sh}{\mathbf{sh}}
\newcommand{\dimK}{\mathrm{dim}_\K}

\begin{document}

\maketitle

\begin{abstract}
We here both unify and generalize nonassociative structures on typed binary trees, that is to say plane binary trees
which edges are decorated by elements of a set $\Omega$. We prove that we obtain such a structure, called an $\Omega$-dendriform
structure, if $\Omega$ has four products satisfying certain axioms (EDS axioms), including the axioms of a diassociative semigroup.
This includes matching dendriform algebras introduced by Zhang, Gao and Guo and family dendriform algebras associated 
to a semigroup introduced by Zhang, Gao and Manchon , and of course dendriform algebras when $\Omega$ is reduced to 
a single element. We also give examples of EDS, including all the EDS of cardinality two; a combinatorial description
of the products of such a structure on typed binary trees, but also on words; a study of the Koszul dual of the associated operads;
and considerations on the existence of a coproduct, in order to obtain dendriform bialgebras.
\end{abstract}

\\

\AMSclass{16T30; 05C05; 18D50.}

\tableofcontents

\section*{Introduction}

Dendriform algebras are associative algebras with an associativity splitting, that is to say their associative product
can be written as a sum of two products $\prec$ and $\succ$, with the following axioms:
\begin{align*}
(x\prec y)\prec z&=x\prec (y\prec z+y\succ z),\\
(x\succ y)\prec z&=x\succ (y\prec z),\\
(x\prec y+x\succ y)\succ z&=x\succ (y\succ z).
\end{align*}
Note that summing these three relations proves that, indeed, $\prec+\succ$ is associative.
Classical examples of dendriform algebras are given by shuffle algebras, based on words, as noticed by Schützenberger
in \cite{Schutzenberger}, which justifies the terminology of noncommutative shuffle algebras used for example
in \cite{FoissyPatras}. Free dendriform algebras were first described by Loday and Ronco \cite{LodayRonco1}
and studied in \cite{AguiarSottile}: the free dendriform algebra on one generator is based on plane binary
trees, and its two products $\prec$ and $\succ$ are inductively defined using the decomposition of any plane  binary tree
(except the unit $\bun$) into a left and a right plane binary tree. 
For example, here are plane binary trees with $k=2$, $3$ or $4$ leaves:
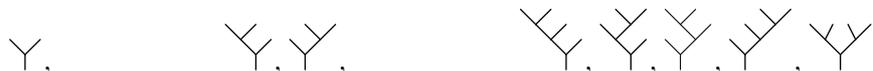
\begin{align*}
&\bdeux,&&\btroisun,\btroisdeux,&&\bquatreun,\bquatredeux,\bquatretrois,\bquatrequatre,\bquatrecinq.
\end{align*}
Here are examples of products on plane binary trees:
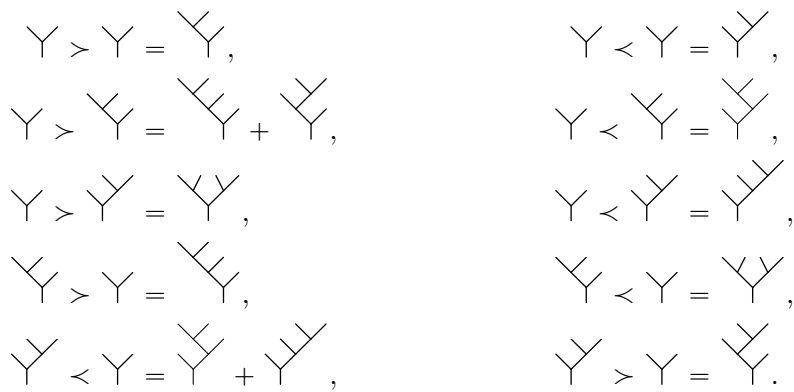
\begin{align*}
\bdeux\succ \bdeux&=\btroisun,&\bdeux\prec \bdeux&=\btroisdeux,\\
\bdeux\succ \btroisun&=\bquatreun+\bquatredeux,&\bdeux \prec \btroisun&=\bquatretrois,\\
\bdeux\succ \btroisdeux&=\bquatrecinq,&\bdeux\prec \btroisdeux&=\bquatrequatre,\\
\btroisun\succ \bdeux&=\bquatreun,&\btroisun\prec \bdeux&=\bquatrecinq,\\
\btroisdeux\prec \bdeux&=\bquatretrois+\bquatrequatre,&\btroisdeux\succ \bdeux&=\bquatredeux.
\end{align*}

Recently, a new interest in typed trees were developed in the seminal work of Bruned, Hairer and Zambotti on 
stochastic PDEs \cite{Hairer}. Given a nonempty set $\Omega$, called the set of types, an $\Omega$-typed tree
is a tree with a map from this set of edges to $\Omega$: they are related to a generalization of pre-Lie algebras \cite{FoissyPrelie}.
Similarly, several generalizations of dendriform algebras were recently introduced, where plane binary trees
are replaced by typed plane binary trees. 
\begin{itemize}
\item Firstly, if $\Omega$ is a set, an $\Omega$-matching dendriform algebras \cite{GaoGuoZhang}  is a vector space
$A$ with products $\prec_\alpha$, $\succ_\alpha$, where $\alpha \in \Omega$, such that:
\begin{align*}
(x\prec_\alpha y)\prec_\beta z&=x\prec_\alpha(y\prec_\beta z)+x\prec_\beta(y\succ_\alpha z),\\
(x\succ_\alpha y)\prec_\beta z&=x\succ_\alpha (y\prec_\beta z),\\
(x\prec_\beta y)\succ_\alpha z+(x\succ_\alpha y)\succ_{\beta} z&=x\succ_\alpha (y\succ_\beta z).
\end{align*}
\item Secondly, if $(\Omega,*)$ is a semigroup, an $\Omega$-family dendriform algebra \cite{ZhangGaoManchon} is a vector space
$A$ with products $\prec_\alpha$, $\succ_\alpha$, where $\alpha \in \Omega$, such that:
\begin{align*}
(x\prec_\alpha y)\prec_\beta z&=x\prec_{\alpha*\beta}(y\prec_\beta z+y\succ_\alpha z),\\
(x\succ_\alpha y)\prec_\beta z&=x\succ_\alpha (y\prec_\beta z),\\
(x\prec_\beta y+x\succ_\alpha y)\succ_{\alpha*\beta} z&=x\succ_\alpha (y\succ_\beta z).
\end{align*}
\end{itemize}
In both cases, it was proved that the free object on one generator is based on plane $\Omega$-typed binary trees,
with products inductively defined in a similar way as the Loday-Ronco's construction.
A plane $\Omega$-typed binary tree is a plane binary tree given a map from the set of its internal edges to $\Omega$.
We shall denote them in the following way: 
\begin{align*}
&\bdeux,&&\bdtroisun(\alpha),\bdtroisdeux(\alpha),
&&\bdquatreun(\alpha,\beta),\bdquatredeux(\alpha,\beta),\bdquatretrois(\alpha,\beta),\bdquatrequatre(\alpha,\beta),
\bdquatrecinq(\alpha,\beta),
\end{align*}
where $\alpha$, $\beta \in \Omega$. 
In all cases, the type of the internal edge $1$ is $\alpha$ and the type of the internal edge $2$ is $\beta$. 
Here are examples of products in the $\Omega$-matching case:
\begin{align*}
\bdeux\succ_{\alpha} \bdeux(\alpha)&=\bdtroisun(\alpha),&\bdeux\prec_{\alpha} \bdeux&=\bdtroisdeux(\alpha),\\
\bdeux\succ_{\alpha} \bdtroisun(\beta)&=\bdquatreun(\beta,\alpha)+\bdquatredeux(\alpha,\beta),
&\bdeux \prec_{\alpha} \bdtroisun(\beta)&=\bdquatretrois(\alpha,\beta),\\
\bdeux\succ_{\alpha} \bdtroisdeux(\beta)&=\bdquatrecinq(\alpha,\beta),
&\bdeux\prec_{\alpha} \bdtroisdeux(\beta)&=\bdquatrequatre(\alpha,\beta),\\
\bdtroisun(\alpha)\succ_{\beta} \bdeux&=\bdquatreun(\beta,\alpha),
&\bdtroisun(\alpha)\prec_{\beta} \bdeux&=\bdquatrecinq(\alpha,\beta),\\
\bdtroisdeux(\alpha)\prec_{\beta} \bdeux&=\bdquatretrois(\alpha,\beta)+\bdquatrequatre(\beta,\alpha),
&\bdtroisdeux(\alpha)\succ_{\beta} \bdeux&=\bdquatredeux(\beta,\alpha).
\end{align*}

Here are examples of products in the $\Omega$-family case:
\begin{align*}
\bdeux\succ_{\alpha} \bdeux(\alpha)&=\bdtroisun(\alpha),&\bdeux\prec_{\alpha} \bdeux&=\bdtroisdeux(\alpha),\\
\bdeux\succ_{\alpha} \bdtroisun(\beta)&=\bdquatreun(\alpha*\beta,\alpha)+\bdquatredeux(\alpha*\beta,\beta),
&\bdeux \prec_{\alpha} \bdtroisun(\beta)&=\bdquatretrois(\alpha,\beta),\\
\bdeux\succ_{\alpha} \bdtroisdeux(\beta)&=\bdquatrecinq(\alpha,\beta),
&\bdeux\prec_{\alpha} \bdtroisdeux(\beta)&=\bdquatrequatre(\alpha,\beta),\\
\bdtroisun(\alpha)\succ_{\beta} \bdeux&=\bdquatreun(\beta,\alpha),
&\bdtroisun(\alpha)\prec_{\beta} \bdeux&=\bdquatrecinq(\alpha,\beta),\\
\bdtroisdeux(\alpha)\prec_{\beta} \bdeux&=\bdquatretrois(\alpha*\beta,\beta)+\bdquatrequatre(\alpha*\beta,\alpha),
&\bdtroisdeux(\alpha)\succ_{\beta} \bdeux&=\bdquatredeux(\beta,\alpha).
\end{align*}

Our aim in this article is to give both a unification and a generalization of the extended dendriform structures.
We start with a set of types $\Omega$, given four operations $\leftarrow$, $\rightarrow$, $\triangleleft$, $\triangleright$.
A $\Omega$-dendriform algebra  is a vector space
$A$ with products $\prec_\alpha$, $\succ_\alpha$, where $\alpha \in \Omega$, such that:
\begin{align*}
(x\prec_\alpha y)\prec_\beta z&=x\prec_{\alpha \leftarrow \beta} (y\prec_{\alpha \triangleleft \beta} z)
+x\prec_{\alpha \rightarrow \beta} (y\succ_{\alpha \triangleright \beta} z),\\
 x\succ_\alpha (\prec_\beta z)&=(x\succ_\alpha y)\prec_\beta z,\\
 x\succ_\alpha (y\succ_\beta z)&=(x\succ_{\alpha \triangleright \beta} y)\succ_{\alpha \rightarrow \beta} z
+(x\prec_{\alpha \triangleleft \beta} y)\succ_{\alpha \leftarrow \beta} z.
\end{align*}
We recover the notion of $\Omega$-matching dendriform algebra taking: 
\begin{align*}
&\forall \alpha,\beta \in \Omega,&
\alpha \leftarrow \beta&=\alpha,&\alpha \rightarrow \beta&=\beta,\\
&&\alpha \triangleleft \beta&=\beta,&\alpha \triangleright \beta&=\alpha;
\end{align*}
and we recover the notion of $\Omega$-family  dendriform algebra taking, for any $\alpha,\beta \in \Omega$:
\begin{align*}
&\forall \alpha,\beta \in \Omega,&
\alpha \leftarrow \beta&=\alpha*\beta,&\alpha \rightarrow \beta&=\alpha*\beta,\\
&&\alpha \triangleleft \beta&=\beta,&\alpha \triangleright \beta&=\alpha;
\end{align*}
We prove in Proposition \ref{prop15} that the free $\Omega$-dendriform algebra on one generator
 is based on plane $\Omega$-typed binary trees,  with an inductive definition  of the products $\prec_\alpha$ and $\succ_\beta$, 
 if, and only if, the four operations of $\Omega$ satisfy a bunch of 15 axioms, see Definitions \ref{defi1} and \ref{defi2};
a similar result is proved for words in Proposition \ref{prop17}, giving typed versions of shuffle algebras.
 Such a structure on $\Omega$ will be called an extended diassociative semigroup (briefly, EDS); 
 in particular, the first five axioms only involve  the two operations $\leftarrow$ and $\rightarrow$:
\begin{align*}
(\alpha \leftarrow \beta)\leftarrow \gamma&=\alpha\leftarrow (\beta \leftarrow \gamma)
=\alpha \leftarrow(\beta \rightarrow \gamma),\\
 (\alpha \rightarrow \beta)\leftarrow \gamma&=\alpha \rightarrow (\beta \leftarrow \gamma),\\
(\alpha \rightarrow \beta)\rightarrow \gamma&=(\alpha \leftarrow \beta) \rightarrow \gamma
=\alpha \rightarrow (\beta \rightarrow \gamma).
\end{align*}
These axioms are ruled by the operad on diassociative algebras, which suggested our terminology.
A noticeable fact is that this operad is the Koszul dual of the dendriform operad.
Examples of EDS include the ones, denoted by $\EDS(\Omega)$, giving matching dendriform algebras;
the ones, denoted by $\EDS(\Omega,*)$, giving family dendriform algebras; and lots more. For example,
if $\Omega$ is of cardinality two, we found 24 EDS, including $\EDS(\Omega)$
and 5 coming from associative semigroups. 

We prove that any $\Omega$-dendriform algebra $A$ gives a dendriform algebra structure on the space $\K\Omega \otimes A$
(Proposition \ref{prop18}): this was already known in the case of $\Omega$-matching dendriform algebras \cite{GaoZhang}.
The converse implication is true under a condition of nondegeneracy of the EDS $\Omega$. 

The description of free $\Omega$-dendriform algebras induces a combinatorial description of their operad.
When $\Omega$ is finite, this is a quadratic finitely generated operad, which Koszul dual is described in Proposition \ref{prop27}.
This operad is not always Koszul, and we produce a necessary condition ($\Omega$ should be weakly nondegenerate,
Definition \ref{defi28}) and a sufficient condition on it ($\Omega$ should be nondegenerate, see Definition \ref{defi4})
for the associated operad to be Koszul. For example,  $\EDS(\Omega)$ is nondegenerate; 
if $(\Omega,*)$ is a finite associative semigroup, then $\EDS(\Omega,*)$ is nondegenerate if, and only $(\Omega,*)$
is a group. \\

We also give a study of these objects, from a Hopf-algebraic and a combinatorial point of view.
In particular, we give a description of the products on trees and on words in Propositions \ref{prop33}
and \ref{prop36}, generalizing in the latter case the usual half-shuffle products. 
Shuffle algebras and the Loday-Ronco algebra are known to be Hopf algebras; this is not always true for $\Omega$-dendriform algebras,
as described in Proposition \ref{prop38}. If $\Omega$ is nondegenerate, then such a structure exists
on trees and on words (Propositions \ref{prop39} and \ref{prop40}),
which is combinatorially described in Propositions \ref{prop42} and \ref{prop43}.
These coproducts generalize the Loday-Ronco coproduct on trees and the deconcatenation coproducts on words.\\

This paper is organized as follows: the first section is devoted to the study of EDS.
We give examples based on (diassociative) monoids, and semidirect products of groups.
We also introduce nondegenerate EDS, with a reformulation of their axioms
due to a transformation of the four defining operations into four other ones; this allows to associate to any group $G$
a nondegenerate EDS $\EDS^*(G)$ (Proposition \ref{prop9}). We prove some results on particular
families of EDS: for example, we give in Proposition \ref{prop10} all nondegenerate finite $\Omega$
such that if $\alpha$ and $\beta \in \Omega$,
\[\alpha \leftarrow \beta=\beta \rightarrow \alpha=\alpha.\]
We also give in this section
a complete classification of EDS of cardinality 2 (24 objects, which 4 are nondegenerate). 

The second section is devoted to the definition of $\Omega$-dendriform algebras and to the structure on trees and words,
when $\Omega$ is an EDS. The operadic aspects are considered in the next section,
with in particular the results on the Koszulity; we also study the associative products and the dendriform products
(that is to say, morphisms from the operad of associative algebras and from the the operad of dendriform algebras)
in $\Omega$-dendriform algebras in particular cases of $\Omega$. 

We finally give a combinatorial description of the products in Section 5 and the last section is devoted to the existence of the coproducts
and their combinatorial descriptions.\\

\textbf{Acknowledgements}. The author is grateful to Professor Xing Gao, his team and Lanzhou University for their warm hospitality.\\

\begin{notation}
$\K$ is a commutative field. All the vector spaces in this text will be taken over $\K$.
If $S$ is a set, we denote by $\K S$ the vector space generated by $S$.
\end{notation}

\section{(Extended) diassociative semigroups}

\subsection{Diassociative semigroups}

\begin{defi}\label{defi1}
A diassociative semigroup is a family $(\Omega,\leftarrow,\rightarrow)$,
where $\Omega$ is a set and $\leftarrow,\rightarrow:\Omega\times \Omega \longrightarrow \Omega$
are maps such that, for any $\alpha,\beta,\gamma\in \Omega$:
\begin{align}
\label{eq1} (\alpha \leftarrow \beta)\leftarrow \gamma&=\alpha\leftarrow (\beta \leftarrow \gamma)
=\alpha \leftarrow(\beta \rightarrow \gamma),\\
\label{eq2} (\alpha \rightarrow \beta)\leftarrow \gamma&=\alpha \rightarrow (\beta \leftarrow \gamma),\\
\label{eq3} (\alpha \rightarrow \beta)\rightarrow \gamma&=(\alpha \leftarrow \beta) \rightarrow \gamma
=\alpha \rightarrow (\beta \rightarrow \gamma).
\end{align}\end{defi}

\begin{example}\begin{enumerate}
\item If $(\Omega,\star)$ is an associative semigroup, then $(\Omega,\star,\star)$ is a diassociative semigroup.
\item Let $\Omega$ be a set. We put:
\begin{align*}
&\forall \alpha,\beta \in \Omega,&\alpha \leftarrow \beta&=\alpha,&\alpha \rightarrow \beta&=\beta.
\end{align*}
Then $(\Omega,\leftarrow,\rightarrow)$ is a diassociative semigroup, denoted by $\DS(\Omega)$. 
\item Let $\Omega=(\Omega,\leftarrow,\rightarrow)$ be a diassociative semigroup. 
We define two new operations on $\Omega$ by:
\begin{align*}
&\forall \alpha,\beta \in \Omega,&\alpha \leftarrow^{op}\beta&=\beta \rightarrow \alpha,&
\alpha \rightarrow^{op}\beta&=\beta \leftarrow \alpha.
\end{align*}
This defines a new diassociative semigroup $\Omega^{op}=(\Omega,\leftarrow^{op},\rightarrow^{op})$.
We shall say that $\Omega$ is commutative if $\Omega=\Omega^{op}$, that is to say:
\begin{align*}
&\forall \alpha,\beta \in \Omega,&\alpha \rightarrow \beta&=\beta \leftarrow \alpha.
\end{align*}
In other words, a commutative diassociative semigroup is a pair $(\Omega,\leftarrow)$ such that, 
for any $\alpha,\beta,\gamma \in \Omega$:
\[(\alpha \leftarrow \beta)\leftarrow \gamma=\alpha \leftarrow (\beta\leftarrow \gamma)=
(\alpha \leftarrow \gamma)\leftarrow \beta.\]
\end{enumerate}\end{example}

\subsection{Extended diassociative semigroups}

\begin{defi}\label{defi2}
An extended  diassociative semigroup (briefly, EDS) is a family $(\Omega,\leftarrow,\rightarrow,\triangleleft,\triangleright)$,
where $\Omega$ is a set and $\leftarrow,\rightarrow,\triangleleft,\triangleright:\Omega\times \Omega \longrightarrow \Omega$
are maps such that:
\begin{enumerate}
\item $(\Omega,\leftarrow,\rightarrow)$ is a diassociative semigroup.
\item For any $\alpha,\beta,\gamma \in \Omega$:
\begin{align}
\label{eq4} \alpha\triangleright (\beta \leftarrow \gamma)&=\alpha \triangleright \beta,\\
\label{eq5} (\alpha \rightarrow \beta)\triangleleft \gamma&=\beta \triangleleft \gamma,\\
\nonumber \\
\label{eq6} (\alpha \triangleleft \beta)\leftarrow ((\alpha \leftarrow \beta) \triangleleft \gamma)
&=\alpha \triangleleft (\beta \leftarrow \gamma),\\
\label{eq7} (\alpha \triangleleft \beta)\triangleleft ((\alpha \leftarrow \beta) \triangleleft \gamma)
&=\beta \triangleleft \gamma,\\
\label{eq8} (\alpha \triangleleft \beta) \rightarrow ((\alpha \leftarrow \beta) \triangleleft \gamma)
&=\alpha \triangleleft (\beta \rightarrow \gamma),\\
\label{eq9} (\alpha \triangleleft \beta) \triangleright ((\alpha \leftarrow \beta) \triangleleft \gamma)
&=\beta \triangleright \gamma,\\
\nonumber\\
\label{eq10} (\alpha \triangleright (\beta \rightarrow \gamma))\leftarrow (\beta \triangleright \gamma)
&=(\alpha \leftarrow \beta) \triangleright \gamma,\\
\label{eq11} (\alpha \triangleright (\beta \rightarrow \gamma))\triangleleft (\beta \triangleright \gamma)
&=\alpha \triangleleft \beta,\\
\label{eq12} (\alpha \triangleright (\beta \rightarrow \gamma))\rightarrow (\beta \triangleright \gamma)
&=(\alpha \rightarrow \beta)\triangleright \gamma,\\
\label{eq13} (\alpha \triangleright (\beta \rightarrow \gamma))\triangleright (\beta \triangleright \gamma)
&=\alpha \triangleright \beta.
\end{align}\end{enumerate}\end{defi}

\begin{example}\begin{enumerate}
\item Let $\Omega=(\Omega,\leftarrow,\rightarrow)$ be a diassociative semigroup. We define two products on $\Omega$ by:
\begin{align*}
&\forall \alpha,\beta \in \Omega,&\alpha \triangleleft \beta&=\beta,&\alpha \triangleright \beta&=\alpha.
\end{align*}
Then $(\Omega,\leftarrow,\rightarrow,\triangleleft,\triangleright)$ is an EDS, denoted
by $\EDS(\Omega,\leftarrow,\rightarrow)$. When $(\Omega,\leftarrow,\rightarrow)=\DS(\Omega)$,
we shall simply write $\EDS(\Omega)$.  
\item Let $\Omega=(\Omega,\star)$ be an associative semigroup. If $\triangleleft$ and $\triangleright$ are products on $\Omega$,
then $(\Omega,\star,\star,\triangleleft,\triangleright)$ is an EDS if, and only if,
for any $\alpha,\beta,\gamma \in \Omega$:
\begin{align}
\label{eq14}\alpha\triangleright (\beta \star \gamma)&=\alpha \triangleright \beta,\\
\label{eq15} (\alpha \star \beta)\triangleleft \gamma&=\beta \triangleleft \gamma,\\
\nonumber \\
\label{eq16} (\alpha \triangleleft \beta) \star(\beta \triangleleft \gamma)&=\alpha \triangleleft(\beta \star \gamma),\\
\label{eq17} (\alpha \triangleleft \beta) \triangleleft(\beta \triangleleft \gamma)&=\beta \triangleleft \gamma,\\
\label{eq18} (\alpha \triangleleft \beta) \triangleright(\beta \triangleleft \gamma)&=\beta \triangleright \gamma,\\
\nonumber \\
\label{eq19} (\alpha \triangleright \beta) \star(\beta \triangleright \gamma)&=(\alpha\star \beta) \triangleright \gamma,\\
\label{eq20} (\alpha \triangleright \beta) \triangleleft(\beta \triangleright \gamma)&=\alpha \triangleleft \beta,\\
\label{eq21} (\alpha \triangleright \beta) \triangleright(\beta \triangleright \gamma)&=\alpha \triangleright \beta.
\end{align}
\item Let $\Omega$ be a set and let $\DS(\Omega)$ be the diassociative semigroup attached to $\Omega$. 
If $\triangleleft$ and $\triangleright$ are products on $\Omega$,
then $(\Omega,\leftarrow,\rightarrow,\triangleleft,\triangleright)$ is an EDS if, and only if,
for any $\alpha,\beta,\gamma \in \Omega$:
\begin{align}
\label{eq22} (\alpha \triangleleft \beta) \triangleleft (\alpha \triangleleft \gamma)&=\beta \triangleleft \gamma,\\
\label{eq23} (\alpha \triangleleft \beta) \triangleright (\alpha \triangleleft \gamma)&=\beta \triangleright \gamma,\\
\label{eq24} (\alpha \triangleright \gamma)\triangleleft(\beta \triangleright\gamma)&=\alpha \triangleleft \beta,\\
\label{eq25} (\alpha \triangleright \gamma)\triangleright(\beta \triangleright\gamma)&=\alpha \triangleright \beta.
\end{align}
\item  Let $\Omega=(\Omega,\leftarrow,\rightarrow)$ be a diassociative semigroup, and let $\phi_\triangleleft,\phi_\triangleright:
\Omega\longrightarrow \Omega$ be two maps. We define two products on $\Omega$ by:
\begin{align*}
&\forall \alpha,\beta \in \Omega,&
\alpha \triangleleft \beta&=\phi_\triangleleft(\beta),&\alpha \triangleright \beta&=\phi_\triangleright(\alpha).
\end{align*}
Then $(\Omega,\leftarrow,\rightarrow,\triangleleft,\triangleright)$ is an EDS if, and only if:
\begin{align}
\label{eq26}\phi_\triangleleft&=\phi_\triangleleft\circ \phi_\triangleleft=\phi_\triangleleft\circ \phi_\triangleright,\\
\label{eq27}\phi_\triangleright&=\phi_\triangleright\circ \phi_\triangleleft=\phi_\triangleright\circ \phi_\triangleright,
\end{align}
and, for any $\alpha,\beta \in \Omega$:
\begin{align}
\label{eq28}\phi_\triangleleft(\alpha \leftarrow \beta)&=\phi_\triangleleft(\alpha) \leftarrow \phi_\triangleleft(\beta),\\
\label{eq29}\phi_\triangleleft(\alpha \rightarrow \beta)&=\phi_\triangleleft(\alpha)\rightarrow\phi_\triangleleft(\beta),\\
\label{eq30}\phi_\triangleright(\alpha  \leftarrow  \beta)&=\phi_\triangleright(\alpha) \leftarrow \phi_\triangleright(\beta),\\
\label{eq31}\phi_\triangleright(\alpha \rightarrow \beta)&=\phi_\triangleright(\alpha)\rightarrow\phi_\triangleright(\beta),
\end{align}
that is to say $\phi_\triangleleft$ and $\phi_\triangleright$ are diassociative semigroup morphisms.
If so, the obtained EDS is denoted by
$\EDS(\Omega,\leftarrow,\rightarrow,\phi_\triangleleft,\phi_\triangleright)$. In particular,
\[\EDS(\Omega,\leftarrow, \rightarrow, Id_\Omega, Id_\Omega)=\EDS(\Omega, \leftarrow, \rightarrow).\]
\item Let $(\Omega,\leftarrow,\rightarrow,\triangleleft,\triangleright)$ be an EDS.
We define four new products on $\Omega$ by:
\begin{align*}
&\forall \alpha,\beta \in \Omega,&\alpha \leftarrow^{op}\beta&=\beta \rightarrow \alpha,&
\alpha \rightarrow^{op}\beta&=\beta \leftarrow \alpha,\\
&&\alpha \triangleleft^{op}\beta&=\beta \triangleright \alpha,&\alpha \triangleright^{op}\beta&=\beta \triangleleft \alpha.
\end{align*}
This defines a new diassociative semigroup $\Omega^{op}=(\Omega,\leftarrow^{op},\rightarrow^{op},
\triangleleft^{op},\triangleright^{op})$.
We shall say that $\Omega$ is commutative if $\Omega=\Omega^{op}$, that is to say, for any 
$\alpha$, $\beta$, $\gamma \in \Omega$:
\begin{align*}
\alpha \rightarrow \beta&=\beta \leftarrow \alpha,&\alpha \triangleright \beta&=\beta \triangleleft\alpha.
\end{align*}\end{enumerate}\end{example}

In the case of groups, we find semidirect products:

\begin{prop}
Let $(\Omega,\star,\star,\triangleleft,\triangleright)$ be an EDS, such that $(\Omega,\star)$
is a group. There exist three subgroups $H$, $K_\triangleleft$ and $K_\triangleright$ of $\Omega$ such that
\begin{align}
\label{eq32} \Omega&=K\rtimes H_\triangleleft=K\rtimes H_\triangleright.
\end{align}
Moreover, for any $\alpha,\beta \in \Omega$, $\alpha\triangleleft \beta$ is the canonical projection of $\beta$
 on $H_\triangleleft$ and $\alpha \triangleright \beta$ is the canonical projection  of $\alpha$ on $H_\triangleright$.
\end{prop}

\begin{proof}
Let $\alpha,\beta,\beta'\in \Omega$.  As $\Omega$ is a group, there exists $\gamma\in \Omega$
such that $\beta \star \gamma=\beta'$. By (\ref{eq14}), $\alpha \triangleright \beta'=\alpha \triangleright \beta$.
Hence, there exists a map $\phi_\triangleright:\Omega \longrightarrow \Omega$ such that:
\begin{align*}
&\forall \alpha,\beta \in \Omega,&\phi_\triangleright(\alpha)&=\alpha \triangleright \beta.
\end{align*}
Similarly, we deduce from (\ref{eq15})  the existence of a map $\phi_\triangleleft:\Omega \longrightarrow \Omega$ such that
\begin{align*}
&\forall \alpha,\beta \in \Omega,&\phi_\triangleleft(\beta)=\alpha \triangleleft \beta.
\end{align*}
By (\ref{eq28})-(\ref{eq31}), $\phi_\triangleleft$ and $\phi_\triangleright$ are group morphisms. Let us denote by
$K_\triangleleft$ and $K_\triangleright$ their respective kernels, and by $H_\triangleleft$ and $H_\triangleright$
their respective images. By (\ref{eq26}) and (\ref{eq27}), $\phi_\triangleleft^2=\phi_\triangleleft$
and $\phi_\triangleright^2=\phi_\triangleright$, so:
\begin{align*}
G&=K_\triangleleft\rtimes H_\triangleleft=K_\triangleright\rtimes H_\triangleright.
\end{align*}
Moreover, $\phi_\triangleleft$ and $\phi_\triangleright$ are the canonical projection on, respectively,
$H_\triangleleft$ and $H_\triangleright$.\\

Let $\alpha \in K_\triangleleft$. Then $\phi_\triangleleft(\alpha)=e_G$, and, by (\ref{eq27}):
\[\phi_\triangleright(\alpha)=\phi_\triangleright \circ \phi_\triangleleft(\alpha)=\phi_\triangleright(e_G)=e_G,\]
so $\alpha \in K_\triangleright$. By symmetry, $K_\triangleleft=K_\triangleright=K$. 
\end{proof}

\begin{remark}
Conversely, if (\ref{eq32}) is satisfied, the canonical projections $\phi_\triangleleft$ and $\phi_\triangleright$
satisfy (\ref{eq26})-(\ref{eq31}), so we obtain an EDS.
\end{remark}

\subsection{Nondegenerate extended diassociative semigroups}

\begin{defi} \label{defi4}
Let $\Omega=(\Omega,\leftarrow,\rightarrow,\triangleleft,\triangleright)$ be an EDS.
We define the following maps:
\begin{align} \label{eq33}
\varphi_\leftarrow&:\left\{\begin{array}{rcl}
\Omega^2&\longrightarrow&\Omega^2\\
(\alpha,\beta)&\longrightarrow&(\alpha \leftarrow \beta,\alpha \triangleleft \beta),
\end{array}\right.&
\varphi_\rightarrow&:\left\{\begin{array}{rcl}
\Omega^2&\longrightarrow&\Omega^2\\
(\alpha,\beta)&\longrightarrow&(\alpha \rightarrow \beta,\alpha \triangleright \beta).
\end{array}\right.&
\end{align}
We shall say that $\Omega$ is nondegenerate if $\varphi_\leftarrow$ and $\varphi_\rightarrow$ are bijective.
\end{defi}

The axioms of EDS can be entirely given with the help of the maps $\varphi_\leftarrow$
and $\varphi_\rightarrow$:

\begin{lemma}
Let $\Omega=(\Omega,\leftarrow,\rightarrow,\triangleleft,\triangleright)$ be a set with four products. 
We define  $\varphi_\leftarrow$ and $\varphi_\rightarrow$ by (\ref{eq33}), and we put:
\[\tau:\left\{\begin{array}{rcl}
\Omega^2&\longrightarrow&\Omega^2\\
(\alpha,\beta)&\longrightarrow&(\beta,\alpha).
\end{array}\right.\]
Then $\Omega$ is an EDS if,  and only if:
\begin{align}
\label{eq34} (\tau \otimes Id)\circ (Id \otimes \varphi_\leftarrow)\circ (\tau \otimes Id)\circ (\varphi_\rightarrow\otimes Id)
&= (\varphi_\rightarrow\otimes Id)\circ (Id \otimes \varphi_\leftarrow),\\ 
\nonumber \\
\label{eq35} (Id \otimes \varphi_\leftarrow)\circ (\tau \otimes Id)\circ (Id \otimes \varphi_\leftarrow)\otimes
(\tau \otimes Id)\circ (\varphi_\leftarrow \otimes Id)
&=(\varphi_\leftarrow\otimes Id)\circ (Id \otimes \varphi_\leftarrow),\\
\label{eq36} (Id \otimes \varphi_\rightarrow)\circ (\tau \otimes Id)\circ (Id \otimes \varphi_\leftarrow)
\circ (\tau \otimes Id)\circ (\varphi_\leftarrow\otimes Id)&=(\varphi_\leftarrow\otimes Id)\circ (Id\otimes \varphi_\rightarrow),
\end{align}
\begin{align}
\label{eq37} (Id \otimes \varphi_\leftarrow)\circ (\varphi_\rightarrow\otimes Id)\circ (Id \otimes \varphi_\rightarrow)
&=(\varphi_\rightarrow\otimes Id)\circ (Id \otimes \tau)\circ (\varphi_\leftarrow \otimes Id),\\
\label{eq38} (Id \otimes \varphi_\rightarrow)\circ (\varphi_\rightarrow\otimes Id)\circ (Id \otimes \varphi_\rightarrow)
&=(\varphi_\rightarrow\otimes Id)\circ (Id \otimes \tau)\circ (\varphi_\rightarrow \otimes Id).
\end{align}
\end{lemma}

\begin{proof} Direct computations prove that (\ref{eq34}) is equivalent to (\ref{eq2}), (\ref{eq4}) and (\ref{eq5});
(\ref{eq35}) to one of the equalities of (\ref{eq1}), (\ref{eq6}) and (\ref{eq7});
(\ref{eq36}) to the other equality of (\ref{eq1}), (\ref{eq8}) and (\ref{eq9});
(\ref{eq37}) to one of the equalities of (\ref{eq3}), (\ref{eq10}) and (\ref{eq11});
(\ref{eq38}) to the other equality of (\ref{eq3}), (\ref{eq12}) and (\ref{eq13}). \end{proof}

\begin{prop}\label{prop6}
Let $\Omega=(\Omega,\leftarrow,\rightarrow,\triangleleft,\triangleright)$ be a nondegenerate EDS.
We define four products $\curvearrowleft$, $\curvearrowright$, $\blacktriangleleft$ and $\blacktriangleright$ on $\Omega$ by:
\begin{align*}
&\forall \alpha,\beta \in \Omega,&
\varphi_\leftarrow^{-1}(\alpha,\beta)&=(\alpha \curvearrowleft \beta, \alpha\blacktriangleleft \beta),&
\varphi_\rightarrow^{-1}(\alpha,\beta)&=(\beta \curvearrowright \beta,\beta \blacktriangleright \alpha).
\end{align*}
Then, for any $\alpha,\beta,\gamma \in \Omega$:
\begin{align*}
(\alpha \curvearrowright \beta)\curvearrowleft \gamma&=\alpha \curvearrowright (\beta\curvearrowleft \gamma),\\
\alpha \blacktriangleleft (\beta \blacktriangleleft \gamma)&=(\alpha \blacktriangleleft \beta)\blacktriangleleft \gamma,&
(\alpha \blacktriangleright \beta)\blacktriangleright \gamma&=\alpha \blacktriangleright(\beta\blacktriangleright \gamma),\\
\alpha \blacktriangleright(\beta \curvearrowleft \gamma)&=\alpha \blacktriangleright \beta,&
(\alpha \curvearrowright \beta) \blacktriangleleft \gamma&=\beta \blacktriangleleft \gamma,\\
(\alpha \curvearrowleft (\beta \blacktriangleleft \gamma))\curvearrowleft(\beta \curvearrowleft \gamma)
&=\alpha \curvearrowleft \beta,&
(\alpha \curvearrowright \beta) \curvearrowright((\alpha \blacktriangleright \beta)\curvearrowright \gamma)
&=\beta \curvearrowright \gamma,\\
(\alpha \curvearrowleft (\beta \blacktriangleleft \gamma))\blacktriangleleft(\beta \curvearrowleft \gamma)
&=(\alpha\blacktriangleleft\beta) \curvearrowleft \gamma,&
(\alpha \curvearrowright \beta) \blacktriangleright((\alpha \blacktriangleright \beta)\curvearrowright \gamma)
&=\alpha \curvearrowright (\beta \blacktriangleright \gamma),\\
(\alpha \curvearrowleft (\beta \curvearrowright \gamma))\curvearrowleft (\beta \blacktriangleright \gamma)&=
\alpha \curvearrowleft \gamma,&
(\alpha \blacktriangleleft \beta)\curvearrowright ((\alpha \curvearrowleft \beta)\curvearrowright \gamma)
&=\alpha \curvearrowright \gamma,\\
(\alpha \curvearrowleft (\beta \curvearrowright \gamma))\blacktriangleleft (\beta \blacktriangleright \gamma)&=
\beta \blacktriangleright (\alpha \blacktriangleleft \gamma),&
(\alpha \blacktriangleleft \beta)\blacktriangleright ((\alpha \curvearrowleft \beta)\curvearrowright \gamma)
&=(\alpha \blacktriangleright \gamma)\blacktriangleleft\beta,\\
(\alpha \curvearrowleft \beta)\blacktriangleright \gamma&=(\alpha \blacktriangleright \gamma) \curvearrowleft \beta,&
\beta \curvearrowright(\alpha \blacktriangleleft \gamma)&=\alpha \blacktriangleleft(\beta \curvearrowright \gamma).
\end{align*}
In particular, $\blacktriangleleft$ and $\blacktriangleleft$ are associative.
\end{prop}

\begin{proof}
By (\ref{eq34}):
\[(\varphi_\leftarrow^{-1}\otimes Id)\circ (\tau \otimes Id)\circ (Id \otimes \varphi_\leftarrow^{-1})
\circ (\tau \otimes Id)=(Id \otimes \varphi^{-1}_\leftarrow)\circ (\varphi_\rightarrow^{-1}\otimes Id).\]
When applied  to $(\beta,\alpha,\gamma)$, we obtain:
\begin{align*}
(\alpha \blacktriangleright(\beta \curvearrowleft \gamma),
(\alpha \curvearrowright \beta)\curvearrowleft \gamma,
(\alpha \curvearrowright \beta) \blacktriangleleft \gamma)
&=(\alpha \blacktriangleright \gamma,\alpha \curvearrowright (\beta\curvearrowleft \gamma),\beta \blacktriangleleft \gamma).
\end{align*}
The other identities are proved in the same way, from (\ref{eq35})-(\ref{eq38}). \end{proof}

We now explore two families of nondegenerate EDS.

\begin{lemma} \label{lem7}
Let $(\Omega,\star)$ be an associative semigroup. The following conditions are equivalent:
\begin{enumerate}
\item $\EDS(\Omega,\star,\star)$ is nondegenerate.
\item For $\EDS(\Omega,\star,\star)$, $\varphi_\leftarrow$ and $\varphi_\rightarrow$ are surjective.
\item $(\Omega,\star)$ is a group. 
\end{enumerate}
If this holds, for any $\alpha,\beta \in \Omega$:
\begin{align*}
\alpha \curvearrowleft \beta&=\alpha \star \beta^{-1},& \alpha \curvearrowright \beta&=\alpha^{-1}\star \beta,\\
\alpha \blacktriangleleft\beta&=\beta,& \alpha \blacktriangleright \beta&=\alpha.
\end{align*}
\end{lemma}

\begin{proof}
Obviously, $1.\Longrightarrow 2$.\\

$2.\Longrightarrow 3$. 
For any $\alpha,\beta \in \Omega$, $\varphi_\leftarrow(\alpha,\beta)=(\alpha\star \beta,\beta)$
and $\varphi_\rightarrow(\alpha,\beta)=(\alpha \star \beta,\alpha)$. Hence, for any $\beta,\gamma \in \Omega$,
there exist $\alpha,\alpha'\in \Omega$, such that $\alpha\star\beta=\beta \star \alpha'=\gamma$. 

Let us fix $\beta_0\in \Omega$, and let us consider elements $e$ and $e'$ such that $e\star \beta_0=\beta_0 \star e'=\beta_0$. 
Let $\gamma \in \Omega$; there exists $\alpha \in \Omega$, such that $\alpha \star \beta_0=\gamma$.
Hence,
\[\gamma\star e'=\alpha \star \beta_0\star e'=\alpha \star \beta_0=\gamma.\]
Similarly, $e\star \gamma=\gamma$ for any $\gamma$. In particular $e\star e'=e=e'$ is a unit of $\Omega$. 
For any $\beta \in \Omega$, there exist $\beta',\beta''\in \Omega$, such that
\[\beta'\star \beta=\beta \star \beta''=e.\]
Moreover, $\beta'\star \beta \star \beta''=\beta'\star e=\beta'=e\star \beta''=\beta''$,
so $\beta'=\beta''$ is an inverse of $\beta$ in $\Omega$: $\Omega$ is a group. \\

$3.\Longrightarrow 1$. The inverse bijections of $\varphi_\leftarrow$ and $\varphi_\rightarrow$ are given by
\begin{align*}
\varphi_\leftarrow^{-1}(\alpha,\beta)&=(\alpha\star \beta^{-1},\beta),&
\varphi_\rightarrow^{-1}(\alpha,\beta)&=(\beta,\beta^{-1}\star \alpha).
\end{align*}
So $\Omega$ is nondegenerate. \end{proof}

\begin{prop}
Let $\Omega=(\Omega,\star,\star,\triangleleft,\triangleright)$ be an associative semigroup. 
We assume that, either $\Omega$ is finite, or either $(\Omega,\star)$ is cancellative: for any $\alpha$, $\beta$, $\gamma\in \Omega$,
\begin{align*}
(\alpha\star \beta=\alpha \star \gamma) &\Longrightarrow (\beta=\gamma),&
(\alpha\star \gamma=\beta \star \gamma)& \Longrightarrow (\alpha=\beta).
\end{align*} 
Then $\Omega$ is nondegenerate if, and only if, the two following conditions hold:
\begin{enumerate}
\item $(\Omega,\star)$ is a group.
\item $\Omega=\EDS(\Omega,\star,\star)$.
\end{enumerate}\end{prop}

\begin{proof} $1.\Longrightarrow 2$. Let us assume that $\Omega$ is nondegenerate. We consider the map:
\begin{align*}
\psi&:\left\{\begin{array}{rcl}
\Omega^3&\longrightarrow&\Omega^3\\
(\alpha,\beta,\gamma)&\longrightarrow&(\alpha \triangleright \beta,\beta \triangleright \gamma,\alpha \star \beta \star \gamma).
\end{array}\right.
\end{align*}
Let us prove that $\psi$ is injective. We denote by $\psi_\rightarrow=(\psi_\rightarrow^1,\psi_\rightarrow^2)$ the inverse
of the bijection $\varphi_\rightarrow$. We put:
\begin{align*}
\psi'&:\left\{\begin{array}{rcl}
\Omega^3&\longrightarrow&\Omega^3\\
(\alpha,\beta,\gamma)&\longrightarrow&
(\psi_\rightarrow^1(\psi_\rightarrow^1(\gamma,\alpha\star \beta),\alpha),
\psi_\rightarrow^2(\psi_\rightarrow^1(\gamma,\alpha\star \beta),\alpha),
\psi_\rightarrow^1(\gamma,\alpha\star \beta)).
\end{array}\right.
\end{align*}
Let $\alpha,\beta,\gamma \in \Omega$. We put $\psi'\circ \psi(\alpha,\beta,\gamma)=(\alpha',\beta',\gamma')$ with:
\begin{align*}
\alpha'&=\psi_\rightarrow^1(\psi_\rightarrow^1(\alpha\star  \beta\star \gamma,
(\alpha \triangleright \beta)\star (\beta \triangleright \gamma)),\alpha\triangleright \beta),\\
\beta'&=\psi_\rightarrow^2(\psi_\rightarrow^1(\alpha\star  \beta\star \gamma,
(\alpha \triangleright \beta)\star (\beta \triangleright \gamma)),\alpha\triangleright \beta),\\
\gamma'&=\psi_\rightarrow^2(\alpha\star  \beta\star \gamma,
(\alpha \triangleright \beta)\star (\beta \triangleright \gamma)).
\end{align*}
By (\ref{eq19}):
\begin{align*}
\varphi_\rightarrow(\alpha \star \beta,\gamma)&=(\alpha\star \beta \star \gamma,
(\alpha \triangleright \beta)\star (\beta \triangleright \gamma)),
\end{align*}
Therefore:
\begin{align*}
\alpha'&=\psi_\rightarrow^1(\alpha\star  \beta,\alpha\triangleright \beta)=\alpha,&
\beta'&=\psi_\rightarrow^2(\alpha\star  \beta,\alpha\triangleright \beta)=\beta,&
\gamma'&=\gamma.
\end{align*}
So $\psi$ is injective. If $\Omega$ is finite, $\psi$ is bijective. If $\Omega$ is cancellative, let us put
$\psi'(\alpha,\beta,\gamma)=(\alpha',\beta',\gamma')$.
By definition of $\psi$, the first component of $\psi(\alpha',\beta',\gamma')$ is $\alpha$
and the third one is $\gamma$. Let us denote its second component by $\beta''=\beta'\triangleright \gamma'$.
By definition of $\psi'$:
\begin{align*}
\alpha'\star \beta'\star \gamma'&=\gamma,&
(\alpha'\star \beta')\triangleright \gamma'&=\alpha \star \beta,&
\alpha'\triangleright \beta'&=\alpha.
\end{align*}
Moreover, by (\ref{eq19}):
\[\alpha \star \beta=
(\alpha'\star \beta')\triangleright \gamma'=(\alpha'\triangleright \beta')\star(\beta'\triangleright \gamma')
=\alpha \star (\beta'\triangleright \gamma')=\alpha \star \beta''.\]
As $\Omega$ is cancellative, $\beta''=\beta$, so $\psi \circ \psi'=Id_\Omega$ and $\psi$ is surjective.\\

Consequently, if $\alpha',\beta' \in \Omega$, there exist $\alpha,\beta,\gamma\in \Omega$,
such that $\alpha'=\alpha\triangleleft \beta$ and $\beta'=\beta\triangleleft \gamma$. By (\ref{eq21}):
\[\alpha'\triangleright \beta'=\alpha'.\]
We prove similarly that $\alpha'\triangleleft \beta'=\beta'$, using (\ref{eq17}). 
By Lemma \ref{lem7}, $(\Omega,\star)$ is a group. \\

$2. \Longrightarrow 1$. The inverse implication comes from Lemma \ref{lem7}. \end{proof}

\begin{prop}\label{prop9}
Let $(H,\star)$ be a group, $K$ be a nonempty set and $\theta:K\longrightarrow H$ be a map.
We define four products on $H\times K$ in the following way:
\begin{align*}
&\forall (\alpha,\alpha'), (\beta,\beta')\in H\times K,&
(\alpha,\alpha')\leftarrow(\beta,\beta')&=(\alpha,\alpha'),\\
&&(\alpha,\alpha')\rightarrow(\beta,\beta')&=(\beta,\beta'),\\
&&(\alpha,\alpha')\triangleleft(\beta,\beta')&=(\alpha^{-1}\star \beta,\beta'),\\
&&(\alpha,\alpha')\triangleright(\beta,\beta')&=(\theta(\beta')\star \beta^{-1}\star\alpha,\alpha').
\end{align*}
This defines a nondegenerate EDS denoted by $\EDS^*(H,\star,K,\theta)$.
It is commutative if, and only if, for any $\alpha'\in K$, $\theta(\alpha')$ is the unit of $H$.  For any 
$(\alpha,\alpha'), (\beta,\beta')\in H\times K$:
\begin{align*}
(\alpha,\alpha')\curvearrowleft (\beta,\beta')&=(\alpha,\alpha'),&
(\alpha,\alpha')\blacktriangleleft (\beta,\beta')&=(\alpha\star \beta,\beta'),\\
(\alpha,\alpha')\curvearrowright (\beta,\beta')&=(\beta,\beta'),&
(\alpha,\alpha')\blacktriangleright (\beta,\beta')&=(\beta \star \theta(\beta')\star \alpha,\alpha').
\end{align*}\end{prop}

\begin{proof} Note that $(H\times K,\leftarrow,\rightarrow)=\DS(H\times K)$.
Direct computations prove that (\ref{eq22})-(\ref{eq25}) are satisfied. Moreover, $\varphi_\leftarrow$ and $\varphi_\rightarrow$
are bijections, which inverses are given by:
\begin{align*}
\varphi_\leftarrow^{-1}((\alpha,\alpha'),(\beta,\beta'))
&=((\alpha,\alpha'),(\alpha\star \beta,\beta')),&
\varphi_\rightarrow^{-1}((\alpha,\alpha'),(\beta,\beta'))
&=(\alpha\star \theta(\alpha')\star \beta,\beta'),(\alpha,\alpha')).
\end{align*}
So this EDS  is nondegenerate.
\end{proof}

\begin{example}\begin{enumerate}
\item If $K$ is reduced to a single element, let us denote by $\omega$ the image of this element by $\theta$.
As a set, $\EDS^*(H,\star,\Omega,\theta)$ is identified with $H$, given the products:
\begin{align*}
\alpha \leftarrow \beta&=\alpha,&\alpha \rightarrow \beta&=\beta,\\
\alpha \triangleleft \beta&=\alpha^{-1}\star \beta,&
\alpha \triangleright \beta&=\omega \star \beta^{-1}\star \alpha.
\end{align*}
This diassociative semigroup will be denoted by $\EDS^*(H,\star,\omega)$.
It is commutative if, and only if, $\omega$ is the unit of $H$. In this case, we shall simply denote it by $\EDS^*(H,\star)$. 
\item If $H$ is a null group, we identify $H\times K$ and $K$. We obtain $\EDS(K)$. 
\end{enumerate}\end{example}

\begin{prop} \label{prop10}
Let $(\Omega,\leftarrow,\rightarrow,\triangleleft,\triangleright)$ be a finite nondegenerate EDS,
such that  $(\Omega,\leftarrow,\rightarrow)=\DS(\Omega)$.
There exist a  group $(H,\star)$, a nonempty set $K$ and a map $\theta:K\longrightarrow H$
such that $\Omega$ is isomorphic to $\EDS^*(H,\star,K,\theta)$.
\end{prop}

\begin{proof} \textit{First step.} For any $\alpha,\beta \in \Omega$, $\varphi_\leftarrow (\alpha,\beta)
=(\alpha,\alpha\triangleleft \beta)$ and $\varphi_\rightarrow(\alpha,\beta)=(\beta,\alpha \triangleright \beta)$.
 With the notations of Proposition \ref{prop6}, for any $\alpha$, $\beta \in \Omega$:
\begin{align*}
\alpha \curvearrowleft \beta&=\alpha,&\alpha \curvearrowright \beta&=\beta.
\end{align*}
Moreover, for any $\alpha$, $\beta$, $\gamma\in \Omega$:
\begin{align*}
\alpha \triangleleft \beta&=\gamma\:\Longleftrightarrow \:\beta =\alpha \blacktriangleleft \gamma,&
\alpha \triangleright \beta&=\gamma\:\Longleftrightarrow \:\alpha =\gamma \blacktriangleright \beta.
\end{align*}

The relations of Proposition \ref{prop6} simplify: for any $\alpha,\beta,\gamma\in \Omega$, 
\begin{align}
\label{eq39} \alpha\blacktriangleleft(\beta\blacktriangleleft \gamma)
&=(\alpha\blacktriangleleft\beta)\blacktriangleleft \gamma,&
 \alpha\blacktriangleright(\beta\blacktriangleright \gamma)
&=(\alpha\blacktriangleright\beta)\blacktriangleright \gamma,\\
\label{eq40} (\alpha\blacktriangleleft\beta)\blacktriangleright \gamma
&=(\alpha\blacktriangleright\gamma)\blacktriangleleft \beta,&
\alpha\blacktriangleleft(\beta\blacktriangleright \gamma)
&=\beta\blacktriangleright(\alpha\blacktriangleleft \gamma).
\end{align}

\textit{Second step}. Let us study the semigroup $(\Omega,\blacktriangleleft)$. 
For any $\alpha \in \Omega$, we consider the map
\[f_\alpha:\left\{\begin{array}{rcl}
\Omega&\longrightarrow&\Omega\\
\beta&\longrightarrow&\alpha \blacktriangleleft\beta.
\end{array}\right.\]
This is an element of the symmetric group $\mathfrak{S}(\Omega)$. 
By (\ref{eq39}), for any $\alpha,\beta \in \Omega$:
\[f_\alpha \circ f_\beta=f_{\alpha \blacktriangleleft \beta}.\]
Hence, if $H'=\{f_\alpha,\alpha \in \Omega\}$, $H'$ is a sub-semigroup of $\mathfrak{S}(\Omega)$. 
As $\Omega$ is finite, this is a subgroup of $\mathfrak{S}(\Omega)$. Consequently, the following set is nonempty:
\[K=\{\alpha \in \Omega,\:f_\alpha=Id_\Omega\}
=\{\alpha \in \Omega,\: \forall \beta \in \Omega,\:\alpha \blacktriangleleft \beta=\beta\}.\]
Let us choose $e\in K$. We consider the map
\[\psi:\left\{\begin{array}{rcl}
H'&\longrightarrow&\Omega\\
f&\longrightarrow&f(e).
\end{array}\right.\]
For any $\alpha,\beta \in \Omega$, as $e\blacktriangleleft \beta=\beta$:
\[\psi(f_\alpha \circ f_\beta)=\alpha \blacktriangleleft \beta \blacktriangleleft e
=\alpha \blacktriangleleft e\blacktriangleleft \beta \blacktriangleleft e
=f_\alpha(e)\blacktriangleleft f_\beta(e).\]
So $\psi$ is a semigroup morphism. Let us assume that $\psi(f_\alpha)=\psi(f_\beta)$.
For any $\gamma \in \Omega$:
\[f_\alpha(\gamma)=\alpha \blacktriangleleft \gamma=\alpha\blacktriangleleft  e\blacktriangleleft \gamma
=\psi(f_\alpha)\blacktriangleleft \gamma=\psi(f_\beta)\blacktriangleleft \gamma=f_\beta(\gamma),\]
so $f_\alpha=f_\beta$: $\psi$ is injective. Let us denote by $H$ its image; then $H$ is a sub-semigroup of $(\Omega,\blacktriangleleft)$
and is a group, of unit $e$. For any $\beta=f_\alpha(e)\in \Omega$, for any $\gamma \in \Omega$:
\[f_\beta(\gamma)=f_\alpha(e)\blacktriangleleft \gamma=\alpha \blacktriangleleft e\blacktriangleleft \gamma
=\alpha \blacktriangleleft \gamma=f_\alpha(\gamma),\]
so $f_\beta=f_\alpha$. Hence, the inverse of $\psi$ is:
\[\psi^{-1}:\left\{\begin{array}{rcl}
H&\longrightarrow&H'\\
\beta&\longrightarrow&f_\beta.
\end{array}\right.\]
We denote by $\star$ the product of $H$: for any $\alpha,\beta\in H$, $\alpha \star \beta=\alpha\blacktriangleleft\beta$.
We define a product on $H\times K$ by:
\[(\alpha,\alpha')\blacktriangleleft (\beta,\beta')=(\alpha \star \beta,\beta').\]
Let us consider the map
\[\Theta:\left\{\begin{array}{rcl}
H\times K&\longrightarrow&\Omega\\
(\alpha,\alpha')&\longrightarrow&\alpha\blacktriangleleft\alpha'.
\end{array}\right.\]
For any $(\alpha,\alpha'),(\beta,\beta')\in H\times K$, as $\alpha'\in K$:
\begin{align*}
\Theta((\alpha,\alpha')\blacktriangleleft(\beta,\beta'))&=\Theta(\alpha\star \beta,\beta')
=\alpha \blacktriangleleft \beta \blacktriangleleft \beta'
=\alpha \blacktriangleleft \alpha' \blacktriangleleft \beta \blacktriangleleft \beta'
=\Theta(\alpha,\alpha')\blacktriangleleft\Theta(\beta,\beta').
\end{align*}
So $\Theta$ is a semigroup morphism. 

Let $(\alpha,\alpha'),(\beta,\beta')\in H\times K$, such that $\Theta(\alpha,\alpha')=\Theta(\beta,\beta')$. 
Hence, $\alpha \blacktriangleleft \alpha'=\beta \blacktriangleleft \beta'$. Therefore, as $\alpha',\beta'\in K$:
\[f_\alpha=f_\alpha \circ f_{\alpha'}=f_{\alpha \blacktriangleleft \alpha'}
=f_{\beta \blacktriangleleft \beta'}=f_\beta\circ f_{\beta'}=f_\beta.\]
As $\psi^{-1}$ is injective, $\alpha=\beta$. Consequently:
\[\alpha'=\alpha \triangleleft(\alpha \blacktriangleleft \alpha')=\beta \triangleleft(\beta \blacktriangleleft \beta')=\beta'.\]
So $\Theta$ is injective.

Let $\gamma \in \Omega$. There exists a unique $\alpha\in H$, such that $f_\alpha=f_\gamma$.
Let $\alpha^{-1}$ be the inverse of $\alpha$ in the group $H$ and $\alpha'=\alpha^{-1}\blacktriangleleft \gamma$.
Then:
\[\alpha\blacktriangleleft \alpha=\alpha\blacktriangleleft \alpha^{-1}\blacktriangleleft \gamma
=f_{\alpha \blacktriangleleft \alpha^{-1}}(\gamma)=Id_\Omega(\gamma)=\gamma.\]
Moreover, for any $\beta \in \Omega$:
\[f_{\alpha'}=f_{\alpha^{-1}}\circ f_\gamma=f_\alpha^{-1}\circ f_\gamma=f_\gamma^{-1}\circ f_\gamma=Id_\Omega,\]
so $\alpha'\in K$ and $\gamma=\Theta(\alpha,\alpha')$: $\Theta$ is surjective.\\

From now, we assume, up to an isomorphism, that $(\Omega,\blacktriangleleft)=(H\times K,\blacktriangleleft)$.
By definition of $\blacktriangleleft$, for any $(\alpha,\alpha'),(\beta,\beta')\in H\times K$,
\[(\alpha,\alpha')\triangleleft (\beta,\beta')=(\alpha^{-1}\star \beta,\beta').\]

\textit{Last step.} Let us now study the product $\blacktriangleright$. Recall that $e$ is the unit of $H$
and let us choose $\alpha \in K$. We put, for any $(\gamma,\gamma')\in H\times K$:
\[(e,a)\blacktriangleright(\gamma,\gamma')=(\iota(\gamma,\gamma'),\iota'(\gamma,\gamma')).\]
 By (\ref{eq40}), for any $(\alpha,\alpha'),(\beta,\beta')\in H\times K$:
\[ (\beta,\beta')\blacktriangleright (\gamma,\gamma')=
((e,\alpha)\blacktriangleleft (\beta,\beta'))\blacktriangleright (\gamma,\gamma')
=((e,\alpha)\blacktriangleright (\gamma,\gamma'))\blacktriangleleft(\beta,\beta')
=(\iota(\gamma,\gamma')\star \beta,\beta').\]
Still by (\ref{eq40}):
\begin{align*}
(\alpha,\alpha')\blacktriangleleft((\beta,\beta)'\blacktriangleright(\gamma,\gamma'))
%&=(\alpha,\alpha')\blacktriangleleft (\iota(\gamma,\gamma')\star \beta,\beta')\\
&=(\alpha \star \iota(\gamma,\gamma')\star \beta,\beta')\\
&=(\beta,\beta')\blacktriangleright((\alpha,\alpha')\blacktriangleleft(\gamma,\gamma')\\
%&=(\beta,\beta,')\blacktriangleright(\alpha\star \gamma),\gamma')\\
&=(\iota(\alpha\star \gamma,\gamma')\star \beta,\beta').
\end{align*}
Hence, $\iota(\alpha\star \gamma,\gamma')=\alpha \star \iota(\gamma,\gamma')$. 
We put $\theta(\alpha')=\iota(e,\alpha')^{-1}$. For any $\alpha,\alpha'\in K$, taking $\gamma=e$:
\[\iota(\alpha,\alpha')=\alpha\star \iota(e,\alpha')=\alpha \star \theta(\alpha')^{-1}.\]
Finally, for any $(\alpha,\alpha'),(\beta,\beta')\in H\times K$:
\[(\alpha,\alpha')\blacktriangleright (\beta,\beta')=(\beta \star \theta(\beta')^{-1}\star \alpha,\alpha').\]
By definition of $\blacktriangleright$, 
$(\alpha,\alpha')\triangleright (\beta,\beta')=(\theta(\beta')\star \beta^{-1}\star \alpha,\alpha')$.
So $\Omega=\EDS^*(H,\star,K,\theta)$.  \end{proof}

\subsection{Extended diassociative semigroups of cardinality two}

Let $\Omega=\{a,b\}$ be a set of cardinality two. There are 16 maps from $\Omega^2$ to $\Omega$.
Testing all possibilities with a computer, we find 13 structures of diassociative semigroups on $\Omega$,
which restrict to 8 up to isomorphism, and 45 structures of EDS on $\Omega$,
which restrict to 24 up to isomorphism. In order to describe them, 
we shall use the maps $\phi_a,\phi_b:\Omega\longrightarrow \Omega$, such that for any $\alpha \in \Omega$:
\begin{align*}
\phi_a(\alpha)&=a,&\phi_b(\alpha)&=b.
\end{align*}
We shall meet six possible products for $\triangleleft$ and $\triangleright$, denoted by:
\begin{align*}
&\begin{array}{|c|c|c|}
\hline m_a&a&b\\
\hline a&a&a\\
\hline b&a&a\\
\hline \end{array}&
&\begin{array}{|c|c|c|}
\hline m_b&a&b\\
\hline a&b&b\\
\hline b&b&b\\
\hline \end{array}&
&\begin{array}{|c|c|c|}
\hline \triangleleft_{\EDS}&a&b\\
\hline a&a&b\\
\hline b&a&b\\
\hline \end{array}&
&\begin{array}{|c|c|c|}
\hline \triangleright_{\EDS}&a&b\\
\hline a&a&a\\
\hline b&b&b\\
\hline \end{array}&\\
&&&\begin{array}{|c|c|c|}
\hline m_1&a&b\\
\hline a&a&b\\
\hline b&b&a\\
\hline \end{array}&
&\begin{array}{|c|c|c|}
\hline m_2&a&b\\
\hline a&b&a\\
\hline b&a&b\\
\hline \end{array}&
\end{align*}
\begin{description}
\item[A.]
\begin{align*}
&\begin{array}{|c|c|c|}
\hline \leftarrow_A=\rightarrow_A&a&b\\
\hline a&a&a\\
\hline b&a&a\\
\hline \end{array}
\end{align*}
This is the diassociative semigroup attached to the semigroup such that:
\begin{align*}
&\forall \alpha,\beta\in \Omega,&\alpha \star _A\beta=a.
\end{align*}
\begin{description}
\item[A1.] $(\{a,b\},\star_A,\star_A,m_a,m_a)$. 
This is $\EDS(\{a,b\},\star_A,\star_A,\phi_a,\phi_a)$.\\ It is commutative.
\item[A2.] $(\{a,b\},\star_A,\star_A,\triangleleft_{\EDS},\triangleright_{\EDS})$.
This is $\EDS(\{a,b\},\star_A,\star_A)$.  \\ It is commutative.
\end{description}

\item[B.]
\begin{align*}
&\begin{array}{|c|c|c|}
\hline \leftarrow_B&a&b\\
\hline a&a&a\\
\hline b&a&a\\
\hline \end{array}&
&\begin{array}{|c|c|c|}
\hline \rightarrow_B&a&b\\
\hline a&a&b\\
\hline b&a&b\\
\hline \end{array}&
\end{align*}
\begin{description}
\item[B1.] $(\{a,b\},\leftarrow_B,\rightarrow_B,m_a,m_a)$.
This is $\EDS(\{a,b\},\leftarrow_B,\rightarrow_B,\phi_a,\phi_a)$. \\ It is the opposite of D1.
\item[B2.] $(\{a,b\},\leftarrow_B,\rightarrow_B,\triangleleft_{\EDS},\triangleright_{\EDS})$.
This is $\EDS(\{a,b\},\leftarrow_B,\rightarrow_B)$.\\  It is the opposite of D2.   
\end{description}

\item[C.]
\begin{align*}
&\begin{array}{|c|c|c|}
\hline \leftarrow_C=\rightarrow_C&a&b\\
\hline a&a&a\\
\hline b&a&b\\
\hline \end{array}&
\end{align*}
This is the diassociative semigroup $(\Z/2\Z,\times,\times)$, with $a=\overline{0}$ and $b=\overline{1}$.
\begin{description}
\item[C1.] $(\{a,b\},\leftarrow_C,\rightarrow_C,m_a,m_a)$.
This is $\EDS(\Z/2\Z,\times,\times,\phi_a,\phi_a)$.\\  It is commutative.
\item[C2.] $(\{a,b\},\leftarrow_C,\rightarrow_C,m_a,m_b)$.
This is $\EDS(\Z/2\Z,\times,\times,\phi_a,\phi_b)$.\\  It is the opposite of C4.
\item[C3.] $(\{a,b\},\leftarrow_C,\rightarrow_C,\triangleleft_{\EDS},\triangleright_{\EDS})$.
This is $\EDS(\Z/2\Z,\times,\times)$.\\  It is commutative.
\item[C4.] $(\{a,b\},\leftarrow_C,\rightarrow_C,m_b,m_a)$.
This is $\EDS(\Z/2\Z,\times,\times,\phi_b,\phi_a)$.\\  It is the opposite of C2.
\item[C5.] $(\{a,b\},\leftarrow_C,\rightarrow_C,m_b,m_b)$.
This is $\EDS(\Z/2\Z,\times,\times,\phi_b,\phi_b)$.\\  It is commutative.
\end{description}

\item[D.]
\begin{align*}
&\begin{array}{|c|c|c|}
\hline \leftarrow_D&a&b\\
\hline a&a&a\\
\hline b&b&b\\
\hline \end{array}&
&\begin{array}{|c|c|c|}
\hline \rightarrow_D&a&b\\
\hline a&a&a\\
\hline b&a&a\\
\hline \end{array}&
\end{align*}
\begin{description}
\item[D1.] $(\{a,b\},\leftarrow_D,\rightarrow_D,m_a,m_a)$.
This is $\EDS(\{a,b\},\leftarrow_D,\rightarrow_D,\phi_a,\phi_a)$.\\ It is the opposite of B1.
\item[D2.] $(\{a,b\},\leftarrow_D,\rightarrow_D,\triangleleft_{\EDS},\triangleright_{\EDS})$.
This is $\EDS(\{a,b\},\leftarrow_D,\rightarrow_D)$.\\  It is the opposite of B2.
\end{description}

\item[E.]
\begin{align*}
&\begin{array}{|c|c|c|}
\hline \leftarrow_E=\rightarrow_E&a&b\\
\hline a&a&a\\
\hline b&b&b\\
\hline \end{array}&
\end{align*}
This is the diassociative semigroup attached to the semigroup 
such that:
\begin{align*}
&\forall\alpha,\beta \in \Omega,&\alpha \star_E \beta&=\alpha.
\end{align*}

\begin{description}
\item[E1.] $(\{a,b\},\star_E,\star_E,m_a,m_a)$.
This is $\EDS(\{a,b\},\star_E,\star_E,\phi_a,\phi_a)$.\\  It is the opposite of G1.
\item[E2.] $(\{a,b\},\star_E,\star_E,m_a,m_b)$.
This is $\EDS(\{a,b\},\star_E,\star_E,\phi_a,\phi_b)$.\\  It is the opposite of G2.
\item[E3.] $(\{a,b\},\star_E,\star_E,\rightarrow_D,\triangleleft_{\EDS},\triangleright_{\EDS})$.
This is $\EDS(\{a,b\},\star_E,\star_E)$.\\  It is the opposite of G3.
\end{description}

\item[F.]
\begin{align*}
&\begin{array}{|c|c|c|}
\hline \leftarrow_F&a&b\\
\hline a&a&a\\
\hline b&b&b\\
\hline \end{array}&
&\begin{array}{|c|c|c|}
\hline \rightarrow_F&a&b\\
\hline a&a&b\\
\hline b&a&b\\
\hline \end{array}&
\end{align*}
This is $\DS(\{a,b\})$. 
\begin{description}
\item[F1.] $(\{a,b\},\leftarrow_F,\rightarrow_F,m_a,m_a)$.
This is $\EDS(\{a,b\},\leftarrow_F,\rightarrow_F,\phi_a,\phi_a)$.\\  It is commutative.
\item[F2.] $(\{a,b\},\leftarrow_F,\rightarrow_F,m_a,m_b)$.
This is $\EDS(\{a,b\},\leftarrow_F,\rightarrow_F,\phi_a,\phi_b)$.\\  It is commutative.
\item[F3.] $(\{a,b\},\leftarrow_F,\rightarrow_F,\triangleleft_{\EDS},\triangleright_{\EDS})$.
This is $\EDS(\{a,b\})$.\\  It is commutative.
\item[F4.] $(\{a,b\},\leftarrow_F,\rightarrow_F,m_1,m_1)$.
This is $\EDS^*(\Z/2\Z,+)$, with $a=\overline{0}$ and $b=\overline{1}$.\\   It is commutative. 
\item[F5.] $(\{a,b\},\leftarrow_F,\rightarrow_F,m_1,m_2)$.
This is $\EDS^*(\Z/2\Z,+,\overline{1})$, with $a=\overline{0}$ and $b=\overline{1}$. \\  
It is not commutative, but is isomorphic to its opposite via the map permuting $a$ and $b$. 
\end{description}

\item[G.]
\begin{align*}
&\begin{array}{|c|c|c|}
\hline \leftarrow_G=\rightarrow_G&a&b\\
\hline a&a&b\\
\hline b&a&b\\
\hline \end{array}&
\end{align*}
This is the diassociative semigroup attached to the semigroup such that:
\begin{align*}
&\forall\alpha,\beta \in \Omega,& \alpha \star_G \beta&=\beta.
\end{align*}
\begin{description}
\item[G1.] $(\{a,b\},\star_G,\star_G,m_a,m_a)$.
This is $\EDS(\{a,b\},\star_G,\star_G,\phi_a,\phi_a)$.\\  It is the opposite of E1.
\item[G2.] $(\{a,b\},\star_G,\star_G,m_a,m_b)$.
This is $\EDS(\{a,b\},\star_G,\star_G,\phi_a,\phi_b)$.\\  It is the opposite of E2.
\item[G3.] $(\{a,b\},\star_G,\star_G,\rightarrow_D,\triangleleft_{\EDS},\triangleright_{\EDS})$.
This is $\EDS(\{a,b\},\star_G,\star_G)$.\\  It is the opposite of E3.
\end{description}

\item[H.]
\begin{align*}
&\begin{array}{|c|c|c|}
\hline \leftarrow_H=\rightarrow_H&a&b\\
\hline a&a&b\\
\hline b&b&a\\
\hline \end{array}&
\end{align*}
This is the diassociative semigroup attached to the group $(\Z/2\Z,+)$, with $a=\overline{0}$ and $b=\overline{1}$. 
\begin{description}
\item[H1.] $(\{a,b\},\star_H,\star_H,m_a,m_a)$.
This is $\EDS(\Z/2\Z,+,+,\phi_a,\phi_a)$.\\  It is commutative.
\item[H2.] $(\{a,b\},\star_H,\star_H,\rightarrow_D,\triangleleft_{\EDS},\triangleright_{\EDS})$.
This is $\EDS(\Z/2\Z,+,+)$.\\  It is commutative.
\end{description}
\end{description}

Only four of these EDS are nondegenerate: F3, F4, F5, H2. 

\begin{remark}
Similar computations can be done for EDS of cardinality 3.
Up to isomorphism, there are four nondegenerate EDS of cardinality 3:
\begin{align*}
&\EDS(\{a,b,c\}),&&\EDS(\Z/3\Z,+,+),&& \EDS^*(\Z/3\Z,+),&& \EDS^*(\Z/3\Z,+,\overline{1}).
\end{align*}\end{remark}

\section{$\Omega$-dendriform algebras}

\subsection{Definition and example}

\begin{defi}
Let $\Omega$ be a set with four products $\leftarrow, \rightarrow, \triangleleft, \triangleright$. 
An $\Omega$-dendriform algebra is a family $(A,(\prec_\alpha)_{\alpha \in \Omega},(\succ_\alpha)_{\alpha \in  \Omega})$
where $A$ is a vector space and $\prec_\alpha,\succ_\alpha:A\otimes A\longrightarrow A$,
such that for any $x,y,z\in A$, for any $\alpha,\beta \in \Omega$:
\begin{align}
\label{eq41} (x\prec_\alpha y)\prec_\beta z&=x\prec_{\alpha \leftarrow \beta} (y\prec_{\alpha \triangleleft \beta} z)
+x\prec_{\alpha \rightarrow \beta} (y\succ_{\alpha \triangleright \beta} z),\\
\label{eq42} x\succ_\alpha (\prec_\beta z)&=(x\succ_\alpha y)\prec_\beta z,\\
\label{eq43} x\succ_\alpha (y\succ_\beta z)&=(x\succ_{\alpha \triangleright \beta} y)\succ_{\alpha \rightarrow \beta} z
+(x\prec_{\alpha \triangleleft \beta} y)\succ_{\alpha \leftarrow \beta} z.
\end{align}
\end{defi}

\begin{example}\begin{enumerate}
\item If $(\Omega,\star)$ is a semigroup, we recover the definition of \emph{dendriform family algebra} 
\cite{ZhangGaoManchon} when we consider $\EDS(\Omega,\star,\star)$:
\begin{align*}
\alpha \leftarrow \beta&=\alpha \rightarrow \beta=\alpha \star \beta,&
\alpha \triangleleft \beta&=\beta,&\alpha \triangleright \beta&=\alpha.
\end{align*}
Note that in this case, $(\Omega,\leftarrow,\rightarrow)$ is an EDS.
\item For any set $\Omega$, considering $\EDS(\Omega)$, we recover the definition of \emph{matching dendriform algebras}
\cite{GaoGuoZhang}.
\end{enumerate}
\end{example}

\begin{remark}
Let $A$ be an $\Omega$-dendriform algebra. For any $a,b\in A$, for any $\alpha \in \Omega$, we put:
\begin{align*}
a\prec_\alpha^{op} b&=b\succ_\alpha a,&
a\succ_\alpha^{op} b&=b\prec_\alpha a.
\end{align*}
Then $(A,(\prec^{op}_\alpha)_{\alpha \in \Omega}, (\succ^{op}_\alpha)_{\alpha \in \Omega})$
is an $\Omega^{op}$-dendriform algebra, where the products of $\Omega^{op}$ are defined by:
\begin{align*}
&\forall \alpha,\beta \in \Omega,&\alpha \leftarrow^{op}\beta&=\beta \rightarrow \alpha,&
\alpha \rightarrow^{op}\beta&=\beta \leftarrow \alpha,\\
&&\alpha \triangleleft^{op}\beta&=\beta \triangleright \alpha,&
\alpha \triangleright^{op}\beta&=\beta \triangleleft \alpha.
\end{align*}
This gives the notion of commutative  $\Omega$-dendriform algebra:
\end{remark}

\begin{defi}
Let $\Omega$ be a set with four products $\leftarrow, \rightarrow, \triangleleft, \triangleright$ such that, 
for any $\alpha$, $\beta\in \Omega$:
\begin{align*}
\alpha \leftarrow\beta&=\beta \rightarrow \alpha,&\alpha \triangleleft\beta&=\beta \triangleright \alpha.
\end{align*}
Let $A$ be an $\Omega$-dendriform algebra. We shall say that $A$ is commutative if for any $\alpha$, $\beta\in \Omega$,
for any $a$, $b\in A$:
\begin{align*}
a\prec_\alpha b&=b\succ_\alpha a.
\end{align*} \end{defi}

\subsection{Structures on typed binary trees}

\begin{defi}
Let $\Omega$ be a set. 
\begin{enumerate}
\item An $\Omega$-typed binary tree is a pair $(T,\tau)$, where $T$ is a plane binary tree and
$\tau$ is a map from the set on internal edges of $T$ to $\Omega$. 
For any internal edge $e$ of $T$, $\tau(e)$ is called the type of $e$.
\item The set of $\Omega$-typed binary trees is denoted by $\calT_\Omega$. 
We denote by $\calT_\Omega^+$ the set of $\Omega$-types binary trees different from the trivial tree $\bun$.
\item For any $n\geq 0$, the set of $\Omega$-typed binary trees with $n$ internal vertices (and $n+1$ leaves) 
is denoted by $\calT_\Omega(n)$. 
\end{enumerate}
Consequently:
\begin{align*}
\calT_\Omega&=\bigsqcup_{n\geqslant 0} \calT_\Omega(n),&
\calT_\Omega^+&=\bigsqcup_{n\geqslant 1} \calT_\Omega(n).
\end{align*}\end{defi}

\begin{example}
Here are plane binary trees with $n\leq 3$ leaves:
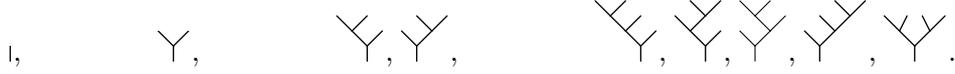
\begin{align*}
&\bun,&&\bdeux,&&\btroisun,\btroisdeux,&&\bquatreun,\bquatredeux,\bquatretrois,\bquatrequatre,\bquatrecinq.
\end{align*}
For any $T=(T,\tau)\in \calT_\Omega$, we shall give indices to internal edges and indicate their types in this way:
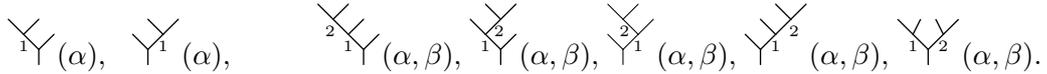
\begin{align*}
&\bdtroisun(\alpha),\bdtroisdeux(\alpha),
&\bdquatreun(\alpha,\beta),\bdquatredeux(\alpha,\beta),\bdquatretrois(\alpha,\beta),\bdquatrequatre(\alpha,\beta),
\bdquatrecinq(\alpha,\beta).
\end{align*}
In all cases, the type of the internal edge $1$ is $\alpha$ and the type of the internal edge $2$ is $\beta$. 
\end{example}

\begin{defi}
Let $T_1,T_2\in \calT_\Omega$, and $\alpha,\beta\in \Omega$. We denote by 
$\displaystyle T_1 \Y{\alpha}{\beta} T_2$ the tree $T\in \calT_\Omega$ obtained by grafting $T_1$ on the left and $T_2$ and the right
on a common root. If $T_1\neq \bun$, the type of the internal edge between the root of $T$ and the root of $T_1$ is  $\alpha$.
If $T_2\neq \bun$, the type of  internal edge between the root of $T$ and the root of $T_2$ is $\beta$.
\end{defi}

\begin{example}
For example, for any $\alpha,\beta,\gamma \in \Omega$:
\begin{align*}
\bdquatreun(\alpha,\beta)&=\bdtroisun(\beta) \Y{\alpha}{\gamma}\bun,&
\bdquatredeux(\alpha,\beta)&=\bdtroisdeux(\beta) \Y{\alpha}{\gamma}\bun,\\
\bdquatretrois(\alpha,\beta)&=\bun \Y{\gamma}{\alpha} \bdtroisun(\beta),&
\bdquatrequatre(\alpha,\beta)&=\bun \Y{\gamma}{\alpha} \bdtroisdeux(\beta),\\
\bdquatrecinq(\alpha,\beta)&=\bdeux \Y{\alpha}{\beta}\bdeux.
\end{align*}\end{example}

\begin{remark}
Note that any element $T\in \calT_\Omega(n)$, with $n\geq 1$, can be written under the form
\[T=T_1\Y{\alpha}{\beta} T_2,\]
with $T_1,T_2\in \calT_\Omega$, $\alpha,\beta \in \Omega$. This writing is unique except if 
$T_1=\mid$ or $T_2=\mid$: in this case, one can change arbitrarily $\alpha$ or $\beta$.
In order to solve this notational problem, we add an element denoted by $\emptyset$ to $\Omega$
and we shall always assume that if $T_1=\mid$, then $\alpha=\emptyset$;
if $T_2=\mid$, then $\beta=\emptyset$.
\end{remark}

\begin{prop}\label{prop15}
Let $\Omega$ be a set with four products $\leftarrow, \rightarrow, \triangleleft, \triangleright$. 
We define products $\prec_\alpha$ and $\succ_\alpha$  on $\K \calT_\Omega^+$, for $\alpha\in \Omega$,
by the following recursive formulas: for any $T,T_1,T_2\in \calT_\Omega^+$, for any $\alpha,\beta,\gamma\in \Omega$,
\begin{align*}
\bdeux\prec_\alpha T&=\bun \Y{\emptyset}{\alpha} T,\\
(T_1\Y{\alpha}{\emptyset} \bun)\prec_\beta T&=T_1\Y{\alpha}{\beta} T,\\
\left(T_1\Y{\alpha}{\beta}T_2\right)\prec_\gamma T
&=T_1\Y{\alpha}{\beta\leftarrow \gamma} (T_2\prec_{\beta\triangleleft \gamma} T)
+T_1\Y{\alpha}{\beta\rightarrow \gamma} (T_2\succ_{\beta\triangleright \gamma} T),\\ \\
T\succ_\alpha \bdeux&=T\Y{\alpha}{\emptyset}\bun,\\
T\succ_\alpha (\bun \Y{\emptyset}{\beta} T_2)&=T\Y{\alpha}{\beta} T_2,\\
T\succ_\alpha \left(T_1 \Y{\beta}{\gamma}T_2\right)&=(T\succ_{\alpha\triangleright \beta}T_1)
\Y{\alpha\rightarrow \beta}{\gamma}T_2+(T\prec_{\alpha\triangleleft \beta}T_1)
\Y{\alpha\leftarrow \beta}{\gamma}T_2.
\end{align*}
The following conditions are equivalent:
\begin{enumerate}
\item With these products, $\K\calT_\Omega^+$ is the free  $\Omega$-dendriform freely generated by $\bdeux$.
\item With these products, $\K\calT_\Omega^+$ is $\Omega$-dendriform.
\item $(\Omega,\leftarrow,\rightarrow,\triangleleft,\triangleright)$ is an EDS.
\end{enumerate}
\end{prop}

\begin{proof}
We extend the products $\prec_\alpha$ and $\succ_\alpha$ to the space
$\K\calT_\Omega^+\otimes \K\calT_\Omega+\K\calT_\Omega\otimes \K\calT_\Omega^+$ by putting:
\begin{align*}
&\forall x\in \K\calT_\Omega^+,&
x\prec_\alpha \bun&=\bun \succ_\alpha x=x,&\bun \prec_\alpha x&=x\succ_\alpha \bun=0.
\end{align*}
By convention, we consider the added element $\emptyset$ as a unit for the four products of $\Omega$.
The definition of the products $\prec_\alpha$ and $\succ_\alpha$ can be rewritten in the following way:
for any $T\in \calT_\Omega^+$, for any $T_1,T_2\in \calT_\Omega$, for any $\alpha,\beta,\gamma \in \Omega$,
\begin{align*}
T\prec_\alpha \bun&=\bun \succ_\alpha T=T,\\
\bun \prec_\alpha T&=T\succ_\alpha \bun=0,\\
\left(T_1\Y{\alpha}{\beta}T_2\right)\prec_\gamma T
&=T_1\Y{\alpha}{\beta\leftarrow \gamma} (T_2\prec_{\beta\triangleleft \gamma} T)
+T_1\Y{\alpha}{\beta\rightarrow \gamma} (T_2\succ_{\beta\triangleright \gamma} T),\\ 
T\succ_\alpha \left(T_1 \Y{\beta}{\gamma}T_2\right)&=(T\succ_{\alpha\triangleright \beta}T_1)
\Y{\alpha\rightarrow \beta}{\gamma}T_2+(T\prec_{\alpha\triangleleft \beta}T_1)
\Y{\alpha\leftarrow \beta}{\gamma}T_2.
\end{align*}

Obviously, $1.\Longrightarrow 2$. Let us prove that $2.\Longrightarrow 3$. Let $\alpha,\beta,\gamma \in \Omega$.
For $x=y=\bdeux$ and $z=\bdtroisun(\gamma)$:
\begin{align*}
x\succ_\alpha(y\succ_\beta z)&=
\bdcinqun(\alpha \rightarrow (\beta \rightarrow \gamma), (\alpha\triangleright(\beta \rightarrow \gamma))\rightarrow
(\beta \triangleright \gamma), (\alpha \triangleright (\beta \rightarrow \gamma)) \triangleright (\beta \triangleright \gamma))\\
&+\bdcinqdeux(\alpha \rightarrow (\beta \rightarrow \gamma), (\alpha\triangleright(\beta \rightarrow \gamma))\leftarrow
(\beta \triangleright \gamma), (\alpha \triangleright (\beta \rightarrow \gamma)) \triangleleft (\beta \triangleright \gamma))\\
&+\bdcinqtrois(\alpha \leftarrow (\beta \rightarrow \gamma), \alpha \triangleleft(\beta \rightarrow \gamma),
\beta \triangleleft \gamma)\\
&+\bdcinqcinq(\alpha \rightarrow (\beta \leftarrow \gamma), \alpha \triangleright (\beta \leftarrow \gamma), 
\beta \triangleleft \gamma)\\
&+\bdcinqquatre(\alpha \leftarrow (\beta \leftarrow \gamma), \alpha \triangleleft(\beta \leftarrow \gamma),
\beta \triangleleft \gamma),\\
\\
(x\succ_{\alpha \triangleright \beta} y)\succ_{\alpha \rightarrow \beta} z&=
\bdcinqun((\alpha \rightarrow \beta)\rightarrow \gamma,(\alpha \rightarrow \beta)\triangleright \gamma,
\alpha \triangleright \beta)\\
&+\bdcinqcinq((\alpha \rightarrow \beta)\leftarrow \gamma,\alpha \triangleright \beta,
(\alpha \rightarrow \beta)\triangleleft \gamma),\\
\\
(x\prec_{\alpha \triangleleft \beta} y)\succ_{\alpha \leftarrow \beta} z&=
\bdcinqdeux((\alpha \leftarrow \beta) \rightarrow \gamma, (\alpha \leftarrow \beta)\triangleleft \gamma,
\alpha \triangleleft \beta)\\
&+\bdcinqquatre((\alpha \leftarrow \beta)\leftarrow \gamma,
(\alpha \triangleleft \beta)\leftarrow ((\alpha \leftarrow \beta)\triangleleft \gamma),
(\alpha \triangleleft \beta)\triangleleft ((\alpha \leftarrow \beta)\triangleleft \gamma))\\
&+\bdcinqtrois((\alpha \leftarrow \beta)\leftarrow \gamma,
(\alpha \triangleleft \beta)\rightarrow ((\alpha \leftarrow \beta)\triangleleft \gamma),
(\alpha \triangleleft \beta)\triangleright ((\alpha \leftarrow \beta)\triangleleft \gamma)).
\end{align*}
Identifying the decorations of the trees in these expressions,  we obtain relations (\ref{eq1})-(\ref{eq13}).\\

$3.\Longrightarrow 1$. Let us first prove that $\K\calT_\Omega^+$ is an $\Omega$-dendriform algebra.
Let us first prove relations (\ref{eq41})-(\ref{eq43}) for $x,y,z\in \calT_\Omega$ by induction
on the total number $N$ of leaves of $x$, $y$ and $z$.
Firstly, observe that (\ref{eq41}) is obviously satisfied if $x=\mid$;
(\ref{eq42}) is obviously satisfied if $y=\mid$;
(\ref{eq43}) is obviously satisfied if $z=\mid$; hence, there is nothing to prove if $N\leqslant 3$. Let us assume the result at all ranks
$<N$. Let us first prove (\ref{eq42}) for $x,y,z$. We can assume that $\displaystyle y=T_1 \Y{\gamma}{\delta} T_2$,
where $T_1,T_2\in \calT_\Omega$ and $\gamma,\delta \in \Omega \sqcup \{\emptyset\}$
($\gamma=\emptyset$ if $T_1=\bun$; $\delta=\emptyset$ if $\delta=\bun$).
Then:
\begin{align*}
(x\succ_\alpha y)\prec_\beta z&=(x\succ_{\alpha \triangleright \gamma} T_1) 
\Y{\alpha \rightarrow \gamma}{\delta\rightarrow \beta}
(T_2\succ_{\delta \triangleright \beta} T_1)
+(x\succ_{\alpha \triangleright \gamma} T_1) \Y{\alpha \rightarrow \gamma}{\delta\leftarrow \beta}
(T_2\prec_{\delta \triangleleft \beta} z)\\
&+(x\prec_{\alpha \triangleleft \gamma} T_1) \Y{\alpha \leftarrow \gamma}{\delta\rightarrow \beta}
(T_2\succ_{\delta \triangleright \beta} T_1)
+(x\prec_{\alpha \triangleleft \gamma} T_1) \Y{\alpha \leftarrow \gamma}{\delta\leftarrow \beta}
(T_2\prec_{\delta \triangleleft \beta} z)\\
&=x\succ_\alpha (y\prec_\beta z).
\end{align*}
Let us now prove (\ref{eq43}) for $x,y,z$. We can assume that $\displaystyle z=T_1\Y{\gamma}{\delta} T_2$. Then:
\begin{align*}
x\succ_\alpha(y\succ_\beta z)&=(x\succ_{\alpha\triangleright(\beta \rightarrow \gamma)} 
(y\succ_{\beta\triangleright \gamma} T_1))\Y{\alpha\rightarrow(\beta \rightarrow \gamma)}{\delta} T_2\\
&+(x\prec_{\alpha\triangleleft(\beta \rightarrow \gamma)} 
(y\succ_{\beta\triangleright \gamma} T_1))\Y{\alpha\leftarrow(\beta \rightarrow \gamma)}{\delta} T_2\\
&+(x\succ_{\alpha\triangleright(\beta \leftarrow \gamma)} 
(y\prec_{\beta\triangleleft \gamma} T_1))\Y{\alpha\rightarrow(\beta \leftarrow \gamma)}{\delta} T_2\\
&+(x\prec_{\alpha\triangleleft(\beta \leftarrow \gamma)} 
(y\prec_{\beta\triangleleft \gamma} T_1))\Y{\alpha\leftarrow(\beta \leftarrow \gamma)}{\delta} T_2,\\ 
\\
(x\succ_{\alpha\triangleright \beta} y)\succ_{\alpha \rightarrow \beta} z&=
((x\succ_{\alpha \triangleright \beta} y)\succ_{(\alpha\rightarrow \beta)\triangleright \gamma} T_1)
\Y{(\alpha \rightarrow \beta)\rightarrow \gamma}{\delta} T_2\\
&+((x\succ_{\alpha \triangleright \beta} y)\prec_{(\alpha\rightarrow \beta)\triangleleft \gamma} T_1)
\Y{(\alpha \rightarrow \beta)\leftarrow \gamma}{\delta} T_2,\\ 
\\
(x\prec_{\alpha\triangleleft \beta} y)\succ_{\alpha \leftarrow \beta} z&=
((x\prec_{\alpha \triangleleft \beta} y)\succ_{(\alpha\leftarrow \beta)\triangleright \gamma} T_1)
\Y{(\alpha \leftarrow \beta)\rightarrow \gamma}{\delta} T_2\\
&+((x\prec_{\alpha \triangleleft \beta} y)\prec_{(\alpha\leftarrow \beta)\triangleleft \gamma} T_1)
\Y{(\alpha \leftarrow \beta)\leftarrow \gamma}{\delta} T_2.
\end{align*}
Using the induction hypothesis and relations (\ref{eq1})-(\ref{eq13}),
putting $\alpha'=\alpha \triangleright(\beta \rightarrow \gamma)$ 
and $\beta'=\beta \triangleright \gamma$:
\begin{align*}
&((x\succ_{\alpha \triangleright \beta} y)\succ_{(\alpha\rightarrow \beta)\triangleright \gamma} T_1)
\Y{(\alpha \rightarrow \beta)\rightarrow \gamma}{\delta} T_2
+((x\prec_{\alpha \triangleleft \beta} y)\succ_{(\alpha\leftarrow \beta)\triangleright \gamma} T_1)
\Y{(\alpha \leftarrow \beta)\rightarrow \gamma}{\delta} T_2\\
&=(x\succ_{\alpha'}(y\succ_{\beta'} T_2))\Y{\alpha \rightarrow(\beta \rightarrow \gamma)}{\delta} T_2\\
&=(y\succ_{\beta\triangleright \gamma} T_1))\Y{\alpha\rightarrow(\beta \rightarrow \gamma)}{\delta} T_2.
\end{align*}
Similarly:
\begin{align*}
&((x\succ_{\alpha \triangleright \beta} y)\prec_{(\alpha\rightarrow \beta)\triangleleft \gamma} T_1)
\Y{(\alpha \rightarrow \beta)\leftarrow \gamma}{\delta} T_2=(x\succ_{\alpha\triangleright(\beta \leftarrow \gamma)} 
(y\prec_{\beta\triangleleft \gamma} T_1))\Y{\alpha\rightarrow(\beta \leftarrow \gamma)}{\delta} T_2,
\end{align*}
and
\begin{align*}
&(x\prec_{\alpha\triangleleft(\beta \rightarrow \gamma)} 
(y\succ_{\beta\triangleright \gamma} T_1))\Y{\alpha\leftarrow(\beta \rightarrow \gamma)}{\delta} T_2
+(x\prec_{\alpha\triangleleft(\beta \leftarrow \gamma)} 
(y\prec_{\beta\triangleleft \gamma} T_1))\Y{\alpha\leftarrow(\beta \leftarrow \gamma)}{\delta} T_2\\
&=((x\prec_{\alpha \triangleleft \beta} y)\prec_{(\alpha\leftarrow \beta)\triangleleft \gamma} T_1)
\Y{(\alpha \leftarrow \beta)\leftarrow \gamma}{\delta} T_2.
\end{align*}
So (\ref{eq43}) is satisfied for $x,y,z$. Relation (\ref{eq41}) is proved similarly. 
We obtain that $\K\calT_\Omega^+$ is $\Omega$-dendriform. \\

Let us now prove its freeness. Let $A$ be an $\Omega$-dendriform algebra and let $a\in A$.
Let us prove the existence and uniqueness of an $\Omega$-dendriform algebra morphism $\Phi$ from $\K\calT_\Omega^+$
to $A$ such that $\Phi(\bdeux)=a$.

We first extend the products of $A$ to $\K \otimes A+A\otimes \K+A\otimes A$
by putting, for any $b\in A$:
\begin{align*}
b\succ_\alpha 1=1\prec_\alpha b&=0,& 1\succ_\alpha b=b\prec_\alpha 1&=b.
\end{align*}
We then define $\Phi(T)$ for any tree $T\in \calT_\Omega$ by induction on its number of leaves:
\begin{align*}
\Phi(\bun)&=1,\\
\Phi\left(T_1\Y{\alpha}{\beta}T_2\right)&=\Phi(T_1)\succ_\alpha a\prec_\beta \Phi(T_2).
\end{align*}
Let us prove that $\Phi$ is an $\Omega$-dendriform algebra morphism. Let $x,y\in \calT_\Omega$
and let us prove that 
\begin{align*}
\Phi(x\succ_\alpha y)&=\Phi(x)\succ_\alpha \Phi(y),&\Phi(x\prec_\alpha y)&=\Phi(x)\prec_\alpha \Phi(y)
\end{align*}
by induction on the total number $N$ of leaves of $x$ and $y$. 
If $y=\bun$, then:
\begin{align*}
\Phi(x\succ_\alpha \bun)&=0=\Phi(x)\succ_\alpha 1,&
\Phi(x\prec_\alpha \bun)&=\Phi(x)=\Phi(x)\succ_\alpha 1.
\end{align*}
The proof is similar if $x=\bun$. Let us now assume that $x,y\neq \bun$. 
Let us put $\displaystyle y=T_1\Y{\beta}{\gamma} T_2$. By the induction hypothesis applied to $T_1$:
\begin{align*}
\Phi(x\succ_\alpha y)&=\Phi\left((x\succ_{\alpha \triangleright \beta} T_1)\Y{\alpha\rightarrow \beta}{\gamma}T_2\right)
+\Phi\left((x\prec_{\alpha \triangleleft \beta} T_1)\Y{\alpha\leftarrow \beta}{\gamma}T_2\right)\\
&=(\Phi(x)\succ_{\alpha \triangleright \beta} \Phi(T_1))\succ_{\alpha\rightarrow \beta} a\prec_{\gamma}\Phi(T_2)\\
&+(\Phi(x)\prec_{\alpha \triangleleft \beta} \Phi(T_1))\succ_{\alpha\leftarrow \beta}a\prec_{\gamma}\Phi(T_2)\\
&=(\Phi(x)\succ_\alpha (\Phi(T_1)\succ_\beta a))\prec_\gamma \Phi(T_2)\\
&=\Phi(x)\succ_\alpha((\Phi(T_1)\succ_\beta a\prec_\gamma\Phi(T_2))\\
&=\Phi(x)\succ_\alpha \Phi(y).
\end{align*}
Similarly, $\Phi(x\prec_\alpha y)=\Phi(x)\prec_\alpha \Phi(y)$. \\

Let us now prove the unicity of $\Phi$. Let $\Psi$ be another morphism from $\K\calT_\Omega^+$ to $A$
such that $\Psi(\bdeux)=a$. For any tree $T\neq \bun$, putting $\displaystyle T=T_1\Y{\alpha}{\beta}T_2$:
\begin{align*}
\Psi(T)&=\Psi(T_1\succ_\alpha \bdeux \prec_\beta T_2)=\Psi(T_1)\succ_\alpha a\prec_\beta \Psi(T_2),
\end{align*}
so $\Psi=\Phi$. \end{proof}

\begin{example}
Let $\alpha,\beta \in \Omega$.
\begin{align*}
\bdeux\prec_\beta \bdeux&=\bdtroisdeux(\beta),\\
\bdeux\succ_\alpha \bdeux&=\bdtroisun(\alpha),\\
\bdeux\succ_\alpha \bdtroisun(\beta)&=\bdquatreun(\alpha\rightarrow \beta, \alpha \triangleright \beta)
+\bdquatredeux(\alpha\leftarrow \beta, \alpha \triangleleft \beta),\\
\bdeux\succ_\alpha \bdtroisdeux(\beta)&=\bdquatrecinq(\alpha,\beta),\\
\bdtroisun(\alpha)\succ_\beta \bdeux&=\bdquatreun(\beta,\alpha),\\
\bdtroisdeux(\alpha)\succ_\beta \bdeux&=\bdquatredeux(\beta,\alpha).
\end{align*}
\end{example}

\begin{remark}
\begin{enumerate}
\item An easy induction proves that the $\Omega$-dendriform algebra $\K\calT_\Omega^+$ is graded:
\begin{align*}
&\forall \alpha \in \Omega,\:\forall k,l\geqslant 1,&
\K\calT_\Omega(k)\prec_\alpha \K\calT_\Omega(l)+\K\calT_\Omega(k)\succ_\alpha \K\calT_\Omega(l)
&\subseteq \K\calT_\Omega(k+l).
\end{align*}
\item Similar results can be proved for $\Omega$-typed $D$-decorated plane binary trees, that is to say
$\Omega$-typed plane binary trees given a map from the set of internal vertices to $D$.
We obtain in this way the free $\Omega$-dendriform generated by $D$.
\end{enumerate}
\end{remark}

\subsection{Structure on typed words}

\begin{defi}
Let $\Omega$ be a set and let $V$ be a vector space. The space of $\Omega$-typed words in $V$ is
\[\Sh^+_\Omega(V)=\bigoplus_{n\geqslant 1} (\K\Omega)^{\otimes (n-1)}\otimes V^{\otimes n}.\]
Tensors of $\Sh^+_\Omega(V)$ will be written in the form
\[\alpha_2\ldots \alpha_n\otimes v_1\ldots v_n,\]
where $n\geqslant 1$, $\alpha_2,\ldots,\alpha_n \in \Omega$ and $v_1,\ldots,v_n \in V$.
Such a tensor will be called an $\Omega$-typed word in $V$; its length is the integer $n$.
We also put $\Sh_\Omega(V)=\K\oplus \Sh_\Omega^+(V)$.
\end{defi}

\begin{prop}\label{prop17}
Let $\Omega$ be a set with four operations $\leftarrow,\rightarrow,\triangleleft,\triangleright$. 
For any vector space $V$, we give $\Sh_\Omega(V)$ products $\prec_\alpha$, $\succ_\alpha$, where $\alpha \in \Omega$,
inductively defined in the following way:
\begin{align*}
1\prec_\alpha \alpha_2\ldots \alpha_m\otimes v_1\ldots v_m&=\alpha_2\ldots \alpha_m\otimes v_1\ldots v_m\succ_\alpha 1=0,\\
\alpha_2\ldots \alpha_m\otimes v_1\ldots v_m\prec_\alpha 1&=
1\succ_\alpha \alpha_2\ldots \alpha_n\otimes v_1\ldots v_m=
\alpha_2\ldots \alpha_m\otimes v_1\ldots v_m,
\end{align*}
and
\begin{align*}
&\alpha_2\ldots \alpha_m\otimes v_1\ldots v_m\prec_\alpha
\beta_2\ldots \beta_n\otimes w_1\ldots w_n\\
&=((\alpha_2\rightarrow \alpha) \otimes v_1)\cdot
(\alpha_3\ldots \alpha_m \otimes v_2\ldots v_m \prec_{\alpha_2\triangleleft \alpha}
\beta_2\ldots \beta_n\otimes w_1\ldots w_n)\\
&+((\alpha_2\leftarrow \alpha) \otimes v_1)\cdot
(\alpha_3\ldots \alpha_m \otimes v_2\ldots v_m \succ_{\alpha_2\triangleright \alpha}
\beta_2\ldots \beta_n\otimes w_1\ldots w_n),\\ \\
&\alpha_2\ldots \alpha_m\otimes v_1\ldots v_m\succ_\alpha
\beta_2\ldots \beta_n\otimes w_1\ldots w_n\\
&=((\alpha\rightarrow \beta_2) \otimes w_1)\cdot
(\alpha_2\ldots \alpha_m \otimes v_1\ldots v_m \prec_{\alpha\triangleleft \beta_2}
\beta_3\ldots \beta_n\otimes w_2\ldots w_n)\\
&+((\alpha\leftarrow \beta_2) \otimes w_1)\cdot
(\alpha_2\ldots \alpha_m \otimes v_1\ldots v_m \succ_{\alpha\triangleright \beta_2}
\beta_3\ldots \beta_n\otimes w_2\ldots w_n),
\end{align*}
where $\cdot$ is the concatenation product:
\begin{align*}
\cdot&\left\{\begin{array}{rcl}
(\K\Omega \otimes V)\otimes \Sh_\Omega(V)&\longrightarrow&\Sh_\Omega^+(V)\\
(\alpha \otimes v)\otimes (\alpha_2\ldots \alpha_n\otimes v_1\ldots v_n)&\longrightarrow&
\alpha \alpha_2\ldots \alpha_n\otimes v v_1\ldots v_n.
\end{array}\right.
\end{align*}
 The following conditions are equivalent:
\begin{enumerate}
\item With these products, $\Sh_\Omega^+(V)$ is an $\Omega$-dendriform algebra for any vector space $V$.
\item $\Omega$ is an EDS.
\end{enumerate}
If this holds and if $\Omega$ is commutative, then $\Sh_\Omega^+(V)$ is the free commutative $\Omega$-dendriform
algebra generated by $V$.
\end{prop}

\begin{proof} Note that these products $\prec_\alpha$, $\succ_\alpha$ are defined on 
\[(\Sh_\Omega(V)\otimes\Sh_\Omega(V))^+
=\Sh_\Omega(V)\otimes \Sh_\Omega^+(V)+\Sh_\Omega^+(V)\otimes \Sh_\Omega(V).\]
$1.\Longrightarrow 2$. Let $V$ be a vector space of dimension 4 and $(v_1,v_2,v_3,v_4)$ be a basis of $V$.
Let $\alpha,\beta,\gamma\in \Omega$.
\begin{align*}
&v_1\succ_\alpha (v_2\succ_\beta \gamma\otimes v_3v_4)\\
&=\alpha \rightarrow (\beta \rightarrow \gamma)\cdot
(\alpha \triangleright(\beta \rightarrow \gamma))\rightarrow (\beta \triangleright \gamma)\cdot
(\alpha \triangleright(\beta \rightarrow \gamma))\triangleright (\beta \triangleright \gamma)
\otimes v_3v_4v_2v_1\\
&+\alpha \rightarrow (\beta \rightarrow \gamma)\cdot
(\alpha \triangleright(\beta \rightarrow \gamma))\leftarrow (\beta \triangleright \gamma)\cdot
(\alpha \triangleright(\beta \rightarrow \gamma))\triangleleft (\beta \triangleright \gamma)
\otimes v_3v_4v_1v_2\\
&+\alpha\leftarrow(\beta \rightarrow \gamma) \cdot \alpha \triangleleft(\beta \rightarrow \gamma)\cdot
\beta \triangleleft \gamma \otimes v_3v_1v_4v_2\\
&+\alpha\rightarrow(\beta\leftarrow \gamma)\cdot(\alpha \triangleright(\beta\leftarrow \gamma))
\rightarrow(\beta \triangleleft \gamma)\cdot (\alpha \triangleright(\beta\leftarrow \gamma))
\triangleright(\beta \triangleleft \gamma)\otimes v_3v_2v_4v_1\\
&+\alpha\rightarrow(\beta\leftarrow \gamma)\cdot(\alpha \triangleright(\beta\leftarrow \gamma))
\leftarrow(\beta \triangleleft \gamma)\cdot (\alpha \triangleright(\beta\leftarrow \gamma))
\triangleleft(\beta \triangleleft \gamma)\otimes v_3v_2v_1v_4\\
&+\alpha\leftarrow (\beta \leftarrow \gamma) \cdot \alpha\triangleleft(\beta\leftarrow \gamma)\cdot
\beta \triangleleft \gamma \otimes v_3v_1v_2v_4,\\
\\
&(v_1\succ_{\alpha \triangleright \beta} v_2)\succ_{\alpha \rightarrow \beta} \gamma \otimes v_3v_4\\
&=(\alpha \rightarrow \beta)\rightarrow \gamma\cdot (\alpha \rightarrow \beta) \triangleright \gamma
\cdot \alpha \triangleright \beta \otimes v_3 v_4v_2v_1\\
&+(\alpha \rightarrow \beta)\leftarrow \gamma\cdot (\alpha \triangleright \beta)\leftarrow
((\alpha \rightarrow \beta)\triangleleft \gamma)\cdot (\alpha \triangleright \beta)\triangleleft
((\alpha \rightarrow \beta)\triangleleft \gamma)\otimes v_3v_2v_1v_4\\
&+(\alpha \rightarrow \beta)\leftarrow \gamma\cdot (\alpha \triangleright \beta)\rightarrow
((\alpha \rightarrow \beta)\triangleleft \gamma)\cdot (\alpha \triangleright \beta)\triangleright
((\alpha \rightarrow \beta)\triangleleft \gamma)\otimes v_3v_2v_4v_1,\\
\\
&(v_1\prec_{\alpha \triangleleft \beta} v_2)\succ_{\alpha \leftarrow \beta} \gamma \otimes v_3v_4\\
&=(\alpha \leftarrow \beta)\rightarrow \gamma\cdot (\alpha \leftarrow \beta) \triangleright \gamma
\cdot \alpha \triangleright \beta \otimes v_3 v_4v_1v_2\\
&+(\alpha \leftarrow \beta)\rightarrow \gamma\cdot (\alpha \triangleleft \beta)\leftarrow
((\alpha \leftarrow \beta)\triangleleft \gamma)\cdot (\alpha \triangleleft \beta)\triangleleft
((\alpha \leftarrow \beta)\triangleleft \gamma)\otimes v_3v_1v_2v_4\\
&+(\alpha \leftarrow \beta)\rightarrow \gamma\cdot (\alpha \triangleleft \beta)\rightarrow
((\alpha \leftarrow \beta)\triangleleft \gamma)\cdot (\alpha \triangleleft \beta)\triangleright
((\alpha \leftarrow \beta)\triangleleft \gamma)\otimes v_3v_1v_4v_2.
\end{align*}
As the family $(v_{\sigma(1)}v_{\sigma(2)}v_{\sigma(3)}v_{\sigma(4)})_{\sigma \in \mathfrak{S}_4}$
is linearly independent, identifying in (\ref{eq43}), we obtain (\ref{eq1})-(\ref{eq13}). \\

$2.\Longrightarrow 1$. Let us prove (\ref{eq41})-(\ref{eq43}) for $x,y,z$ typed words by induction on the total
length $N$ of $x$, $y$ and $z$. If $x=1$, then (\ref{eq41}) is trivially satisfied;
If $y=1$, then (\ref{eq42}) is trivially satisfied; if $z=1$, then (\ref{eq43}) is trivially satisfied.
This proves the result if $N\leqslant 2$. We now suppose that $x,y,z\neq 1$ and let us assume the result at all ranks $<N$.
Let us put $z=\gamma \otimes v\cdot z'$. Using the induction hypothesis:
\begin{align*}
x\succ_\alpha (y\succ_\beta z)&=
\alpha \rightarrow(\beta \rightarrow \gamma) \otimes v\cdot 
(x\succ_{\alpha\triangleright (\beta \rightarrow \gamma)} (y\succ_{\beta \triangleright \gamma} z'))\\
&+\alpha \leftarrow(\beta \rightarrow \gamma) \otimes v\cdot 
(x\prec_{\alpha\triangleleft (\beta \rightarrow \gamma)} (y\succ_{\beta \triangleright \gamma} z'))\\
&+\alpha \rightarrow(\beta \leftarrow \gamma) \otimes v\cdot 
(x\succ_{\alpha\triangleright (\beta \leftarrow \gamma)} (y\prec_{\beta \triangleleft \gamma} z'))\\
&+\alpha \leftarrow(\beta \leftarrow \gamma) \otimes v\cdot 
(x\prec_{\alpha\triangleleft (\beta \leftarrow \gamma)} (y\prec_{\beta \triangleleft \gamma} z'))\\
&=\alpha \rightarrow(\beta \rightarrow \gamma) \otimes v\cdot 
(x\succ_{(\alpha \triangleright(\beta \rightarrow \gamma))\triangleright (\beta \triangleleft \gamma)} y)
\succ_{(\alpha \triangleright(\beta \rightarrow \gamma))\rightarrow (\beta \triangleleft \gamma)}z'\\
&+\alpha \rightarrow(\beta \rightarrow \gamma) \otimes v\cdot 
(x\prec_{(\alpha \triangleright(\beta \rightarrow \gamma))\triangleleft (\beta \triangleleft \gamma)} y)
\succ_{(\alpha \triangleright(\beta \rightarrow \gamma))\leftarrow (\beta \triangleleft \gamma)}z'\\
&+\alpha \rightarrow(\beta \leftarrow \gamma) \otimes v\cdot 
(x\succ_{\alpha \triangleright (\beta \leftarrow \gamma)} y)\prec_{\beta \leftarrow \gamma} z'\\
&+\alpha \leftarrow(\beta \leftarrow \gamma) \otimes v\cdot 
x \prec_{\alpha \triangleleft (\beta \leftarrow \gamma)} (y\prec_{\beta \triangleleft \gamma} z')\\
&+\alpha \leftarrow(\beta \rightarrow \gamma) \otimes v\cdot 
x \prec_{\alpha \triangleleft (\beta \rightarrow \gamma)} (y\succ_{\beta \triangleright \gamma} z').
\end{align*} 
Moreover:
\begin{align*}
(x\succ_{\alpha \triangleright \beta} y)\succ_{\alpha \rightarrow \beta} z
&=(\alpha \rightarrow \beta) \rightarrow \gamma \otimes v\cdot
(x\succ_{\alpha \triangleright \beta} y)\succ_{(\alpha \rightarrow \beta)\triangleright \gamma} z'\\
&+(\alpha \rightarrow \beta) \leftarrow \gamma \otimes v\cdot
(x\succ_{\alpha \triangleright \beta} y)\prec_{(\alpha \rightarrow \beta)\triangleleft \gamma} z';\\
\\
(x\prec_{\alpha \triangleleft \beta} y)\succ_{\alpha \leftarrow \beta} z
&=(\alpha \leftarrow \beta)\rightarrow\gamma\otimes v\cdot
(x\prec_{\alpha \triangleleft \beta} y)\succ_{(\alpha \leftarrow \beta)\triangleright \gamma} z'\\
&+(\alpha \leftarrow \beta)\leftarrow\gamma\otimes v\cdot
(x\prec_{\alpha \triangleleft \beta} y)\prec_{(\alpha \leftarrow \beta)\triangleleft \gamma} z'\\
&=(\alpha \leftarrow \beta)\rightarrow\gamma\otimes v\cdot
(x\prec_{\alpha \triangleleft \beta} y)\succ_{(\alpha \leftarrow \beta)\triangleright \gamma} z'\\
&+(\alpha \leftarrow \beta)\leftarrow \gamma\otimes v\cdot
x\prec_{(\alpha \triangleleft \beta)\leftarrow ((\alpha \leftarrow \beta)\triangleleft \gamma)}
(y\prec_{(\alpha \triangleleft \beta)\triangleleft ((\alpha \leftarrow \beta)\triangleleft \gamma)} z')\\
&+(\alpha \leftarrow \beta)\leftarrow \gamma\otimes v\cdot
x\prec_{(\alpha \triangleleft \beta)\rightarrow ((\alpha \leftarrow \beta)\triangleleft \gamma)}
(y\succ_{(\alpha \triangleleft \beta)\triangleright ((\alpha \leftarrow \beta)\triangleleft \gamma)} z').
\end{align*}
With (\ref{eq1})-(\ref{eq13}), we conclude that (\ref{eq43}) is satisfied for $x,y,z$. 
Relations (\ref{eq41}) and (\ref{eq42}) are proved in the same way. \\

Observe that:
\begin{align*}
1\prec^{op}_\alpha \alpha_2\ldots \alpha_m\otimes v_1\ldots v_m
&=\alpha_2\ldots \alpha_m\otimes v_1\ldots v_m\succ^{op}_\alpha 1=0,\\
\alpha_2\ldots \alpha_m\otimes v_1\ldots v_m\prec^{op}_\alpha 1&=
1\succ^{op}_\alpha \alpha_2\ldots \alpha_n\otimes v_1\ldots v_m=
\alpha_2\ldots \alpha_m\otimes v_1\ldots v_m,
\end{align*}
and
\begin{align*}
&\alpha_2\ldots \alpha_m\otimes v_1\ldots v_m\prec^{op}_\alpha
\beta_2\ldots \beta_n\otimes w_1\ldots w_n\\
&=((\alpha\leftarrow \alpha_2) \otimes v_1)\cdot
(\alpha_3\ldots \alpha_m \otimes v_2\ldots v_m \prec^{op}_{\alpha_2\triangleleft \alpha}
\beta_2\ldots \beta_n\otimes w_1\ldots w_n)\\
&+((\alpha\rightarrow \alpha_2) \otimes v_1)\cdot
(\alpha_3\ldots \alpha_m \otimes v_2\ldots v_m \succ^{op}_{\alpha_2\triangleright \alpha}
\beta_2\ldots \beta_n\otimes w_1\ldots w_n)\\
&=((\alpha_2\rightarrow^{op} \alpha) \otimes v_1)\cdot
(\alpha_3\ldots \alpha_m \otimes v_2\ldots v_m \prec^{op}_{\alpha_2\triangleleft \alpha}
\beta_2\ldots \beta_n\otimes w_1\ldots w_n)\\
&+((\alpha_2\leftarrow^{op} \alpha) \otimes v_1)\cdot
(\alpha_3\ldots \alpha_m \otimes v_2\ldots v_m \succ^{op}_{\alpha_2\triangleright \alpha}
\beta_2\ldots \beta_n\otimes w_1\ldots w_n),\\ \\
&\alpha_2\ldots \alpha_m\otimes v_1\ldots v_m\succ^{op}_\alpha
\beta_2\ldots \beta_n\otimes w_1\ldots w_n\\
&=((\beta_2\leftarrow \alpha) \otimes w_1)\cdot
(\alpha_2\ldots \alpha_m \otimes v_1\ldots v_m \prec^{op}_{\alpha\triangleleft \beta_2}
\beta_3\ldots \beta_n\otimes w_2\ldots w_n)\\
&+((\beta_2\rightarrow \alpha)  \otimes w_1)\cdot
(\alpha_2\ldots \alpha_m \otimes v_1\ldots v_m \succ^{op}_{\alpha\triangleright \beta_2}
\beta_3\ldots \beta_n\otimes w_2\ldots w_n)\\
&=((\alpha\rightarrow^{op} \beta_2) \otimes w_1)\cdot
(\alpha_2\ldots \alpha_m \otimes v_1\ldots v_m \prec^{op}_{\alpha\triangleleft \beta_2}
\beta_3\ldots \beta_n\otimes w_2\ldots w_n)\\
&+((\alpha\leftarrow^{op} \beta_2) \otimes w_1)\cdot
(\alpha_2\ldots \alpha_m \otimes v_1\ldots v_m \succ^{op}_{\alpha\triangleright \beta_2}
\beta_3\ldots \beta_n\otimes w_2\ldots w_n),
\end{align*}
so
\[\Sh^+_\Omega(V)^{op}=\Sh^+_{\Omega^{op}}(V).\]
In particular, if $\Omega$ is commutative, the $\Omega$-dendriform algebra $\Sh^+_\Omega(V)$
is commutative.\\

Let us assume that $\Omega$ is commutative. Let $A$ be a commutative $\Omega$-dendriform algebra
and let $\phi:V\longrightarrow A$ be any linear map. Let us prove that there exists a unique map
$\Phi:\Sh_\Omega^+(V)\longrightarrow A$ of $\Omega$-dendriform algebras such that $\Phi_{\mid V}=\phi$. \\

\textit{Existence.} We inductively define $\Phi$ by:
\begin{align*}
\Phi(v)&=\phi(v),\\
\Phi(\alpha_2\ldots \alpha_n\otimes v_1\ldots v_n)&=\phi(v_1) \prec_{\alpha_2} \Phi(\alpha_3\ldots \alpha_n\otimes
v_2\ldots v_n)\mbox{ if }n\geqslant 2.
\end{align*}
Let us prove that $\Phi(x\prec_\alpha y)=\Phi(x)\prec_\alpha \Phi(y)$ for any typed words $x$ and $y$
by induction on the total length $N$ of $x$ and $y$. If the length of $x$ is $1$:
\begin{align*}
\Phi(x \prec_\alpha y)&=\Phi(\alpha \otimes x\cdot y)=\phi(x)\prec_\alpha \Phi(y)=\Phi(x)\prec_\alpha \Phi(y).
\end{align*}
This proves the result if $N=2$. Let us assume the result at all ranks $<N$. We can restrict ourselves to the case
where the length of $x$ is not $1$. We put $x=(\beta\otimes v)\cdot x'$, with $v\in V$
and $x'$ is a typed word. Then:
\begin{align*}
\Phi(x\prec_\alpha y)&=\Phi(\beta \leftarrow \alpha \otimes v\cdot x'\prec_{\alpha_\triangleleft \beta}y)
+\Phi(\beta \rightarrow \alpha \otimes v\cdot x'\succ_{\alpha_\triangleright \beta}y)\\
&=\phi(v)\prec_{\beta \leftarrow \alpha}\Phi( x'\prec_{\alpha_\triangleleft \beta}y)
+\phi(v) \prec_{\beta \rightarrow \alpha}\Phi(x'\succ_{\alpha_\triangleright \beta}y)\\
&=\phi(v)\prec_{\beta \leftarrow \alpha}\Phi( x'\prec_{\alpha_\triangleleft \beta}y)
+\phi(v) \prec_{\beta \rightarrow \alpha}\Phi(y\prec_{\alpha_\triangleright \beta}x')\\
&=\phi(v)\prec_{\beta \leftarrow \alpha}(\Phi( x')\prec_{\alpha_\triangleleft \beta}\Phi(y))
+\phi(v) \prec_{\beta \rightarrow \alpha}(\Phi(y)\prec_{\alpha_\triangleright \beta}\Phi(x'))\\
&=\phi(v)\prec_{\beta \leftarrow \alpha}(\Phi( x')\prec_{\alpha_\triangleleft \beta}\Phi(y))
+\phi(v) \prec_{\beta \rightarrow \alpha}(\Phi(x')\succ_{\alpha_\triangleright \beta}\Phi(y))\\
&=(\phi(v)\prec_\beta \Phi(x'))\prec_\alpha \Phi(y)\\
&=\Phi(x) \prec_\alpha \Phi(y).
\end{align*}
So $\Phi$ is compatible with $\prec_\alpha$. As $A$ and $\Sh_\Omega^+(V)$ are commutative, for any $x,y\in \Sh_\Omega^+(V)$,
\[\Phi(x\succ_\alpha y)=\Phi(y\prec_\alpha x)=\Phi(y)\prec_\alpha \Phi(x)=\Phi(x)\succ_\alpha \Phi(y).\]
So $\Phi$ is a morphism of $\Omega$-dendriform algebras.\\

\textit{Unicity}. Let $\Psi$ be such a morphism. Then for any typed word $x=(\alpha \otimes v)\cdot x'$
of length $\geqslant 2$:
\begin{align*}
\Psi(x)&=\Psi(v\prec_\alpha x')=\phi(v)\prec_\alpha \Psi(x').
\end{align*}
Hence, $\Psi=\Phi$.  \end{proof}

\subsection{From $\Omega$-dendrifrom algebras to  dendriform algebras}

\begin{prop}\label{prop18}
Let $\Omega$ be an EDS and let $A$ be a vector space equipped with bilinear products
$\prec_\alpha$ and $\succ_\alpha$. We equip $\K\Omega \otimes A$ with two bilinear products
$\prec,\succ$ defined in the following way: 
\begin{align*}
&\forall\alpha,\beta \in \Omega,\: \forall x,y\in A,&
\alpha \otimes x\prec \beta \otimes y&=\alpha \leftarrow \beta \otimes x\prec_{\alpha \triangleleft \beta} y,\\
&&\alpha \otimes x\succ\beta \otimes y&=\alpha \rightarrow \beta \otimes x\succ_{\alpha \triangleright \beta} y.
\end{align*}
\begin{enumerate}
\item If $(A,(\prec_\alpha)_{\alpha \in \Omega},(\succ_\alpha)_{\alpha\in \Omega})$ is $\Omega$-dendriform,
then $(\K\Omega \otimes A,\prec,\succ)$ is dendriform.
\item If $\varphi_\leftarrow$ and $\varphi_\rightarrow$ are surjective, then the converse implication is true.
\end{enumerate}\end{prop}

\begin{proof}
Let $\alpha,\beta,\gamma \in \Omega$ and $x,y,z\in A$.
\begin{align*}
(\alpha \otimes x\prec \beta \otimes y)\prec \gamma \otimes z
&=(\alpha \rightarrow \beta)\rightarrow \gamma \otimes
(x\prec_{\alpha \triangleleft \beta} y)\prec_{(\alpha \leftarrow \beta)\triangleleft \gamma} z,\\
\alpha \otimes x\prec (\beta \otimes y\prec \gamma \otimes z+\beta \otimes y\succ \gamma \otimes z)
&=\alpha \leftarrow (\beta \leftarrow \gamma)\otimes
x\prec_{\alpha \triangleleft(\beta \leftarrow \gamma)}(y\prec_{\beta \triangleleft \gamma} z)\\
&+\alpha \leftarrow (\beta \rightarrow \gamma)\otimes
x\prec_{\alpha \triangleleft(\beta \rightarrow \gamma)}(y\succ_{\beta \triangleright \gamma} z).
\end{align*}
As $(\Omega,\leftarrow,\rightarrow)$ is diassociative,
\[(\alpha \rightarrow \beta)\rightarrow \gamma=\alpha \leftarrow (\beta \leftarrow \gamma)
=\alpha \leftarrow (\beta \rightarrow \gamma).\]

1. Let us assume that $A$ is $\Omega$-dendriform. Then, as $\Omega$ is an EDS:
\begin{align*}
(x\prec_{\alpha \triangleleft \beta} y)\prec_{(\alpha \leftarrow \beta)\triangleleft \gamma} z
&=x\prec_{(\alpha \triangleleft \beta)\leftarrow ((\alpha \leftarrow \beta)\triangleleft \gamma)}
(y\prec_{(\alpha \triangleleft \beta)\triangleleft((\alpha \leftarrow \beta)\triangleleft \gamma)}z)\\
&+x\prec_{(\alpha \triangleleft \beta)\rightarrow ((\alpha \leftarrow \beta)\triangleleft \gamma)}
(y\succ_{(\alpha \triangleleft \beta)\triangleright(((\alpha \leftarrow \beta)\triangleleft \gamma)}z)\\
&=x\prec_{\alpha \triangleleft(\beta \leftarrow \gamma)}(y\prec_{\beta \triangleleft \gamma} z)
+x\prec_{\alpha \triangleleft(\beta \rightarrow \gamma)}(y\succ_{\beta \triangleright \gamma} z).
\end{align*}
So the first dendriform relation is satisfied. The second and third ones are proved in the same way.

2. Let us assume that $\K\Omega\otimes A$ is dendriform and that $\varphi_\leftarrow$ and $\varphi_\rightarrow$ are surjective.
For any $x,y,z\in A$, for any $\alpha,\beta,\gamma \in \Omega$:
\begin{align*}
(x\prec_{\alpha \triangleleft \beta} y)\prec_{(\alpha \leftarrow \beta)\triangleleft \gamma} z
&=x\prec_{\alpha \triangleleft(\beta \leftarrow \gamma)}(y\prec_{\beta \triangleleft \gamma} z)
+x\prec_{\alpha \triangleleft(\beta \rightarrow \gamma)}(y\succ_{\beta \triangleright \gamma} z).
\end{align*}
By hypothesis, the following map is surjective:
\[\varphi'_\leftarrow:\left\{\begin{array}{rcl}
\Omega^2&\longrightarrow&\Omega^2\\
(\alpha,\beta)&\longrightarrow&(\alpha\triangleleft,\alpha\leftarrow \beta).
\end{array}\right.\]
By composition, the following map is surjective:
\[
(Id \otimes \varphi'_\leftarrow)\circ (\varphi'_\leftarrow\otimes Id):\left\{\begin{array}{rcl}
\Omega^3&\longrightarrow&\Omega^3\\
(\alpha,\beta,\gamma)&\longrightarrow&(\alpha \triangleleft\beta,(\alpha\leftarrow \beta)\triangleleft \gamma,
\alpha \leftarrow \beta \leftarrow \gamma).
\end{array}\right.\]
Let $(\alpha',\beta',\gamma')\in \Omega^3$ and let $(\alpha,\beta,\gamma)\in \Omega^3$
such that:
\begin{align*}
\alpha'&=\alpha \triangleleft\beta,&\beta'&=(\alpha\rightarrow \beta)\triangleleft \gamma,&
\gamma'&=\alpha \leftarrow \beta \leftarrow \gamma.
\end{align*}
Then, by (\ref{eq6})-(\ref{eq9}):
\begin{align*}
(x\prec_{\alpha'} y)\prec_{\beta'} z&=
x\prec_{\alpha \triangleleft(\beta \leftarrow \gamma)}(y\prec_{\beta \triangleleft \gamma} z)
+x\prec_{\alpha \triangleleft(\beta \rightarrow \gamma)}(y\succ_{\beta \triangleright \gamma} z)\\
&=x\prec_{\alpha'\leftarrow \beta'}(y\prec_{\alpha'\triangleleft \beta'} z)
+x\prec_{\alpha'\rightarrow \beta'}(y\succ_{\alpha'\triangleright \beta'} z)
\end{align*}
So the first $\Omega$-dendriform relation is satisfied. The two other ones are similarly proved.
\end{proof}

Let us now study the dendriform algebras $\K\Omega \otimes \K\calT_\Omega^+$ and $\K\Omega \otimes \Sh_\Omega^+(V)$. 

\begin{prop}\label{prop19}
Let $\Omega$ be an EDS. 
\begin{enumerate}
\item The following assertions are equivalent:
\begin{enumerate}
\item The dendriform algebra $\K\Omega \otimes \K\calT_\Omega^+$ is generated by the elements
$\alpha \otimes \bdeux$, $\alpha \in \Omega$.
\item $\varphi_\leftarrow$ and $\varphi_\rightarrow$ are surjective.
\end{enumerate}
\item The following assertions are equivalent:
\begin{enumerate}
\item The dendriform subalgebra of $\K\Omega \otimes \K\calT_\Omega^+$ generated by the elements
$\alpha \otimes \bdeux$, $\alpha \in \Omega$, is free.
\item $\varphi_\leftarrow$ and $\varphi_\rightarrow$ are injective.
\end{enumerate}
\end{enumerate}
\end{prop}

\begin{proof} Firstly, observe that the dendriform algebra $\K\Omega \otimes \calT_\Omega^+$ is graded, 
with for any $n\geqslant 1$,
\[(\K\Omega \otimes \calT_\Omega^+)(n)=\K\Omega \otimes \calT_\Omega(n).\]

1. $(a)\Longrightarrow (b)$. Let $\alpha,\beta \in \Omega$. As $\K\Omega \otimes \calT_\Omega^+$
is graded, there exists families of scalars $(\lambda_{a,b})_{a,b\in\Omega}$ and $(\mu_{a,b})_{a,b\in\Omega}$ such that:
\begin{align*}
\alpha \otimes\bdtroisun(\beta)&=\sum_{a,b\in \Omega}\lambda_{a,b} a\otimes \bdeux \prec b\otimes\bdeux+
\sum_{a,b\in \Omega} \mu_{a,b}a\otimes\bdeux \succ b\otimes\bdeux\\
&=\sum_{a,b\in \Omega}\lambda_{a,b} a\leftarrow b\otimes \bdtroisdeux(a\triangleleft b)+
\sum_{a,b\in \Omega} \mu_{a,b}a\rightarrow b\otimes\bdtroisun(a\triangleright b).
\end{align*}
Hence, there exists $(a,b)\in \Omega^2$, such that $a\rightarrow b=\alpha$ and $a\triangleright b=\beta$:
$\varphi_\rightarrow$ is surjective. Similarly, $\varphi_\leftarrow$ is surjective.\\

1. $(b)\Longrightarrow (a)$. Let us denote by $A$ the dendriform subalgebra of $\K\Omega\otimes \K\calT_\Omega^+$
generated by the elements $\alpha\otimes \bdeux$. Let us prove that for any $\alpha \in \Omega$,
$T\in \calT_\Omega$, $\alpha \otimes T\in A$ by induction on the number $N$ of leaves of $T$.
If $N=2$, then $T=\bdeux$ and it is obvious. Otherwise, let us put $T=T_1 \Y{\beta}{\gamma}T_2$. 
By the induction hypothesis, for $i=1$ or $2$, $T_i=\bun$ or $T_i\in A$.
If $T_1\neq\bun$, let $(\alpha',\beta')\in \Omega$ such that $\varphi_\rightarrow (\alpha',\beta')=(\alpha,\beta)$. 
Then
\begin{align*}
\alpha'\otimes T_1 \succ \beta'\otimes \bun \Y{\emptyset}{\gamma}T_2&
=\alpha'\rightarrow \beta'\otimes T_1 \succ_{\alpha'\triangleright \beta'}\bun \Y{\emptyset}{\gamma}T_2\\
&=\alpha \otimes T_1\succ_\beta \bun \Y{\emptyset}{\gamma}T_2\\
&=\alpha \otimes T_1\Y{\beta}{\gamma} T_2.
\end{align*}
So $\alpha \otimes T\in A$. Similarly, if $T_2\neq \bun$, then $\alpha \otimes T \in A$.\\

2. $(a)\Longrightarrow (b)$. Because of the graduation, $A$ is freely generated by the elements
$\alpha \otimes \bdeux$, with $\alpha \in \Omega$. Let $(\alpha,\beta), (\alpha',\beta') \in \Omega^2$,
such that $\varphi_\leftarrow(\alpha,\beta)=\varphi_\leftarrow(\alpha',\beta')$.
Then:
\[\alpha \otimes \bdeux \prec \beta \otimes \bdeux=\alpha \leftarrow \beta \otimes
\bdtroisdeux(\alpha \triangleleft \beta)\alpha' \leftarrow \beta' \otimes
\bdtroisdeux(\alpha' \triangleleft \beta')=\alpha' \otimes \bdeux \prec \beta' \otimes \bdeux.\]
By freeness of $\K\Omega \otimes \K\calT_\Omega^+$, $(\alpha,\beta)=(\alpha',\beta') $, so $\varphi_\leftarrow$ is injective.
The proof is similar for $\varphi_\rightarrow$.\\

2. $(b)\Longrightarrow (a)$. Let $\mathrm{Dend}(\Omega)$ be the free dendriform algebra generated
by $\Omega$. As a vector space, it is generated by plane binary trees which internal vertices are decorated by $\Omega$.
Let $\Theta:\mathrm{Dend}(\Omega)\longrightarrow \K\Omega \otimes \K\calT_\Omega^+$ be the unique dendriform algebra
morphism sending $\alpha\in \Omega$ to $\alpha \otimes \bdeux$. Then,
for any tree $T\in \mathrm{Dend}(\Omega)$, writing it as $T=T_1\bigvee_\alpha T_2$, $\alpha$ being the decoration of the root of $T$,
let us denote
\begin{align*}
\Theta(T_1)&=\sum_i \alpha_i T_1^{(i)},&\Theta(T_2)&=\sum_j\beta_j T_2^{(j)}.
\end{align*}
Then:
\[\Theta(T)=\sum_{i,j}\alpha_i\rightarrow \alpha\leftarrow \beta_j
\otimes T_1^{(i)}\Y{\alpha_i\triangleright \alpha}{(\alpha_i\rightarrow \alpha)\triangleleft \beta_j} T_2^{(j)}.\]
We conclude that $\Theta(T)$ is a typed tree of the same form as $T$, 
with types of edges obtained from the decorations of the vertices of $T$ by the application of compositions of maps
$Id^{\otimes (i-1)}\otimes \varphi_\leftarrow \otimes Id^{\otimes (n-i)}$ 
and $Id^{\otimes (i-1)}\otimes \varphi_\rightarrow \otimes Id^{\otimes (n-i)}$.
As $\varphi_\leftarrow$ and $\varphi_\rightarrow$ are injective, $\Theta$ is injective. \end{proof}

\begin{prop}\label{prop20}
Let $\Omega$ be a commutative EDS and $V$ be a nonzero vector space. 
\begin{enumerate}
\item The following assertions are equivalent:
\begin{enumerate}
\item The dendriform algebra $\K\Omega \otimes \Sh^+_\Omega(V)$ is generated by the elements
$\alpha \otimes v$, $\alpha \in \Omega$, $v\in V$.
\item $\varphi_\leftarrow$ is  surjective.
\end{enumerate}
\item The following assertions are equivalent:
\begin{enumerate}
\item The commutative dendriform subalgebra of $\K\Omega \otimes \Sh^+_\Omega(V)$ generated by the elements
$\alpha \otimes v$, $\alpha \in \Omega$, $v\in V$, is free.
\item $\varphi_\leftarrow$ is injective.
\end{enumerate}
\end{enumerate}
\end{prop}

\begin{proof}
1. $(a) \Longrightarrow (b)$. Let $\alpha,\beta \in \Omega$. Let us choose a nonzero element of $V$. 
Then $\alpha \otimes (\beta \otimes v)$ belongs to the dendriform subalgebra of $\K\Omega\otimes \Sh_\Omega^+(V)$.
As it is graded, it can be written under the form:
\begin{align*}
\alpha \otimes (\beta \otimes vv)&=\sum_i \alpha_i \otimes v_i \prec \beta_i\otimes w_i
=\sum_i \alpha_i\leftarrow \beta_i \otimes (\alpha_i\triangleleft \beta_i \otimes v_iw_i).
\end{align*}
where $v_i,w_i\in V$ and $\alpha_i,\beta_i\in \Omega$ for any $i$. Hence, there exists $i$, such that $\varphi_\leftarrow(\alpha_i,
\beta_i)=(\alpha,\beta)$.\\

1. $(b) \Longrightarrow (a)$. Let us assume that $\varphi_\leftarrow$ is surjective.
Let us denote by $A$ the dendriform subalgebra of $\K\Omega\otimes \Sh^+_\Omega(V)$ 
generated by the elements $\alpha \otimes v$.
Let us prove that any typed word $\alpha_2\ldots \alpha_k\otimes v_1\ldots v_k$, for any $\alpha_1\in \Omega$,
$\alpha_1\otimes (\alpha_2\ldots \alpha_k\otimes v_1\ldots v_k)$ belongs to $A$ by induction on $n$.
It is obvious if $n=1$. Otherwise, let $(\beta_1,\beta_2)\in \Omega^2$, such that $\varphi(\beta_1,\beta_2)=(\alpha_1,\alpha_2)$.
Then:
\begin{align*}
\beta_1\otimes v_1\prec \beta_2\otimes (\alpha_3\ldots \alpha_k\otimes v_2\ldots v_k)
&=\beta_1\leftarrow \beta_2 \otimes (v_1 \prec_{\beta_1\triangleleft \beta_2} \alpha_3\ldots \alpha_k\otimes v_2\ldots v_k)\\
&=\alpha_1\otimes (v_1 \prec_{\alpha_2} \alpha_3\ldots \alpha_k\otimes v_2\ldots v_k)\\
&=\alpha_1\otimes (\alpha_2\ldots \alpha_k\otimes v_1\ldots v_k).
\end{align*}
By the induction hypothesis, this belongs to $A$.\\

2. $(a)\Longrightarrow (b)$. Let $\alpha,\beta,\alpha',\beta'\in \Omega$ such that $\varphi_\leftarrow(\alpha,\beta)
=\varphi_\leftarrow(\alpha',\beta')$. Let $v\in V$, nonzero. Then:
\begin{align*}
\alpha\otimes v\prec \beta \otimes v&=\alpha \leftarrow \beta \otimes (\alpha \triangleleft \beta \otimes vv)
=\alpha' \leftarrow \beta' \otimes (\alpha' \triangleleft \beta' \otimes vv)=\alpha'\otimes v\prec \beta' \otimes v.
\end{align*}
By freeness of $\K\Omega \otimes \Sh_\Omega^+(V)$, $(\alpha,\beta)=(\alpha',\beta')$.\\

3. $(b)\Longrightarrow (a)$. Recall that the free commutative dendriform algebra generated by
$\K\Omega \otimes V$ is the shuffle algebra $\Sh^+(\K\Omega \otimes V)$, with the usual half-shuffle product. 
Hence, there exists a dendriform algebra morphism
$\Phi:\Sh^+(\K\Omega \otimes V)\longrightarrow \K\Omega\otimes \Sh^+_\Omega(V)$, sending $\alpha \otimes v$
to itself. For any $\alpha_1,\ldots,\alpha_n\in \Omega$, $v_1,\ldots,v_n \in V$,  in $\Sh^+(V)$:
\[(\alpha_1\otimes v_1)\ldots (\alpha_n\otimes v_n)
=(\alpha_1\otimes v_1)\prec \left((\alpha_2\otimes v_2)\ldots (\alpha_n\otimes v_n)\right).\]
Hence, an easy induction allows to prove that
\begin{align*}
\Psi((\alpha_1\otimes v_1)\ldots (\alpha_n\otimes v_n))&=(\varphi_\leftarrow \otimes Id^{\otimes (n-2)})\circ \ldots
\circ (Id^{\otimes (n-2)}\otimes \varphi_\leftarrow)(\alpha_1\ldots \alpha_n)\otimes v_1 \ldots v_n.
\end{align*}
As $\varphi_\leftarrow$ is injective, $\Psi$ is injective. \end{proof}

\section{Operad of $\Omega$-dendriform algebras}

We fix in this section an EDS $(\Omega,\leftarrow,\rightarrow,\triangleleft,\triangleright)$.

\subsection{Combinatorial description of the operad}

Let us denote by $\calP_\Omega$ the (nonsymmetric) operad of $\Omega$-dendriform algebras.
It is generated by elements $\prec_\alpha$, $\succ_\alpha \in \calP(2)$, with $\alpha \in \Omega$, and the relations:
\begin{align*}
&\forall \alpha,\beta \in \Omega,&
\prec_\beta \circ (\prec_\alpha,I)&=\prec_{\alpha \leftarrow \beta}\circ(I,\prec_{\alpha \triangleleft \beta})
+\prec_{\alpha \rightarrow \beta}\circ(I,\succ_{\alpha \triangleright \beta}),\\
&&\succ_\alpha \circ (I,\prec_\beta)&=\prec_\beta \circ (\succ_\alpha,I),\\
&&\succ_\alpha\circ (I,\succ_\beta)&=\succ_{\alpha\rightarrow \beta}\circ (\succ_{\alpha \triangleright \beta},I)
+\succ_{\alpha\leftarrow \beta}\circ (\prec_{\alpha \triangleleft \beta},I).
\end{align*}
As we know from Proposition \ref{prop15} a combinatorial description of the free $\Omega$-dendriform algebra
on one generator, we obtain a combinatorial description of this operad:
\begin{align*}
&\forall n \geqslant 1,&\calP_\Omega(n)&=\K\calT_\Omega(n).
\end{align*}
The composition is given by the actions of the products of $\K\calT_\Omega^+$. In particular:
\begin{align*}
I&=\bdeux,&\prec_\alpha&=\bdeux \prec_\alpha \bdeux=\bdtroisdeux(\alpha),&
\succ_\alpha&=\bdeux \succ_\alpha \bdeux=\bdtroisun(\alpha).
\end{align*}
The operadic composition can be inductively computed with the help of the following formula:
\[T_1\Y{\alpha}{\beta} T_2\circ (T'_1,\ldots,T'_k)
=T_1\circ (T'_1,\ldots,T'_i) \succ_\alpha T'_{i+1} \prec_\beta T_2\circ (T'_{i+2},\ldots,T'_k),\]
where $T_1$ is a tree with $i$ internal vertices, $T_2$ is a tree with  $k-i-1$ internal vertices, 
and $T'_1,\ldots,T'_k$ are trees. 

\begin{example} Here are examples of operadic compositions:
\begin{align*}
\prec_\alpha\circ (\prec_\beta,I)&=\bdtroisdeux(\beta)\prec_\alpha \bdeux
=\bdquatrequatre(\beta \leftarrow \alpha,\beta \triangleleft \alpha)+
\bdquatretrois(\beta \rightarrow \alpha,\beta \triangleright \alpha),\\
\prec_\alpha\circ (I,\prec_\beta)&= \bdeux\prec_\alpha\bdtroisdeux(\beta)=\bdquatrequatre(\alpha,\beta),\\
\succ_\alpha\circ(\prec_\beta,I)&=\bdtroisdeux(\beta)\succ_\alpha \bdeux
=\bdquatredeux(\alpha,\beta),\\
\succ_\alpha\circ(I,\prec_\beta)&=\bdeux\succ_\alpha \bdtroisdeux(\beta)
=\bdquatrecinq(\alpha,\beta),\\
\prec_\alpha\circ (\succ_\beta,I)&=\bdtroisun(\beta)\prec_\alpha \bdeux
=\bdquatrecinq(\beta,\alpha),\\
\prec_\alpha\circ (I,\succ_\beta)&= \bdeux\prec_\alpha\bdtroisun(\beta)=\bdquatretrois(\alpha,\beta),\\
\succ_\alpha\circ(\succ_\beta,I)&=\bdtroisun(\beta)\succ_\alpha \bdeux
=\bdquatreun(\alpha,\beta),\\
\succ_\alpha\circ(I,\succ_\beta)&=\bdeux\succ_\alpha \bdtroisun(\beta)
=\bdquatredeux(\alpha\leftarrow \beta,\alpha \triangleleft \beta)+
\bdquatreun(\alpha \rightarrow \beta,\alpha \triangleright \beta).
\end{align*}\end{example}

\subsection{Associative products}

\begin{prop} \label{prop21}
Let $m\in \calP_\Omega(2)$, written under the form
\begin{align*}
m&=\sum_{\alpha \in \Omega} a_\alpha \prec_\alpha+\sum_{\alpha \in \Omega} b_\alpha \succ_\alpha.
\end{align*}
Then $m\circ (I,m)=m\circ(m,I)$ if, and only if, for any $\alpha,\beta\in \Omega$:
\begin{align}
\label{eq44}
b_\alpha b_\beta&=\sum_{\varphi_\rightarrow(\alpha',\beta')=(\alpha,\beta)}b_{\alpha'}b_{\beta'},&
a_\alpha a_\beta&=\sum_{\varphi_\leftarrow(\alpha',\beta')=(\alpha,\beta)}a_{\alpha'}a_{\beta'},\\
\nonumber b_\alpha a_\beta&=\sum_{\varphi_\leftarrow(\alpha',\beta')=(\alpha,\beta)}b_{\alpha'}b_{\beta'},&
a_\alpha b_\beta&=\sum_{\varphi_\rightarrow(\alpha',\beta')=(\alpha,\beta)}a_{\alpha'}a_{\beta'}.
\end{align}
\end{prop}

\begin{proof}
Indeed:
\begin{align*}
m\circ(I,m)&=\sum_{\alpha,\beta\in \Omega} a_\alpha a_\beta\bdquatrequatre(\alpha,\beta)
+\sum_{\alpha,\beta\in \Omega} a_\alpha b_\beta\bdquatretrois(\alpha,\beta)\\
&+\sum_{\alpha,\beta\in \Omega} b_\alpha b_\beta\left(\bdquatreun(\alpha\rightarrow \beta,\alpha \triangleright\beta)
+\bdquatredeux(\alpha\leftarrow \beta,\alpha \triangleleft \beta)\right)
+\sum_{\alpha,\beta\in \Omega} b_\alpha a_\beta\bdquatrecinq(\alpha,\beta),\\
m\circ(m,I)&=\sum_{\alpha,\beta\in \Omega} a_\alpha a_\beta
\left(\bdquatretrois(\alpha\rightarrow \beta,\alpha \triangleright\beta)
+\bdquatrequatre(\alpha\leftarrow \beta,\alpha \triangleleft \beta)\right)
+\sum_{\alpha,\beta\in \Omega} a_\alpha b_\beta \bdquatrecinq(\beta,\alpha)\\
&+\sum_{\alpha,\beta\in \Omega} b_\alpha a_\beta \bdquatredeux(\alpha,\beta)
+\sum_{\alpha,\beta\in \Omega} b_\alpha b_\beta \bdquatreun(\alpha,\beta).
\end{align*}
Identifying, we obtain the announced equations. \end{proof}

\begin{cor}
If $\Omega$ is nondegenerate, then $m\circ (I,m)=m\circ(m,I)$ if, and only if, 
for any $\alpha,\beta\in \Omega$:
\begin{align}
\label{eq45} b_{\alpha\rightarrow \beta} b_{\alpha \triangleright\beta}
&=b_{\alpha}b_{\beta},&
a_{\alpha\leftarrow \beta} a_{\alpha \triangleleft\beta}
&=a_{\alpha}a_{\beta},\\
\nonumber b_{\alpha\leftarrow \beta} a_{\alpha \triangleleft\beta}
&=b_{\alpha}b_{\beta},&
a_{\alpha\rightarrow \beta} b_{\alpha \triangleright\beta}
&=a_{\alpha}a_{\beta}.
\end{align}\end{cor}

In particular cases of EDS:

\begin{prop}\begin{enumerate}
\item Let $(\Omega,\star)$ be a group. In $\calP_{\EDS(\Omega,\star,\star)}$, the nonzero associative products are of the form
\[\lambda\sum_{\alpha \in G}\prec_\alpha+\succ_\alpha,\]
where $\lambda$ is a nonzero scalar and $G$ is a subgroup of $(\Omega,\star)$.
\item Let $(H,\star)$ be a group, $K$ be a nonempty set and $\theta:K\longrightarrow H$ be a map.
In $\calP_{\EDS^*(H,\star,K,\theta)}$,  the nonzero associative products are of the form
\[\sum_{\alpha'\in K} \lambda_{\alpha'}\left(\sum_{\alpha \in G}\prec_{(\alpha,\alpha')}
+\succ_{(\theta(\alpha')\star\alpha,\alpha')}\right),\]
where $(\lambda_{\alpha'})_{\alpha'\in K}$ is a nonzero family of scalars with finite support 
and $G$ is a subgroup of $(H,\star)$.
\end{enumerate}\end{prop}

\begin{proof}
1. In this case, (\ref{eq45}) becomes:
\begin{align*}
b_{\alpha \star \beta} b_\alpha&=b_\alpha b_\beta,&
a_{\alpha \star \beta} a_\beta&=a_\alpha a_\beta,\\
b_{\alpha \star \beta} b_\beta&=b_\alpha a_\beta,&
a_{\alpha \star \beta} a_\alpha&=a_\alpha b_\beta.
\end{align*}
We put $G_a=\{\alpha \in \Omega, a_\alpha\neq 0\}$ and $G_b=\{\alpha \in \Omega,b_\alpha\neq 0\}$.
At least one of them is not empty: let us assume for example that $G_b\neq \emptyset$. 
Let $\alpha \in G_b$. If $\beta=e$ is the unit of $\Omega$:
\[b_{\alpha \star e}b_\alpha=b_\alpha^2=b_\alpha b_e,\]
so $b_e=b_\alpha\neq 0$: $e\in G_b$ and for any $\alpha \in G_b$, $b_\alpha=b_e$. 
If $\alpha,\beta \in G_b$, then:
\[b_{\alpha \star \beta} b_\alpha=b_\alpha b_\beta \neq 0,\]
so $\alpha \star \beta \in G_b$. For any $\alpha \in G_b$, if $\beta=\alpha^{-1}$:
\[b_e b_\alpha=b_\alpha b_{\alpha^{-1}}\neq 0,\]
so $\alpha^{-1}\in G_b$: we proved that $G_b$ is a subgroup of $\Omega$.
If $\beta \in G_a$, then
\[a_e b_\beta=a_{e\star \beta}a_\beta=a_\beta^2\neq 0,\]
so $\beta \in G_b$: $G_a\subseteq G_b$.  Conversely, if $\beta \in G_b$,
\[b_e a_\beta=b_{e\star \beta}b_\beta=b_\beta^2\neq 0,\]
so $a_\beta \neq 0$: $G_b\subseteq G_a$. Moreover, as $b_e=b_\beta$ for any $\beta \in \Omega$,
we obtain that $a_\beta=b_e$. Putting $\lambda=b_\beta$ and $G=G_a=G_b$, we obtain that
\[m=\lambda \sum_{\alpha \in G} \prec_\alpha+\succ_\alpha.\]
Conversely, for such a $m$, (\ref{eq45}) is satisfied, so $m$ is associative. \\

2. In this case, (\ref{eq45}) becomes:
\begin{align*}
b_{(\beta,\beta')}b_{(\theta(\beta')\star \beta^{-1}\star \alpha,\alpha')}&=b_{(\alpha,\alpha')}b_{(\beta,\beta')},\\
a_{(\alpha,\alpha')}a_{(\alpha^{-1}\star \beta,\beta')}&=a_{(\alpha,\alpha')}a_{(\beta,\beta')},\\
b_{(\alpha,\alpha')}a_{(\alpha^{-1}\star \beta,\beta')}&=b_{(\alpha,\alpha')}b_{(\beta,\beta')},\\
a_{(\beta,\beta')}b_{(\theta(\beta')\star \beta^{-1}\star \alpha,\alpha')}&=a_{(\alpha,\alpha')}a_{(\beta,\beta')}.
\end{align*}
For any $\alpha'\in K$, we put:
\begin{align*}
G_a(\alpha')&=\{\alpha \in H,\: a_{(\alpha,\alpha')}\neq 0\},&
G_b(\alpha')&=\{\alpha \in H,\: b_{(\theta(\alpha')\star\alpha,\alpha')}\neq 0\}.
\end{align*}

Let $\alpha'\in K$, such that $G_a(\alpha')\neq \emptyset$. For any $\alpha,\beta \in G_a(\alpha')$:
\begin{align*}
a_{(\alpha,\alpha')}a_{(\beta,\alpha')}&=a_{(\alpha,\alpha')} a_{(\alpha^{-1}\star \beta,\alpha')},&
a_{(\beta,\alpha')}&= a_{(\alpha^{-1}\star \beta,\alpha')}.
\end{align*}
So $\alpha^{-1}\star \beta \in G_a(\alpha')$. Hence, $G_a(\alpha')$ is a subgroup of $H$.
Moreover, there exists a nonzero scalar $a_{\alpha'}$ such that for any $\alpha \in G_a(\alpha')$,
$a_{(\alpha,\alpha')}=a_{\alpha'}$.

Let $\alpha'\in K$, such that $G_b(\alpha')\neq \emptyset$. For any $\alpha,\beta \in G_b(\alpha')$:
\begin{align*}
b_{(\theta(\alpha')\star \alpha,\alpha')}b_{(\theta(\alpha')\star \beta,\alpha')}
&=b_{(\theta(\alpha')\star \beta^{-1}\star \theta(\alpha')^{-1}\star \theta(\alpha') \star \alpha,\alpha')}
b_{(\theta(\alpha')\star \beta,\alpha')},\\
b_{(\theta(\alpha')\star \alpha,\alpha')}&=
b_{(\theta(\alpha')\star \beta^{-1} \star \alpha,\alpha')}.
\end{align*}
So $\beta^{-1}\star \alpha \in G_b(\alpha')$. Hence, $G_b(\alpha')$ is a subgroup of $H$.
Moreover, there exists a nonzero scalar $b_{\alpha'}$ such that for any $\alpha \in G_b(\alpha')$,
$b_{(\theta(\alpha')\star\alpha,\alpha')}=b_{\alpha'}$. \\

Let $\alpha \in G_a(\alpha')$. Then $G_a(\alpha')\neq \emptyset$, so is a subgroup of $H$, and the unit
$e$ of $H$ belongs to $G_a(\alpha')$. Then:
\[0\neq a_{(\alpha,\alpha')}a_{(e,\alpha')}=a_{(e,\alpha')} b_{(\theta(\alpha')\star \alpha,\alpha')}.\]
Therefore, $\alpha \in G_b(\alpha')$: we obtain that $G_a(\alpha')\subseteq G_b(\alpha')$. 

Let $\beta \in G_b(\alpha')$. Then $G_b(\alpha')\neq \emptyset$ is a subgroup of $H$, and $e\in G_b(\alpha')$. 
Hence:
\[0\neq b_{(\theta(\alpha'),\alpha')}b_{(\theta(\alpha')\star \beta,\alpha')}
=b_{(\theta(\alpha'),\alpha')}a_{(\theta(\alpha')^{-1}\star \theta(\alpha')\star \beta,\alpha'}.\]
We obtain that $\beta \in G_a(\alpha')$. Finally, for any $\alpha' \in K$, $G_a(\alpha')=G_b(\alpha')$.
We denote this set by $G(\alpha')$. \\

If $G(\alpha')\neq \emptyset$, we obtain, for $\alpha=\beta=e$:
\[a_{\alpha'}b_{\alpha'}=a_{\alpha'}a_{\alpha'}.\]
Consequently, $b_{\alpha'}=a_{\alpha'}$. We denote by $\lambda_{\alpha'}$ this scalar. \\

As $m\neq 0$, at least one of the $G(\alpha')$ is nonempty. We consider
\[K'=\{\alpha'\in K,\: G(\alpha')\neq \emptyset\}.\]
Let $\alpha',\beta'\in K'$. For any $\alpha \in G(\alpha')$, for $\beta=e$:
\[\lambda_{\alpha'}a_{(\alpha^{-1},\beta')}=\lambda_{\alpha'}\lambda_{\beta'}\neq 0.\]
Hence, $\alpha^{-1}\in G(\beta')$. As this is a subgroup, $\alpha \in G(\beta')$, and $G(\alpha')\subseteq G(\beta')$.
By symmetry, $G(\alpha')=G(\beta')$. We denote by $G$ this subset. Then:
\begin{align*}
m&=\sum_{\alpha'\in K'} \lambda_{\alpha'}\sum_{\alpha \in G} \prec_{(\alpha,\alpha')}
+\succ_{(\theta(\alpha')\star \alpha,\alpha')}
=\sum_{\alpha'\in K} \lambda_{\alpha'}\sum_{\alpha \in G} \prec_{(\alpha,\alpha')}
+\succ_{(\theta(\alpha')\star \alpha,\alpha')},
\end{align*}
where we put $\lambda_{\alpha'}=0$ if $\alpha'\notin K$. 
Conversely, for such a $m$, (\ref{eq45}) is satisfied, so $m$ is associative.  \end{proof}

\begin{example}\begin{enumerate}
\item If $H$ is a null group, we obtain the case of $\EDS(\Omega)$. The associative products are of the form
\[\sum_{\alpha \in \Omega} \lambda_\alpha(\prec_\alpha+\succ_\alpha),\]
where $(\lambda_{\alpha})_{\alpha\in \Omega}$ is family of scalars with finite support.
\item If $K$ is reduced to a single element, we obtain the case of $\EDS^*(\Omega,\star,\omega)$. 
The associative products are of the form
\[\lambda \sum_{\alpha \in H} \prec_\alpha+\succ_{\omega\star \alpha},\]
where $\lambda$ is a scalar and $H$ is a subgroup of $(\Omega,\star)$.
\end{enumerate}\end{example}

\begin{prop}
Let $\Omega$ be an EDS. We suppose that there exist $\alpha_0,\beta_0\in \Omega$
such that:
\begin{align*}
&\forall\alpha,\beta \in \Omega,&\alpha \triangleleft \beta&=\beta_0,&\alpha \triangleright \beta&=\alpha_0.
\end{align*}
The associative products of $\calP_\Omega$ are of the form
\[\lambda(\prec_{\beta_0}+\succ_{\alpha_0}),\]
where $\lambda$ is a scalar.
\end{prop}

\begin{proof} 
Let $m$ be an associative product. In this case, (\ref{eq44}) becomes:
\begin{align*}
b_\alpha b_\beta&=\begin{cases}
\displaystyle \sum_{\alpha'\rightarrow \beta'=\alpha} b_{\alpha'}b_{\beta'}\mbox{ if }\beta=\alpha_0,\\
0\mbox{ otherwise;}
\end{cases}
&a_\alpha a_\beta&=\begin{cases}
\displaystyle \sum_{\alpha'\leftarrow \beta'=\alpha} a_{\alpha'}a_{\beta'}\mbox{ if }\beta=\beta_0,\\
0\mbox{ otherwise;}
\end{cases}\\
b_\alpha a_\beta&=\begin{cases}
\displaystyle \sum_{\alpha'\leftarrow \beta'=\alpha} b_{\alpha'}b_{\beta'}\mbox{ if }\beta=\beta_0,\\
0\mbox{ otherwise;}
\end{cases}
&a_\alpha b_\beta&=\begin{cases}
\displaystyle \sum_{\alpha'\rightarrow \beta'=\alpha} a_{\alpha'}a_{\beta'}\mbox{ if }\beta=\alpha_0,\\
0\mbox{ otherwise.}
\end{cases}
\end{align*}
In particular, if $\alpha=\beta\neq \alpha_0$, $b_\alpha^2=0$, so $b_\alpha=0$.
Similarly, if $\beta \neq \beta_0$, $a_\beta=0$. By (\ref{eq6}), (\ref{eq8}), (\ref{eq10}) and (\ref{eq12}):
\begin{align*}
\alpha_0\leftarrow \alpha_0&=\alpha_0\rightarrow \alpha_0=\alpha_0,&
\beta_0\leftarrow \beta_0&=\beta_0\rightarrow \beta_0=\beta_0.
\end{align*}
Hence, $b_{\alpha_0}a_{\beta_0}=b_{\alpha_0}^2$, so $b_{\alpha_0}=0$ or $b_{\alpha_0}=a_{\beta_0}$;
$a_{\beta_0}b_{\alpha_0}=a_{\beta_0}^2$, so $a_{\beta_0}=0$ or $b_{\alpha_0}=a_{\beta_0}$.
Finally, $a_{\beta_0}=b_{\alpha_0}=\lambda$ and $m=\lambda(\prec_{\beta_0}+\succ_{\alpha_0})$.
The converse is trivial. \end{proof}

\begin{example}
Let us give the associative products in the 24 four cases of cardinality $2$.
Here, $\lambda$, $\mu$ and $\nu$ are scalars.
\[\begin{array}{|c|c||c|c|}
\hline A1&\lambda(\prec_a+\succ_a)&A2&\lambda(\prec_a+\succ_a)\\
\hline B1&\lambda(\prec_a+\succ_a)&B2&\lambda(\prec_a+\succ_a),\mu(\succ_a-\succ_b)\\
\hline C1&\lambda(\prec_a+\succ_a)&C2&\lambda(\prec_a+\succ_b)\\
\hline C3&\lambda(\prec_a+\succ_a),\mu(\prec_b+\succ_b)&C4&\lambda(\prec_b+\succ_a)\\
\hline C5&\lambda(\prec_b+\succ_b)&D1&\lambda(\prec_a+\succ_a)\\
\hline D2&\lambda(\prec_a+\succ_a),\mu(\prec_a-\prec_b)&E1&\lambda(\prec_a+\succ_a)\\
\hline E2&\lambda(\prec_a+\succ_b)&E3&\lambda(\prec_a+\succ_a), \mu(\prec_b+\succ_b),
\nu(\prec_a-\prec_b)\\
\hline F1&\lambda(\prec_a+\succ_a)&F2&\lambda(\prec_a+\succ_b)\\
\hline F3&\lambda(\prec_a+\succ_a)+\mu(\prec_b+\succ_b)&F4&
\lambda(\prec_a+\succ_a),\mu(\prec_a+\prec_b+\succ_a+\succ_b)\\
\hline F5&\lambda(\prec_a+\succ_b),\nu(\prec_a+\prec_b+\succ_a+\succ_b)&G1&\lambda(\prec_a+\succ_a)\\
\hline G2&\lambda(\prec_a+\succ_b)&G3&\lambda(\prec_a+\succ_a),\mu(\prec_b+\succ_b),\nu(\succ_a-\succ_b)\\
\hline H1&\lambda(\prec_a+\succ_a)&H2&\lambda(\prec_a+\succ_a),
\mu(\prec_a+\prec_b+\succ_a+\succ_b)\\
\hline \end{array}\]
\end{example}

\subsection{Dendriform products}

\begin{prop}
Let $\prec,\succ\in \calP(2)$, written under the form
\begin{align*}
\prec&=\sum_{\alpha \in \Omega} a_\alpha \prec_\alpha+\sum_{\alpha \in \Omega} b_\alpha \succ_\alpha,&
\succ&=\sum_{\alpha \in \Omega}c_\alpha \prec_\alpha+\sum_{\alpha \in \Omega} d_\alpha \succ_\alpha,\\
\end{align*}
Then $(\prec,\succ)$ satisfies the dendriform relations
\begin{align*}
\prec \circ (\prec,I)&=\prec \circ (I,\prec+\succ),&
\prec \circ (\succ,I)&=\succ\circ (I,\prec),&
\succ \circ (I,\succ)&=\succ \circ (\prec+\succ,I)
\end{align*}
if, and only if, for any $\alpha,\beta\in \Omega$:
\begin{align}
\label{eq46}
b_\alpha b_\beta&=\sum_{\varphi_\rightarrow(\alpha',\beta')=(\alpha,\beta)}b_{\alpha'}(b_{\beta'}+d_{\beta'}),&
a_\alpha (a_\beta+c_\beta)&=\sum_{\varphi_\leftarrow(\alpha',\beta')=(\alpha,\beta)}a_{\alpha'}a_{\beta'},\\
\nonumber b_\alpha a_\beta&=\sum_{\varphi_\leftarrow(\alpha',\beta')=(\alpha,\beta)}b_{\alpha'}(b_{\beta'}+d_{\beta'}),&
a_\alpha (b_\beta+d_\beta)&=\sum_{\varphi_\rightarrow(\alpha',\beta')=(\alpha,\beta)}a_{\alpha'}a_{\beta'},\\
\nonumber b_\alpha c_\beta&=0,\\
\nonumber \\
\nonumber 0&=\sum_{\varphi_\leftarrow(\alpha',\beta')=(\alpha,\beta)}d_{\alpha'}b_{\beta'},&
b_\alpha d_\beta&=\sum_{\varphi_\rightarrow(\alpha',\beta')=(\alpha,\beta)}d_{\alpha'}b_{\beta'},\\
\nonumber 0&=\sum_{\varphi_\rightarrow(\alpha',\beta')=(\alpha,\beta)}c_{\alpha'}a_{\beta'},&
c_\alpha a_\beta&=\sum_{\varphi_\leftarrow(\alpha',\beta')=(\alpha,\beta)}c_{\alpha'}a_{\beta'},\\
\nonumber\\
\nonumber c_\alpha c_\beta&=\sum_{\varphi_\leftarrow(\alpha',\beta')=(\alpha,\beta)}(a_{\alpha'}+c_{\alpha'})c_{\beta'},&
d_\alpha (b_\beta+d_\beta)&=\sum_{\varphi_\rightarrow(\alpha',\beta')=(\alpha,\beta)}d_{\alpha'}d_{\beta'},\\
\nonumber c_\alpha d_\beta&=\sum_{\varphi_\rightarrow(\alpha',\beta')=(\alpha,\beta)}(a_{\alpha'}+c_{\alpha'})c_{\beta'},&
d_\alpha (a_\beta+c_\beta)&=\sum_{\varphi_\leftarrow(\alpha',\beta')=(\alpha,\beta)}d_{\alpha'}d_{\beta'}.
\end{align}
\end{prop}

\begin{proof}
By direct computations in the operad $\calP_\Omega$, as for Proposition \ref{prop21}.
\end{proof}

Note that if $(\prec,\succ)$ satisfies the dendriform relations, then $\shuffle=\prec+\succ$ is associative:
\[\shuffle\circ (\shuffle,I)=\shuffle\circ (I,\shuffle).\]
In the nondegenerate case, the knowledge of the associative products of $\calP\Omega$ induces the knowledge
of all dendriform products:

\begin{cor}
Let $\Omega$ be a nondegenerate EDS. For any associative product
\[m=\sum_{\alpha \in \Omega} a_\alpha \prec_\alpha+\sum_{\alpha \in \Omega} d_\alpha \succ_\alpha\in \calP(2),\]
the only pairs of  dendriform products $(\prec,\succ)$ such that $\prec+\succ=m$ are the following:
\begin{align*}
&(m,0),&&(0,m),&&
\left(\sum_{\alpha\in \Omega}a_\alpha\prec_\alpha,\sum_{\alpha\in \Omega}d_\alpha\succ_\alpha\right).
\end{align*}\end{cor}

\begin{proof}
Let $(\prec,\succ)$ be a pair of dendriform products and $m=\prec+\succ$. 
As $\varphi_\leftarrow$ and $\varphi_\rightarrow$ are bijective, (\ref{eq46}) gives 
(third, fourth fifth and eighth rows, first column) that:
\begin{align*}
&\forall \alpha,\beta \in \Omega,&b_\alpha c_\beta&=0,&d_\alpha b_\beta&=0,&c_\alpha a_\beta&=0,&c_\alpha b_\beta&=0.
\end{align*}
If one of the $b_\beta$ is nonzero, then for any $\alpha \in \Omega$, $c_\alpha=d_\alpha=0$, so $\succ=0$ and $\prec=m$.
Similarly, if one of the $c_\beta$ is nonzero, then for $\prec=0$ and $\succ=m$.
If for any $\beta\in \Omega$, $b_\beta=c_\beta=0$, then:
\begin{align}
\label{eq47} \prec&=\sum_{\alpha \in \Omega} a_\alpha \prec_\alpha,&
\succ&=\sum_{\alpha \in \Omega} d_\alpha \succ_\alpha.
\end{align}
Conversely, if $m$ is an associative product, written under the form:
\begin{align*}
m&=\sum_{\alpha \in \Omega} a_\alpha \prec_\alpha+d_\alpha \succ_\alpha,
\end{align*}
then obviously, $(0,m)$ and $(m,0)$ are pairs of dendriform products. 
If we define $(\prec,\succ)$ by (\ref{eq47}), that is to say $b_\alpha=c_\alpha=0$ for any $\alpha\in \Omega$, 
then (\ref{eq44}) implies (\ref{eq46}), so $(\prec,\succ)$ is dendriform. 
\end{proof}

\subsection{Koszul dual}

When $\Omega$ is finite, the operad $\calP_\Omega$ is a quadratic algebra, finitely generated.
By direct computations, we obtain the Koszul dual of $\calP_\Omega$:

\begin{prop} \label{prop27}
Let $\Omega$ be a finite EDS. The Koszul dual $\calP_\Omega^!$ of $\calP_\Omega$
is generated by the elements $\dashv_\alpha$, $\vdash_\alpha$, $\alpha\in \Omega$, with the relations:
\begin{align*}
&\forall \alpha,\beta \in \Omega,&
\dashv_\beta\circ(\dashv_\alpha,I)&=\dashv_{\alpha\leftarrow \beta}\circ(I,\dashv_{\alpha\triangleleft\beta})
=\dashv_{\alpha\rightarrow \beta}\circ(I,\vdash_{\alpha\triangleright \beta}),\\
&&\vdash_\beta\circ(I,\dashv_\alpha)&=\dashv_\alpha\circ(\vdash_\beta,I),\\
&&\vdash_\alpha\circ (I\vdash_\beta)&=\vdash_{\alpha\rightarrow \beta}\circ(\vdash_{\alpha\triangleright \beta},I)
=\vdash_{\alpha\leftarrow \beta}\circ(\dashv_{\alpha\triangleleft\beta},I).
\end{align*}
%Algebras on $\calP_\Omega^!$ will be called $\Omega$-diassociative algebras.
\end{prop}

\begin{defi} \label{defi28}
Let $(\Omega,\leftarrow,\rightarrow,\triangleleft,\triangleright)$ be an EDS.
We consider the following linear map:
\[\varphi:\left\{\begin{array}{rcl}
\K \Omega^2&\longrightarrow&\K \Omega^2 \times\K \Omega^2\\
(\alpha,\beta)&\longrightarrow&(\varphi_\leftarrow(\alpha,\beta),\varphi_\rightarrow(\alpha,\beta)).
\end{array}\right.\]
The dimension of the kernel of $\varphi$ is called the corank of $\Omega$ and denoted by $\mathrm{coRk}(\Omega)$.
We shall say that $ \Omega$ is weakly nondegenerate if $\mathrm{coRk}(\Omega)=0$. 
\end{defi}

\begin{example} \begin{enumerate}
\item If $\varphi_\leftarrow$ or $\varphi_\rightarrow$ is injective (which happens if $\Omega$ is nondegenerate),
then $\Omega$ is weaky nondegenerate.
\item If $(\Omega,\leftarrow,\rightarrow)$ is a diassociative semigroup, then $\EDS(\Omega,\leftarrow,\rightarrow)$
is weakly nondegenerate. Indeed, in this case,
\[\varphi(\alpha,\beta)=((\alpha \leftarrow \beta,\beta),(\alpha,\rightarrow \beta,\beta)),\]
so $\varphi$ is injective.
\item Here are the coranks of the 24 EDS of cardinality 2.
\[\begin{array}{|c|c||c|c||c|c||c|c|}
\hline A1&3&A2&1&B1&2&B2&0\\
\hline C1&2&C2&2&C3&0&C4&2\\
\hline C5&2&D1&2&D2&0&E1&2\\
\hline E2&2&E3&0&F1&1&F2&1\\
\hline F3&0&F4&0&F5&0&G1&2\\
\hline G2&2&G3&0&H1&2&H2&0\\
\hline \end{array}\]
\end{enumerate} \end{example}

\begin{prop}
Let $(\Omega,\leftarrow,\rightarrow,\triangleleft,\triangleright)$ be a finite EDS.
Then:
\[\dimK(\calP_\Omega^!(3))=3|\Omega|^2+2\mathrm{coRk}(\Omega).\]
\end{prop}

\begin{proof} For any $m^{(1)},m^{(2)}\in \{\dashv,\vdash\}$, we shall consider
the following subspaces of the free operad generated by $\vdash_\alpha$, $\dashv_\alpha$, with $\alpha\in \Omega$:
\begin{align*}
L(m_1,m_2)&=Vect\left(m^{(1)}_\alpha\circ (m^{(2)}_\beta,I),\alpha,\beta \in \Omega\right),\\
R(m_1,m_2)&=Vect\left(m^{(1)}_\alpha\circ (I,m^{(2)}_\beta),\alpha,\beta \in \Omega\right).
\end{align*}
According to the form of the relations defining $\calP_\Omega^!$:
\begin{align*}
\calP_\Omega^!(3)=R(\vdash,\dashv)\oplus \frac{L(\vdash,\vdash)\oplus L(\vdash,\dashv)}{E_1}\oplus
\frac{R(\dashv,\dashv)\oplus R(\dashv,\vdash)}{E_2},
\end{align*}
with:
\begin{align*}
E_1&=Vect(\vdash_{\alpha\leftarrow \beta}\circ (\dashv_{\alpha\triangleleft \beta},I)
-\vdash_{\alpha\rightarrow \beta}\circ (\vdash_{\alpha\triangleright \beta},I),\: (\alpha,\beta)\in \Omega^2),\\
E_2&=Vect(\dashv_{\alpha\leftarrow \beta}\circ (I,\dashv_{\alpha\triangleleft \beta})
-\dashv_{\alpha\rightarrow \beta}\circ (I,\vdash_{\alpha\triangleright \beta}),\: (\alpha,\beta)\in \Omega^2).
\end{align*}
Hence:
\[\dimK(\calP_\Omega^!(3))=5|\Omega|^2-\dimK(E1)-\dimK(E_2).\]
By definition of $\varphi$, $\dimK(E_1)=\dimK(E_2)=\dimK(Im(\varphi))=|\Omega|^2-\mathrm{coRk}(\Omega)$, which gives the result.
\end{proof}

\begin{theo}
Let $\Omega$ be a finite EDS. 
\begin{enumerate}
\item If $\calP_\Omega$ is Koszul, then $\Omega$ is weaky nondegenerate.
\item If $\Omega$ is nondegenerate, then $\calP_\Omega$ is Koszul.
\end{enumerate}
\end{theo}

\begin{proof} We put $\omega=|\Omega|$.\\
$1$. Let us assume that $\calP_\Omega$ is Koszul. The Poincaré-Hilbert formal series of $\calP_\Omega$ is:
\[F=\sum_{k=1}^\infty \dimK(\calP_\Omega(n))X^n=X+2\omega X^2+5\omega^2X^3+\ldots
=\frac{1-\sqrt{1-4\omega X}}{2\omega^2 X}-\frac{1}{\omega}.\]
We denote by $G$ the Poincaré-Hilbert formal series of $\calP_\Omega^!$. As $\calP_\Omega$ is Koszul:
\[F(-G(-X))=X,\]
so:
\[G=\frac{X}{(1-\omega X)^2}=\sum_{n=1}^\infty n\omega^{n-1}X^n,\]
so $\dimK(\calP_\Omega^!(3))=3\omega^2$. Therefore, $\mathrm{coRk}(\Omega)=0$. \\

$2$. We use the rewriting method of \cite{Dotsenko} to prove that $\calP_\Omega^!$ is Koszul.
The rewriting rules are the following:
\begin{align*}
&\xymatrix{ \dashv_\beta\circ (\dashv_\alpha,I)\ar[rd]&\\
&\dashv_{\alpha \rightarrow \beta}\circ (I,\vdash_{\alpha \triangleright \beta})
\\
\dashv_{\alpha \leftarrow \beta}\circ(I,\dashv_{\alpha \triangleleft \beta}) \ar[ru]
&}&
&\xymatrix{ \vdash_{\alpha \rightarrow\beta}\circ (\vdash_{\alpha\triangleright \beta},I)\ar[rd]&\\
&\vdash_\alpha\circ (I,\vdash_\beta)
\\
\vdash_{\alpha \leftarrow\beta}\circ (\dashv_{\alpha\triangleleft \beta},I) \ar[ru]
&}\\ \\
&\xymatrix{\dashv_\beta\circ (\vdash_\alpha,I)\ar[r]&\vdash_\alpha\circ(I,\dashv_\beta)}
\end{align*}
There are 14 critical trees, giving 14 diagrams which turn out to be all confluent. Let us describe two of them.
\[\xymatrix{\dashv_\gamma\circ (\dashv_\beta\circ (\dashv_\alpha,I))\ar[r] \ar[d]&
\dashv_\gamma\circ(\dashv_{\alpha \leftarrow \beta}\circ(I,\dashv_{\alpha\triangleleft \beta}))\ar[d]\\
\dashv_{\beta\leftarrow \gamma}\circ (\dashv_\alpha,\dashv_{\beta \triangleleft\gamma})\ar[d]
&\dashv_{(\alpha\leftarrow \beta)\leftarrow \gamma}\circ (I,\dashv_{(\alpha\leftarrow \beta)\triangleleft \gamma}\circ
(\dashv_{\alpha\triangleleft \beta},I))\ar[d]\\
\dashv_{\alpha_1}\circ (I,\dashv_{\beta_1}\circ
(I,\dashv_{\gamma_1}))\ar@{=}[r]^?
&\dashv_{\alpha_2}\circ
(I,\dashv_{\beta_2} 
\circ(I,\dashv_{\gamma_2} ))}\]
with:
\begin{align*}
\alpha_1&=\alpha\leftarrow(\beta\leftarrow \gamma),&
\alpha_2&=(\alpha\leftarrow \beta)\leftarrow \gamma,\\
\beta_1&=\alpha \triangleleft(\beta\leftarrow \gamma),&
\beta_2&=(\alpha \triangleleft \beta)\leftarrow ((\alpha\leftarrow \beta)\triangleleft \gamma),\\
\gamma_1&=\beta\triangleleft \gamma,&
\gamma_2&=(\alpha \triangleleft \beta)\triangleleft((\alpha\leftarrow \beta)\triangleleft \gamma).
\end{align*}
By (\ref{eq1}), (\ref{eq6}) and (\ref{eq7}), $(\alpha_1,\beta_1,\gamma_1)=(\alpha_2,\beta_2,\gamma_2)$.\\

We denote the inverse of $\varphi_\leftarrow$ by $\psi_\leftarrow=(\psi^1_\leftarrow,\psi^2_\leftarrow)$.
\[\xymatrix{\vdash_\gamma\circ (\dashv_\beta\circ (\dashv_\alpha,I))\ar[r] \ar[d]&
\vdash_\gamma\circ(\dashv_{\alpha \leftarrow \beta}\circ(I,\dashv_{\alpha\triangleleft \beta}))\ar[d]\\
\vdash_{\psi_\leftarrow^1(\gamma,\beta)}\circ (\dashv_\alpha,\vdash_{\psi_\leftarrow^2(\gamma,\beta)})\ar[d]
&\vdash_{\psi_\leftarrow^1(\gamma,\alpha\leftarrow \beta)}\circ (I,\vdash_{\psi_\leftarrow^2(\gamma,\alpha\leftarrow \beta)}
\circ(\dashv_{\alpha\triangleleft \beta},I))\ar[d]\\
\vdash_{\alpha_1}\circ (I,\vdash_{\beta_1}\circ(I,\vdash_{\gamma_1}))\ar@{=}[r]^?
&\vdash_{\alpha_2}\circ(I,\vdash_{ \beta_2} \circ(I,\vdash_{\gamma_2} ))}\]
with:
\begin{align*}
\alpha_1&=\psi_\leftarrow^1(\psi_\leftarrow^1(\gamma,\beta),\alpha),&
\alpha_2&=\psi_\leftarrow^1(\gamma,\alpha\leftarrow \beta),\\
\beta_1&=\psi_\leftarrow^1(\psi_\leftarrow^1(\gamma,\beta),\alpha),&
\beta_2&=\psi_\leftarrow^1(\psi_\leftarrow^2(\gamma,\alpha\leftarrow \beta),\alpha \triangleleft \beta),\\
\gamma_1&=\psi_\leftarrow^2(\gamma,\beta),&
\gamma_2&= \psi_\leftarrow^1(\psi_\leftarrow^2(\gamma,\alpha\leftarrow \beta),\alpha \triangleleft \beta).
\end{align*}
By definition:
\[(\varphi_\leftarrow\otimes Id)\circ (Id \otimes \varphi_\leftarrow)(\alpha_2,\beta_2,\gamma_2)
=(\gamma,\alpha\leftarrow \beta,\alpha \triangleleft \beta).\]
Let us compute 
\[(\varphi_\leftarrow\otimes Id)\circ (Id \otimes \varphi_\leftarrow)(\alpha_1,\beta_1,\gamma_1)
=(\alpha'_1,\beta'_1,\gamma'_1).\]
We obtain, by (\ref{eq1}):
\begin{align*}
\alpha'_1&=\alpha_1\leftarrow (\beta \leftarrow \gamma_1)
=(\alpha_1\leftarrow \beta_1)\leftarrow \gamma_1
=\psi_1(\gamma,\beta)\leftarrow \psi_2(\gamma,\beta)
=\gamma.
\end{align*}
Moreover, by (\ref{eq6}):
\begin{align*}
\beta'_1&=\alpha_1\triangleleft(\beta_1\leftarrow \gamma_1)
=(\alpha_1\triangleleft \beta_1)\leftarrow ((\alpha_1\leftarrow \beta_1)\triangleleft \gamma_1)
=\alpha \leftarrow(\psi_\leftarrow^1(\gamma,\beta)\triangleleft \psi_\leftarrow^2(\gamma,\beta))
=\alpha \leftarrow \beta.
\end{align*}
By (\ref{eq7}):
\begin{align*}
\gamma'_1&=\beta_1\triangleleft \gamma_1
= (\alpha_1 \triangleleft \beta_1)\triangleleft ((\alpha_1 \leftarrow \beta_1) \triangleleft \gamma_1
=\alpha \triangleleft(\psi_\leftarrow^1(\gamma,\beta)\triangleleft \psi_\leftarrow^2(\gamma,\beta))
=\alpha \triangleleft \beta.
\end{align*}
So:
\[(\varphi_\leftarrow\otimes Id)\circ (Id \otimes \varphi_\leftarrow)(\alpha_1,\beta_1,\gamma_1)
=(\varphi_\leftarrow\otimes Id)\circ (Id \otimes \varphi_\leftarrow)(\alpha_2,\beta_2,\gamma_2).\]
As $\varphi_\leftarrow$ is injective, $(\alpha_1,\beta_1,\gamma_1)=(\alpha_2,\beta_2,\gamma_2)$. 
\end{proof}

\begin{example}
The first point implies that the operads associated to the  EDS
A1, A2, B1, C1, C2, C4, C5, D1, E1, E2, G1, G2 and H1 are not Koszul. The second point implies that
the operads associated to the EDS F2, F3, F4, F5, H2 are Koszul.
We do not know if the operads associated to B2, C3, D2, E3 and G3 are Koszul or not.
\end{example}

\section{Combinatorial description of the products}

\subsection{On typed trees}

\begin{defi}
Let $k\geqslant 0$, $\alpha_2,\ldots,\alpha_k \in \Omega$, $\beta_1,\ldots,\beta_k\in \Omega\sqcup\{\emptyset\}$
and $T_1,\ldots,T_k\in \Omega$, with the convention that $\beta_i=\emptyset$ if and only if $T_i=\bun$.
We put:
\begin{align*}
R_{(\alpha_2,\ldots,\alpha_k)}(\beta_1\otimes T_1,\ldots, \beta_k\otimes T_k)&=\begin{cases}
\bun\mbox{ if }k=0, \\
\displaystyle T_1\Y{\alpha_1}{\alpha_2} R_{(\alpha_3,\ldots,\alpha_k)}(\beta_2\otimes T_2,\ldots, \beta_k\otimes T_k)
\mbox{ if }k\geqslant 2;
\end{cases}\\
L_{(\alpha_2,\ldots,\alpha_k)}(\beta_1\otimes T_1,\ldots, \beta_k\otimes T_k)
&=\begin{cases}
\bun\mbox{ if }k=0, \\
\displaystyle L_{(\alpha_3,\ldots,\alpha_k)}(\beta_2\otimes T_2,\ldots, \beta_k\otimes T_k)
\Y{\alpha_2}{\alpha_1} T_1\mbox{ if }\geqslant 2.
\end{cases}
\end{align*}
\end{defi}

Let us denote by $L_k$ the ladder of length $k$:
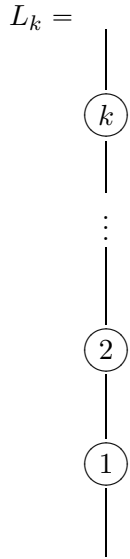
\[L_k=\xymatrix{\\
\rond{k}\ar@{-}[u]\\
\vdots\ar@{-}[u]\\
\rond{2}\ar@{-}[u]\\
\rond{1}\ar@{-}[u]\\
\ar@{-}[u]}\]
Roughly speaking, $R_{(\alpha_2,\ldots,\alpha_k)}(\beta_1\otimes T_1,\ldots, \beta_k\otimes T_k)$,
(respectively $L_{(\alpha_2,\ldots,\alpha_k)}(\beta_1\otimes T_1,\ldots, \beta_k\otimes T_k)$)
is obtained by grafting $T_i$ on the vertex $i$ of $L_k$ for any $i$ on the left (respectively on the right).
The type of the edge from $i$ to the root of $T_i$ is $\beta_i$; the type of the root between the vertex $i-1$ and the vertex $i$
is $\alpha_i$ for any $i\geqslant 2$. \\

Note that any tree $T$ in $\calT_\Omega$ can uniquely written under the form
\[T=R_{(\alpha_2,\ldots,\alpha_k)}(\beta_1\otimes T_1,\ldots, \beta_k\otimes T_k).\]
This is the right comb decomposition of $T$. It can also be uniquely written under the form
\[T=L_{(\alpha'_2,\ldots,\alpha'_l)}(\beta'_1\otimes T'_1,\ldots, \beta'_l\otimes T'_l).\]
This is the left comb decomposition of $T$. 

\begin{defi}
Let $k,l\geqslant 0$. A $(k,l)$-shuffle is a permutation $\sigma \in \mathfrak{S}_{k+l}$ such that
\begin{align*}
&\sigma(1)<\ldots<\sigma(k),&
\sigma(k+1)<\ldots<\sigma(k+l).
\end{align*}
The set of $(k,l)$-shuffle will be denoted by $\sh(k,l)$. If $\sigma \in \sh(k,l)$, then $\sigma^{-1}(1) \in \{1,k+1\}$:
we put
\begin{align*}
\sh_\prec(k,l)&=\{\sigma \in \sh(k,l),\sigma^{-1}(1)=1\},&
\sh_\succ(k,l)&=\{\sigma \in \sh(k,l),\sigma^{-1}(1)=k+1\}.
\end{align*}
\end{defi}

\begin{notation}
Let $k,l\geqslant 0$, $\sigma\in \sh(k,l)$, $\alpha_2,\ldots,\alpha_{k+l}\in \Omega$,
$\beta_1,\ldots,\beta_{k+l}\in \Omega\sqcup\{\emptyset\}$
and $T_1,\ldots,T_{k+l}\in \Omega$, with the convention that $\beta_i=\emptyset$ if and only if $T_i=\bun$.
Let
\[T_{\sigma}^{(\alpha_2,\ldots,\alpha_{k+l})}(\beta_1\otimes T_1,\ldots,\beta_k\otimes T_k;
\beta_{k+1}\otimes T_{k+1},\ldots,\beta_{k+l}\otimes T_{k+l})\]
be the typed tree obtained in the following process: starting form the ladder $L_{k+l}$,
\begin{itemize}
\item For any $1\leqslant i\leqslant k+l$, graft $T_{\sigma^{-1}(i)}$ on the vertex $i$,
on the left if $\sigma^{-1}(i)\leqslant k$, and on right otherwise.
\item The type of the edge between the vertex $i$ and the root of $T_{\sigma^{-1}(i)}$ is $\beta_{\sigma^{-1}(i)}$.
\item The type of the edge between the vertex $i-1$ and $i$ is $\alpha_i$ for any $i\geqslant 2$.
\end{itemize}
\end{notation}

\begin{notation}
Let $k,l\geqslant 0$,with $k+l\geqslant 1$, and $\sigma\in \sh(k,l)$. 
We define a map $D^{k,l}_\sigma:\Omega^{k+l-1}\longrightarrow \Omega^{k+l-1}$:
\begin{itemize}
\item If $k=0$ or $l=0$, then $\sigma=Id_{k+l}$.We put $D^{k,l}_\sigma=Id_{\Omega^{k+l-1}}$.
\item Otherwise:
\begin{itemize}
\item If $\sigma \in \sh_\prec(1,l)$, then $\sigma=Id_{l+1}$, and
\[D^{1,l}_\sigma(\alpha_2,\ldots,\alpha_{k+1})
=(\alpha_2,\ldots,\alpha_{k+1}).\]
\item If $\sigma \in \sh_\prec(k,l)$, with $k\geqslant 2$, let $\sigma'$ be the following permutation:
\[\sigma'=(\sigma(2)-1,\ldots,\sigma(k+l)-1)\in \sh(k-1,l).\]
If $\sigma'\in \sh_\prec(k-1,l)$, for any $(\alpha_2,\ldots,\alpha_{k+l})\in \Omega^{k+l-1}$:
\[D^{k,l}_\sigma(\alpha_2,\ldots,\alpha_{k+l})=\left(\alpha_2\leftarrow \alpha_{k+1},D^{k-1,l}_{\sigma'}
(\alpha_3,\ldots,\alpha_k,\alpha_2\triangleleft \alpha_{k+1},\ldots,\alpha_{k+l})\right).\]
If $\sigma'\in \sh_\succ(k-1,l)$, for any $(\alpha_2,\ldots,\alpha_{k+l})\in \Omega^{k+l-1}$:
\[D^{k,l}_\sigma(\alpha_2,\ldots,\alpha_{k+l})=\left(\alpha_2\rightarrow \alpha_{k+1},D^{k-1,l}_{\sigma'}
(\alpha_3,\ldots,\alpha_k,\alpha_2\triangleright \alpha_{k+1},\ldots,\alpha_{k+l})\right).\]
\item If $\sigma \in \sh_\succ(k,1)$, then $\sigma=(2,3\ldots,k+1,1)$ and
\[D^{k,1}_\sigma(\alpha_2,\ldots,\alpha_{k+1})=(\alpha_{k+1},\alpha_2,\ldots,\alpha_k).\]
\item If $\sigma \in \sh_\succ(k,l)$, with $l\geqslant 2$, let $\sigma'$ be the following permutation:
\[\sigma'=(\sigma(1)-1,\ldots,\sigma(k)-1,\sigma(k+2)-1,\sigma(k+l)-1)\in \sh(k,l-1).\]
If $\sigma'\in \sh_\prec(k,l-1)$, for any $(\alpha_2,\ldots,\alpha_{k+l})\in \Omega^{k+l-1}$:
\[D^{k,l}_\sigma(\alpha_2,\ldots,\alpha_{k+l})=\left(\alpha_{k+1}\leftarrow \alpha_{k+2},D^{k-1,l}_{\sigma'}
(\alpha_2,\ldots,\alpha_k,\alpha_{k+1}\triangleleft \alpha_{k+2},\ldots,\alpha_{k+l})\right).\]
If $\sigma'\in \sh_\succ(k,l-1)$, for any $(\alpha_2,\ldots,\alpha_{k+l})\in \Omega^{k+l-1}$:
\[D^{k,l}_\sigma(\alpha_2,\ldots,\alpha_{k+l})=\left(\alpha_{k+1}\rightarrow \alpha_{k+2},D^{k-1,l}_{\sigma'}
(\alpha_2,\ldots,\alpha_k,\alpha_{k+1}\triangleright \alpha_{k+2},\ldots,\alpha_{k+l})\right).\]
\end{itemize}
\end{itemize}\end{notation}

\begin{example}
\begin{align*}
D^{2,1}_{(123)}(\alpha_2,\alpha_3)&=(\alpha_2\leftarrow \alpha_3,\alpha_2\triangleleft \alpha_3),&
D^{1,2}_{(213)}(\alpha_2,\alpha_3)&=(\alpha_2\leftarrow \alpha_3,\alpha_2\triangleleft \alpha_3),\\
D^{2,1}_{(132)}(\alpha_2,\alpha_3)&=(\alpha_2\rightarrow \alpha_3,\alpha_2\triangleright \alpha_3),&
D^{1,2}_{(312)}(\alpha_2,\alpha_3)&=(\alpha_2\rightarrow \alpha_3,\alpha_2\triangleright \alpha_3),\\
D^{2,1}_{(231)}(\alpha_2,\alpha_3)&=(\alpha_3,\alpha_2),&
D^{1,2}_{(123)}(\alpha_2,\alpha_3)&=(\alpha_2,\alpha_3).
\end{align*}\end{example}

\begin{prop}\label{prop33}
Let us consider two elements of $\calT_\Omega$:
\begin{align*}
T&=R_{(\alpha_2,\ldots,\alpha_k)}(\beta_1\otimes T_1,\ldots,\beta_k \otimes T_k),\\
T'&=L_{(\alpha_{k+2},\ldots,\alpha_{k+l})}(\beta_{k+1}\otimes T_{k+1},\ldots,\beta_{k+l}\otimes T_{k+l})
\end{align*}
Let $\alpha_{k+1} \in \Omega$.  Then:
\begin{align*}
T\prec_{\alpha_{k+1}} T'&=\sum_{\sigma \in \sh_\prec(k,l)}T_\sigma^{D^{k,l}_\sigma(\alpha_2,\ldots,\alpha_{k+l})}
(\beta_1\otimes T_1,\ldots,\beta_k\otimes T_k;\beta_{k+1}\otimes T_{k+1},\ldots,\beta_{k+l}\otimes T_{k+l}),\\
T\succ_{\alpha _{k+1}}T'&=\sum_{\sigma \in \sh_\succ(k,l)}T_\sigma^{D^{k,l}_\sigma(\alpha_2,\ldots,\alpha_{k+l})}
(\beta_1\otimes T_1,\ldots,\beta_k\otimes T_k;\beta_{k+1}\otimes T_{k+1},\ldots,\beta_{k+l}\otimes T_{k+l}).
\end{align*}
\end{prop}

\begin{proof}
If $k=0$, observe that:
\begin{align*}
\sh_\prec(0,l)&=\emptyset,&\sh_\succ(0,l)&=\{Id_l\},\\
\bun \prec_{\alpha_1} T'&=0,&\bun\succ_{\alpha_1} T'&=T'.
\end{align*}
So the result is immediate if $k=0$. It is proved in the same way if $l=0$. We now assume that $k,l\geqslant 1$, 
and we proceed by induction on $k+l$. There is nothing more to prove if $k+l\leq 1$. Otherwise, by the induction hypothesis,
putting $S=R_{(\alpha_3,\ldots,\alpha_k)}(\beta_2\otimes T_2,\ldots,\beta_k \otimes T_k),$:
\begin{align*}
T\prec_{\alpha_{k+1}} T'&=T_1\Y{\beta_1}{\alpha_2} S\prec_{\alpha_{k+1}}T'\\
&=T_1\Y{\beta_1}{\alpha_2\leftarrow \alpha_{k+1}} (S\prec_{\alpha_2\triangleleft \alpha_{k+1}} T')
+T_1\Y{\beta_1}{\alpha_2\rightarrow \alpha_{k+1}} (S\succ_{\alpha_2\triangleright \alpha_{k+1}} T')\\
&=\sum_{\sigma \in \sh_\prec(k-1,l)} T_1\Y{\beta_1}{\alpha_2\leftarrow \alpha_{k+1}} 
T_\sigma^{D^{k-1,l}_\sigma(\alpha_3,\ldots, \alpha_2\triangleleft \alpha_{k+1},\ldots \alpha_{k+l})}
(\beta_2\otimes T_2,\ldots,\beta_{k+l}\otimes T_{k+l})\\
&+\sum_{\sigma \in \sh_\succ(k-1,l)} T_1\Y{\beta_1}{\alpha_2\rightarrow \alpha_{k+1}} 
T_\sigma^{D^{k-1,l}_\sigma(\alpha_3,\ldots, \alpha_2\triangleright \alpha_{k+1},\ldots \alpha_{k+l})}
(\beta_2\otimes T_2,\ldots,\beta_{k+l}\otimes T_{k+l})\\
&=\sum_{\sigma \in \sh_\prec(k-1,l)} T_{\overline{\sigma}}^{D^{k,l}_{\overline{\sigma}}
(\alpha_2,\ldots,\alpha_{k+l})}(\beta_1\otimes T_1,\ldots,\beta_{k+l}\otimes T_{k+l})\\
&+\sum_{\sigma \in \sh_\succ(k-1,l)} T_{\overline{\sigma}}^{D^{k,l}_{\overline{\sigma}}
(\alpha_2,\ldots,\alpha_{k+l})}(\beta_1\otimes T_1,\ldots,\beta_{k+l}\otimes T_{k+l})\\
&=\sum_{\sigma \in \Sh_(k,l)} T_{\sigma}^{D^{k,l}_\sigma
(\alpha_2,\ldots,\alpha_{k+l})}(\beta_1\otimes T_1,\ldots,\beta_{k+l}\otimes T_{k+l}),
\end{align*}
where $\overline{\sigma}=(1,\sigma(1)+1,\ldots,\sigma(k+l-1)+1)$. The formula for $T\succ_{\alpha_{k+1}} T'$
is proved in the same way. \end{proof}

The formulas for $D^{k,l}_\sigma$ can be simplified when $\triangleleft$ and $\triangleright$ are trivial:

\begin{prop}
Let $(\Omega,\leftarrow,\rightarrow)$ be a diassociative semigroup. We work with the EDS
$\EDS(\Omega)$. Let $k,l\geqslant 0$, $\sigma \in \sh(k,l)$,  and $\alpha_2,\ldots,\alpha_{k+l}\in \Omega$. We put:
\begin{enumerate}
\item For any $2\leqslant i\leqslant k+l$:
\begin{align*}
L_\sigma(i)&=\begin{cases}
\alpha_{k+1}\mbox{ if }i\leqslant \sigma(1),\\
\alpha_p\mbox{ if }\sigma(p-1)<i\leqslant \sigma(p), \mbox{ with }2\leqslant p\leqslant k,\\
\emptyset \mbox{ if }i>\sigma(k);
\end{cases}\\
R_\sigma(i)&=\begin{cases}
\alpha_{k+1}\mbox{ if }i\leqslant \sigma(k+1),\\
\alpha_p\mbox{ if }\sigma(p-1)<i\leqslant \sigma(p), \mbox{ with }k+2\leqslant p\leqslant k+l,\\
\emptyset \mbox{ if }i>\sigma(k+l).
\end{cases}
\end{align*}
\item For any $2\leqslant i\leqslant k+l$,
\[D_\sigma(i)=\begin{cases}
L_\sigma(i)\leftarrow R_\sigma(i) \mbox{ if }\sigma^{-1}(i)\leqslant k,\\
L_\sigma(i)\rightarrow R_\sigma(i) \mbox{ if }\sigma^{-1}(i)> k,
\end{cases}\]
with the convention $\alpha \leftarrow \emptyset=\emptyset \rightarrow \alpha=\alpha$ for any $\alpha \in \Omega$. 
\end{enumerate}
Then:
\[D^{k,l}_\sigma(\alpha_2,\ldots,\alpha_{k+l})
=(D_\sigma(2),\ldots,D_\sigma(k+l)).\]
\end{prop}

\begin{proof} Induction on $k+l$. \end{proof}

\begin{remark}
Working with $\Omega$ reduced to a single element, we obtain the dual description of the coproduct of the Hopf algebra
$\mathcal{Y}Sym$ described in \cite{AguiarSottile}.
\end{remark}

\begin{cor}
Let $\Omega$ be a set. We work in $\EDS(\Omega)$. For any $k,l\geqslant 0$, for any $\sigma \in \sh(k,l)$:
\[D^{k,l}_\sigma(\alpha_2,\ldots,\alpha_{k+l})
=(\alpha_{\sigma^{-1}(2)},\ldots,\alpha_{\sigma^{-1}(k+l)}),\]
with the convention $\alpha_1=\alpha_{k+1}$. 
\end{cor}

\begin{proof}
Let $2\leqslant i\leqslant k+l$. 
If $\sigma^{-1}(i)\leqslant k$, then $D_\sigma(i)=L_\sigma(i)=\alpha_{\sigma^{-1}(i)}$.
If $\sigma^{-1}(i)>k$, then $D_\sigma(i)=R_\sigma(i)=\alpha_{\sigma^{-1}(i)}$.
\end{proof}

\subsection{On typed words}

Remark that:
\[\K\Omega \otimes \Sh_\Omega^+(V)=\bigoplus_{n=1}^\infty (\K\Omega)^{\otimes n}\otimes V^{\otimes n}.\]

\begin{prop}\label{prop36}
Let $\Omega$ be an EDS, and $V$ be a vector space.
For any $\alpha_2,\ldots,\alpha_{k+l}\in \Omega$, for any $v_1,\ldots,v_{k+l}\in V$:
\begin{align*}
&\alpha_2\ldots \alpha_k\otimes v_1\ldots v_k \prec_{\alpha_ {k+1}}
\alpha_{k+2}\ldots \alpha_{k+l}\otimes v_{k+1}\ldots v_{k+l}\\
&=\sum_{\sigma\in \sh_\prec(k,l)}D^{k,l}_\sigma(\alpha_2,\ldots,\alpha_{k+l})
\otimes v_{\sigma^{-1}(1)}\ldots v_{\sigma^{-1}(k+l)},\\
&\alpha_2\ldots \alpha_k\otimes v_1\ldots v_k \succ_{\alpha_ {k+1}}
\alpha_{k+2}\ldots \alpha_{k+l}\otimes v_{k+1}\ldots v_{k+l}\\
&=\sum_{\sigma\in \sh_\succ(k,l)}D^{k,l}_\sigma(\alpha_2,\ldots,\alpha_{k+l})
\otimes v_{\sigma^{-1}(1)}\ldots v_{\sigma^{-1}(k+l)}.
\end{align*}
\end{prop}

\begin{proof} Similar as the proof of Proposition \ref{prop33}.  \end{proof}

\section{Hopf algebraic structure}

\subsection{Existence of dendriform bialgebraic structures}

Let us recall the notion of dendriform bialgebra introduced by Loday and Ronco \cite{LodayRonco2,LodayRonco3,LodayRonco1,Ronco}:

\begin{defi}
A dendriform bialgebra is a family $(A,\prec,\succ,\Delta)$, where $(A,\prec,\succ)$ is a dendriform algebra,
$(A,\Delta)$ a coassociative coalgebra (not necessarily counitary) such that, for any $x,y\in A$:
\begin{align*}
\Delta(x\prec y)&=x\otimes y+x'\prec y\otimes x''+x'\otimes x''\cdot y+x\prec y'\otimes y''+x'\prec y'\otimes x''\cdot y'',\\
\Delta(x\succ y)&=y\otimes x+x'\succ y\otimes x''+y'\otimes x\cdot y''+x\succ y'\otimes y''+x'\succ y'\otimes x'' \cdot y'', 
\end{align*}
where $\cdot=\prec+\succ$ is the associative product associated to $(A,\prec,\succ)$.
We use Sweedler's notations $\Delta(a)=a'\otimes a''$ for any $a\in A$. 
\end{defi}

\begin{prop} \label{prop38}
Let $\Omega$ be an extended disasociative semigroup. 
If there exists a nonzero graded $\Omega$-dendriform algebra $A$, with $A_0=(0)$, with a homogeneous
coproduct $\Delta$ making $\K\Omega \otimes A$ a dendriform bialgebra, then 
$\varphi_\leftarrow$ and $\varphi_\rightarrow$ are injective.
\end{prop}

\begin{proof} Let $a$ be a nonzero element of $A$ of minimal degree $n$. As $n\geqslant 0$, 
necessarily, for any $\alpha \in \Omega$,
\[\Delta(\alpha \otimes a)=0.\]
Let $\alpha,\beta,\alpha',\beta' \in \Omega$, such that $\varphi_\leftarrow(\alpha,\beta)=
\varphi_\leftarrow(\alpha',\beta')$. Then, in $\K\Omega\otimes A$:
\[\alpha \otimes a\prec \beta \otimes a=\alpha \leftarrow \beta \otimes a\prec_{\alpha\triangleleft \beta} a
=\alpha' \leftarrow \beta' \otimes a\prec_{\alpha'\triangleleft \beta'} a=\alpha'\otimes a\prec \beta'\otimes a.\]
Hence, by the compatibility between $\Delta$ and $\prec$:
\begin{align*}
\Delta(\alpha \otimes a\prec \beta \otimes a)&=(\alpha\otimes a)\otimes (\beta \otimes a)
=\Delta(\alpha' \otimes a\prec \beta' \otimes a)=(\alpha'\otimes a)\otimes (\beta' \otimes a).
\end{align*}
As $a\neq 0$, $(\alpha,\beta)=(\alpha',\beta')$. Using the compatibility between $\Delta$ and $\succ$,
we obtain that $\varphi_\rightarrow$ is injective. \end{proof}

\begin{prop} \label{prop39}
If $\Omega$ is a nondegenerate EDS, there exists a unique coproduct $\Delta$ on 
$\K\Omega\otimes \K\calT_\Omega^+$, making it a dendriform bialgebra, such that:
\begin{align*}
&\forall \alpha \in \Omega,&\Delta(\alpha \otimes \bdeux)&=0.
\end{align*}
Moreover, this dendriform bialgebra is graded by the number of internal vertices.
\end{prop}

\begin{proof}
If $\Omega$ is nondegenerate, by Proposition \ref{prop19}, $\K\Omega \otimes \K\calT_\Omega^+$
is freely generated, as a dendriform algebra, by the elements $\alpha \otimes \bdeux$. From \cite{LodayRonco1,LodayRonco2,LodayRonco3,Ronco,AguiarSottile},
such a $\Delta$ exists and is unique.
\end{proof}

\begin{prop}\label{prop40}
Let $\Omega$ be a nondegenerate EDS and let $V$ be a nonzero vector space.
The following conditions are equivalent:
\begin{enumerate}
\item There exists a unique coproduct on $\K\Omega \otimes \Sh^+_\Omega(V)$, making it a dendrifrom bialgebra,
such that for any $v\in V$, for any $\alpha \in \Omega$, $\Delta(\alpha \otimes v)=0$. 
\item $\Omega$ is commutative.
\end{enumerate}
\end{prop}

\begin{proof}
$1.\Longrightarrow 2$.
Let $(\alpha,\beta)\in \Omega^2$. As $\Omega$ is nondegenerate, there exists a unique $(\alpha',\beta')\in \Omega^2$,
such that $\varphi_\leftarrow(\alpha,\beta)=\varphi_\rightarrow(\alpha',\beta')$. Let $v$ be a nonzero element of $V$.
By construction of $(\alpha',\beta')$:
\begin{align*}
\alpha \otimes v\prec \beta \otimes v&=\alpha \leftarrow \beta \otimes (\alpha \triangleleft \beta \otimes vv)
=\alpha'\rightarrow \beta'\otimes (\alpha'\triangleright \beta' \otimes vv)
=\alpha'\otimes v\succ \beta'\otimes v.
\end{align*}
Hence:
\begin{align*}
\Delta(\alpha \otimes v\prec \beta \otimes v)&=(\alpha \otimes v)\otimes (\beta \otimes v)
=\Delta(\alpha'\otimes v\succ \beta'\otimes v)
=(\beta'\otimes v)\otimes (\alpha'\otimes v).
\end{align*}
As $v\neq 0$, $(\alpha',\beta')=(\beta,\alpha)$, so, by definition of $(\alpha',\beta')$:
\begin{align*}
\beta \rightarrow \alpha&=\alpha \leftarrow \beta,&
\beta \triangleright \alpha&=\alpha \triangleleft \beta.
\end{align*}
Therefore, $\Omega$ is commutative.\\

$2.\Longrightarrow 1$. By Proposition \ref{prop20}, $\K\Omega\otimes \Sh^+_\Omega(V)$
is freely generated, as a commutative dendriform algebra, by $\K\Omega \otimes V$.
From \cite{AguiarSottile,LodayRonco1,Foissy1}, such a $\Delta$ exists and is unique. \end{proof}

\subsection{Combinatorial description of the coproducts on typed trees}

Let us generalize the combinatorial description of the coproduct given in \cite{GaoZhang}. 
We work with $\K\Omega \otimes \calT^+_\Omega$ where $\Omega$ is a nondegenerate EDS.
We shall use the notations of Proposition \ref{prop6}. 
%For commodity reasons, for any $\alpha\in \Omega$ and
%$T\in \calT_\Omega$, we shall consider $\alpha \otimes T$ as a typed binary tree which root is also typed, of type $\alpha$.

\begin{notation}\begin{enumerate}
\item For any $\alpha\in \Omega$, let $\varphi^\blacktriangleleft_\alpha$ and 
$\varphi^\blacktriangleright_\alpha:\Omega\longrightarrow\Omega$
defined by:
\begin{align*}
\varphi^\blacktriangleleft_\alpha(\beta)&=\beta\blacktriangleleft \alpha,&
\varphi^\blacktriangleright_\alpha(\beta)&=\alpha \blacktriangleright \beta.
\end{align*}
\item Let $T\in \calT^+_\Omega$ and let $\alpha_0\in \Omega$. Let us choose an internal edge $e$ of $T$.
\begin{enumerate}
\item We denote by $T_e$ the typed plane binary subtree of $T$ formed by all the vertices of $T$ which are born from $e$.
\item Let $e_1,\ldots,e_k$ be the internal edges on the unique path in $T$ from its root to the extremity of $e$;
in particular, $e_k=e$. For any $i$, let $\alpha_i$ be the type of $e_i$, and:
\begin{itemize}
\item $\blacksquare_i=\blacktriangleleft$ if $e_i$ is a right edge;
\item $\blacksquare_i=\blacktriangleright$ if $e_i$ is a left edge.
\end{itemize}
We then put:
\[T_e(\alpha_0)=\varphi_{\alpha_k}^{\blacksquare_k}\circ \ldots \circ \varphi_{\alpha_1}^{\blacksquare_1}(\alpha_0)\otimes 
T_e \in \K\Omega \otimes \K\calT_\Omega^+.\]
\end{enumerate}\end{enumerate} \end{notation}

\begin{defi}
Let $T\in \calT_\Omega^+$. 
\begin{enumerate}
\item A cut of $T$ is a nonempty subset $c$ of the set of internal edges of $T$. 
\item A cut of $T$ is \emph{admissible} if any path in the tree meets at most one element of $c$. 
The set of admissible cuts of $T$ is denoted by $\adm(T)$.  Note that if $c$ is an admissible cut of $T$,
its elements are naturally ordered from left to right, and we shall write $c=\{e_1<\ldots<e_k\}$. 
\item Let $c$ be an admissible cut of $T$. The typed plane binary subtree obtained from $T$
by deleting $T_e$ for any $e\in c$ is denoted by $R^c(T)$. 
\end{enumerate}
\end{defi}

\begin{prop}\label{prop42}
Let $\Omega$ be a nondegenerate EDS, and let $T\in \calT_\Omega^+$ and $\alpha \in \Omega$. 
Then:
\[\Delta(\alpha \otimes T)=\sum_{c=\{e_1<\ldots<e_k\}\in \adm(T)} (\alpha \otimes R^c(T))\otimes 
(T_{e_1}(\alpha)\cdot\ldots\cdot T_{e_k}(\alpha)),\]
where $\cdot=\prec+\succ$. 
\end{prop}

\begin{proof} \textit{First step.}
Let us first prove that, for any admissible cut $c \in \adm(T)$, there exists a tree $\iota_c(\alpha \otimes R^c(T))$ obtained from 
$\alpha \otimes R^c(T)$ by an action on the types of the internal edges and on $\alpha$, such that:
 \[\Delta(\alpha \otimes T)=\sum_{c=\{e_1<\ldots<e_k\}\in \adm(T)} \iota_c(\alpha \otimes R^c(T))\otimes 
(T_{e_1}(\alpha)\cdot\ldots\cdot T_{e_k}(\alpha)).\]

For any trees $T_1,T_2$, for any $\alpha,\beta \in \Omega$:
\begin{align*}
\alpha \otimes \bun \Y{\emptyset}{\gamma}{T_2}
&=\beta \curvearrowleft \gamma \otimes \bdeux\prec \beta \blacktriangleleft \gamma \otimes T_2,\\
\alpha \otimes T_1\Y{\beta}{\emptyset} \bun&=\beta \blacktriangleright \alpha \otimes T_1\succ
\beta \curvearrowright \alpha\otimes \bdeux,\\
\alpha \otimes T_1\Y{\beta}{\gamma} T_2&=
\beta \blacktriangleright \alpha \otimes T_1\succ (\beta \curvearrowright \alpha)\curvearrowleft \gamma \otimes\bdeux
\prec \alpha \blacktriangleleft \gamma \otimes T_2.
\end{align*}
Remark that if $T=T_1\Y{\alpha}{\beta}T_2$, then any admissible cut $c$ of $T$ is of the form $c_1\sqcup c_2$,
where $c_i$ is either an admissible cut, or the empty cut, or the total cut (which means that $R^c(T)=\emptyset$), of $c_i$;
at least $c_1$ or $c_2$ is not empty. Then $\iota_c(\alpha \otimes R^c(T))$ is inductively defined by:
\begin{itemize}
\item If $c$ is total, then $\iota_c(\alpha \otimes R^c(T))=\emptyset \otimes \bun$.
\item If $c$ is empty, then $\iota_c(\alpha \otimes R^c(T))=\alpha \otimes T$.
\item If $c=c_1\sqcup c_2$, then:
\[\iota_c(\alpha \otimes R^c(T))=
\iota_{c_1}(\beta \blacktriangleright \alpha \otimes T_1)\succ
(\beta \curvearrowright \alpha)\curvearrowleft \gamma \otimes \bdeux\prec
\iota_{c_2}(\alpha \blacktriangleleft \gamma\otimes T_2).\]
\end{itemize}
Using the compatibilities between the dendriform products and the coproducts, we obtain the result by induction on
the number of internal vertices of $T$.  \\

\textit{Second step.} Let us prove that $\iota_c(\alpha \otimes R^c(T))=\alpha \otimes R^c(T)$ by induction on the number
$n$ of internal vertices of $T$. It is obvious if $n=1$. Otherwise, we put $\displaystyle T=T_1\Y{\beta}{\gamma} T_2$,
and $c=c_1\sqcup c_2$. Then, using the induction hypothesis on $T_1$ and $T_2$:
\begin{align*}
\iota_c(\alpha \otimes R^c(T))&=\iota_{c_1}(\beta \blacktriangleright \alpha \otimes R^{c_1}(T_1))\succ
(\beta \curvearrowright \alpha)\curvearrowleft \gamma \otimes \bdeux\prec
\iota_{c_2}(\alpha \blacktriangleleft \gamma\otimes R^{c_2}(T_2))\\
&=\beta \blacktriangleright \alpha \otimes R^{c_1}(T_1)\succ
(\beta \curvearrowright \alpha)\curvearrowleft \gamma \otimes \bdeux\prec
\alpha \blacktriangleleft \gamma\otimes R^{c_2}(T_2)\\
&=(\beta \blacktriangleright \alpha)\rightarrow ((\beta \curvearrowright \alpha)\curvearrowleft \gamma)
\otimes R^{c_1}(T_1)\succ_{(\beta \blacktriangleright \alpha)\triangleright ((\beta \curvearrowright \alpha)\curvearrowleft \gamma) }
\bdeux \prec \alpha \blacktriangleleft \gamma\otimes R^{c_2}(T_2).
\end{align*}
By (\ref{eq34}):
\[(\varphi_\rightarrow \otimes Id)\circ (Id \otimes \varphi_\leftarrow^{-1})=(\tau\otimes Id)
\circ (Id \otimes \varphi_\leftarrow^{-1}) \circ (\tau \otimes Id)\circ (\varphi_\leftarrow\otimes Id),\]
so, for any $\alpha',\beta',\gamma'\in \Omega$:
\[(\alpha'\rightarrow(\beta'\curvearrowleft \gamma'),\alpha'\triangleright(\beta' \curvearrowleft \gamma'),
\beta'\blacktriangleleft \gamma')=(
(\alpha' \rightarrow \beta')\curvearrowleft \gamma',\alpha'\triangleright \beta',
(\alpha' \rightarrow \beta')\blacktriangleleft \gamma').\]
For $\alpha'=\beta \blacktriangleright \alpha$, $\beta'=\beta \curvearrowright \alpha$ and $\gamma'=\gamma$, we obtain:
\begin{align*}
(\beta \blacktriangleright \alpha)\rightarrow ((\beta \curvearrowright \alpha)\curvearrowleft \gamma)
&=((\beta \blacktriangleright \alpha)\rightarrow (\beta \curvearrowright \alpha))\curvearrowleft \gamma
=\alpha \curvearrowleft \gamma,\\
(\beta \blacktriangleright \alpha)\triangleright((\beta \curvearrowright \alpha)\curvearrowleft \gamma)
&=(\beta \blacktriangleright \alpha)\triangleright(\beta \curvearrowright \alpha)=\beta.
\end{align*}
Hence:
\begin{align*}
\iota_c(\alpha \otimes R^c(T))&=\alpha \curvearrowleft \gamma\otimes
R^{c_1}(T_1)\succ_{\beta} \bdeux\prec \alpha \blacktriangleleft \gamma\otimes R^{c_2}(T_2)\\
&=\alpha \curvearrowleft \gamma\otimes R^{c_1}(T_1)\Y{\beta}{\emptyset} \bun\prec \alpha \blacktriangleleft \gamma\otimes R^{c_2}(T_2)\\
&=(\alpha \curvearrowleft \gamma)\leftarrow ( \alpha \blacktriangleleft \gamma)
\otimes R^{c_1}(T_1)\Y{\beta}{\emptyset} \bun \prec_{(\alpha \curvearrowleft \gamma)\triangleleft( \alpha \blacktriangleleft \gamma)}
R^{c_2}(T_2)\\
&=\alpha \otimes R^{c_1}(T_1) \Y{\beta}{\emptyset}\bun \prec_{\gamma} R^{c_2}(T_2)\\
&=\alpha \otimes R^{c_1}(T_1)\Y{\beta}{\gamma} R^{c_2}(T_2)\\
&=\alpha \otimes R^c(T).
\end{align*}
Hence, $\iota_c(\alpha \otimes R^c(T))=\alpha \otimes R^c(T)$ for any admissible cut of any tree $T$. \end{proof}

\begin{remark}
Working with $\Omega$ reduced to a single element, we obtain the dual description of the product of the Hopf algebra
$\mathcal{Y}Sym$ described in \cite{AguiarSottile}.
\end{remark}

\begin{example}\begin{enumerate}
\item Let $(\Omega,\leftarrow,\rightarrow)$ be a diassociative monoid. We assume that 
$\EDS(\Omega,\leftarrow,\rightarrow)$ is nondegenerate. In this case, for any $\alpha$, $\beta \in \Omega$,
$\varphi_\alpha^\triangleleft(\beta)=\varphi_\alpha^\triangleright(\beta)=\alpha$.
Hence, for any edge $e$ of a given tree $T$, of type $\alpha_e$, for any $\alpha_0\in \Omega$:
\[T_e(\alpha_0)=\alpha_e \otimes T_e.\]
\item Let $(\Omega,\star)$ be a group. In $\EDS^*(\Omega,\star)$, for any $\alpha$, $\beta \in \Omega$,
$\varphi_\alpha^\triangleleft(\beta)=\varphi_\alpha^\triangleright(\beta)=\beta\star\alpha$.
Hence, for any edge $e$ of a given tree $T$, of type $\alpha_e$, for any $\alpha_0\in \Omega$:
\[T_e(\alpha_0)=\alpha_0\star \ldots \star \alpha_k \otimes T_e.\]
\end{enumerate}\end{example}

\subsection{Combinatorial description of the coproducts on typed words}

\begin{prop} \label{prop43}
Let $\Omega$ be a nondegenerate commutative EDS.
We use the notations of Proposition \ref{prop6}.  In the dendriform bialgebra $\K\Omega \otimes \Sh^+_\Omega(V)$,
for any $\alpha_1,\ldots,\alpha_n\in \Omega$, for any $v_1,\ldots,v_n \in V$:
\begin{align*}
&\Delta(\alpha_1\ldots\alpha_n\otimes v_1\ldots v_n)\\
&=\sum_{i=1}^{n-1} \left(\alpha_1\ldots \alpha_i
\otimes v_1\ldots v_i\right)
\otimes\left((\alpha_1\blacktriangleleft\ldots\blacktriangleleft \alpha_{i+1})\alpha_{i+2}\ldots \alpha_n\otimes v_{i+1}\ldots v_n\right).
\end{align*}\end{prop}

\begin{proof}
We work with the $\Omega$-dendriform algebra $A$ of $\Omega$-typed plane binary trees which internal vertices are decorated
by $V$. All the results presented for nondecorated trees can be extended to this context. For any $v\in V$,
we denote by $\bdeux(v)$ the plane binary tree $\bdeux$ which unique internal vertex is decorated by $v$.
By freeness, there exists a unique $\Omega$-dendriform algebra morphism $\Phi$ from $A$ to $\Sh_\Omega^+(V)$,
sending $\bdeux(v)$ to $v$ for any $v\in V$. It naturally induces a dendriform algebra morphism from
$\K\Omega \otimes A$ to $\K\Omega \otimes \Sh_\Omega^+(V)$, also denoted by $\Phi$. 
For any $v\in V$, any $\alpha \in \Omega$, $\alpha\otimes \bdeux(v)$ is primitive in $\K\Omega\otimes A$, and 
$\alpha \otimes v$ is primitive in $\K\Omega \otimes \Sh_\Omega^+(V)$: this implies that $\Phi$ is a dendriform bialgebra
morphism. \\

Let us introduce some notations. For any $n\geqslant n$ and $\alpha_2,\ldots,\alpha_n\in \Omega$, 
$R_n(\alpha2\ldots\alpha_n)$ is inductively defined by:
\begin{align*}
R_1&=\bdeux,&
R_n(\alpha_2\ldots \alpha_n)&=\bun\Y{\emptyset}{\alpha_2} R_{n-1}(\alpha_3\ldots \alpha_n) \mbox{ if }n\geqslant 2.
\end{align*}
Note that the underlying plane binary tree of $R_n(\alpha_2\ldots \alpha_n)$ is a right comb.  For example:
\begin{align*}
R_2(\alpha_2)&=\bdtroisdeux(\alpha_2),&R_3(\alpha_2\alpha_3)&=\bdquatrequatre(\alpha_2,\alpha_3).
\end{align*}
For $v_1,\ldots,v_n \in V$ we denote by $R_n(\alpha_2\ldots\alpha_n;v_1\ldots v_n)$ 
the $\Omega$-typed plane binary tree by giving the $n$ internal vertices of $R_n(\alpha_2\ldots \alpha_n)$,
naturally ordered starting from the root, the decorations $v_1,\ldots,v_n$. \\

By definition of the products on trees:
\begin{align*}
&(\alpha \curvearrowleft \beta) \otimes \bdeux(v) \prec (\alpha \blacktriangleleft \beta) \otimes R_n(\alpha_2\ldots \alpha_n;
v_1\ldots v_n)\\
&=(\alpha \curvearrowleft \beta) \leftarrow  (\alpha \blacktriangleleft \beta)\otimes 
\bdeux(v)\prec_{(\alpha \curvearrowleft \beta) \triangleleft  (\alpha \blacktriangleleft \beta)}
R_n(\alpha_2\ldots \alpha_n;v_1\ldots v_n)\\
&=\alpha\otimes \bdeux(v)\prec_\beta R_n(\alpha_2\ldots \alpha_n;v_1\ldots v_n)\\
&=\alpha\otimes \bun \Y{\emptyset}{\beta}R_n(\alpha_2\ldots \alpha_n;v_1\ldots v_n)\\
&=\alpha\otimes  R_{n+1}(\beta\alpha_2\ldots \alpha_n;vv_1\ldots v_n).
\end{align*}
An easy induction proves that:
\[\Phi(\alpha_1 \otimes R_n(\alpha_2\ldots\alpha_n;v_1\ldots v_n)) =\alpha_1\ldots \alpha_n\otimes v_1\ldots v_n.\]
The admissible cuts of $R_n$ are the cuts of a single internal edge: hence, by Proposition \ref{prop42},
\begin{align*}
&\Delta(\alpha_1 \otimes R_n(\alpha_2\ldots\alpha_n;v_1\ldots v_n))\\
&=\sum_{i=1}^{n-1}(\alpha_1\otimes R_i(\alpha_2\ldots \alpha_i;v_1\ldots v_i))\otimes
(\varphi_{\alpha_{i+1}}^\blacktriangleleft\circ\ldots \circ \varphi_{\alpha_2}^\blacktriangleleft(\alpha_1)
\otimes R_{n-i}(\alpha_{i+2}\ldots \alpha_n;v_{i+1}\ldots v_n))\\
&=\sum_{i=1}^{n-1}(\alpha_1\otimes R_i(\alpha_2\ldots \alpha_i;v_1\ldots v_i))\otimes
(\alpha_1\blacktriangleleft\ldots \blacktriangleleft \alpha_{i+1}
\otimes R_{n-i}(\alpha_{i+2}\ldots \alpha_n;v_{i+1}\ldots v_n)).
\end{align*}
The result is obtained by application of $\Phi$.  \end{proof}

\bibliographystyle{amsplain}
\addcontentsline{toc}{section}{References}
\bibliography{biblio}

\providecommand{\bysame}{\leavevmode\hbox to3em{\hrulefill}\thinspace}
\providecommand{\MR}{\relax\ifhmode\unskip\space\fi MR }
\providecommand{\MRhref}[2]{%
  \href{http://www.ams.org/mathscinet-getitem?mr=#1}{#2}
}
\providecommand{\href}[2]{#2}
\begin{thebibliography}{10}

\bibitem{AguiarSottile}
Marcelo Aguiar and Frank Sottile, \emph{Structure of the {L}oday-{R}onco {H}opf
  algebra of trees}, J. Algebra \textbf{295} (2006), no.~2, 473--511.
  \MR{2194965}

\bibitem{Dotsenko}
Murray~R. Bremner and Vladimir Dotsenko, \emph{Algebraic operads}, CRC Press,
  Boca Raton, FL, 2016, An algorithmic companion. \MR{3642294}

\bibitem{Hairer}
Y.~Bruned, M.~Hairer, and L.~Zambotti, \emph{Algebraic renormalisation of
  regularity structures}, Invent. Math. \textbf{215} (2019), no.~3, 1039--1156.
  \MR{3935036}

\bibitem{Foissy1}
L.~Foissy, \emph{Les alg\`ebres de {H}opf des arbres enracin\'{e}s
  d\'{e}cor\'{e}s. {II}}, Bull. Sci. Math. \textbf{126} (2002), no.~4,
  249--288. \MR{1909461}

\bibitem{FoissyPrelie}
Lo{\"\i}c Foissy, \emph{Algebraic structures on typed decorated rooted trees},
  arXiv:1811.07572, 2018.

\bibitem{FoissyPatras}
Lo{\"\i}c Foissy and Fr\'ed\'eric Patras, \emph{Lie theory for quasi-shuffle
  bialgebras}, arXiv:1605.02444, 2016.

\bibitem{LodayRonco2}
Jean-Louis Loday and Mar\'{\i}a Ronco, \emph{Trialgebras and families of
  polytopes}, Homotopy theory: relations with algebraic geometry, group
  cohomology, and algebraic {$K$}-theory, Contemp. Math., vol. 346, Amer. Math.
  Soc., Providence, RI, 2004, pp.~369--398. \MR{2066507}

\bibitem{LodayRonco3}
\bysame, \emph{Combinatorial {H}opf algebras}, Quanta of maths, Clay Math.
  Proc., vol.~11, Amer. Math. Soc., Providence, RI, 2010, pp.~347--383.
  \MR{2732058}

\bibitem{LodayRonco1}
Jean-Louis Loday and Mar\'{\i}a~O. Ronco, \emph{Hopf algebra of the planar
  binary trees}, Adv. Math. \textbf{139} (1998), no.~2, 293--309. \MR{1654173}

\bibitem{Ronco}
Mar\'{\i}a Ronco, \emph{A {M}ilnor-{M}oore theorem for dendriform {H}opf
  algebras}, C. R. Acad. Sci. Paris S\'{e}r. I Math. \textbf{332} (2001),
  no.~2, 109--114. \MR{1813766}

\bibitem{Schutzenberger}
M.~P. Sch\"utzenberger, \emph{Sur une propri\'et\'e combinatoire des alg\`ebres
  de {L}ie libres pouvant \^etre utilis\'ee dans un probl\`eme de
  math\'ematiques appliqu\'ees}, S\'eminaire Dubreil--Jacotin Pisot (Alg\`ebre
  et th\'eorie des nombres) (1958/59).

\bibitem{GaoZhang}
Yi~Zhang and Xing Gao, \emph{Hopf algebras of planar binary trees: an operated
  algebra approach}, Journal of Algebraic Combinatorics \textbf{??} (2019),
  ??--??

\bibitem{GaoGuoZhang}
Yi~Zhang, Xing Gao, and Li~Guo, \emph{Matching {R}ota-{B}axter algebras,
  matching dendriform algebras and matching pre-{L}ie algebras}, J. Algebra
  \textbf{??} (2019), ??--??

\bibitem{ZhangGaoManchon}
Yuanyuan Zhang, Xing Gao, and Dominique Manchon, \emph{Free (tri)dendriform
  family algebras}, J. Algebra \textbf{547} (2020), 456--493, arXiv:1909.08946.

\end{thebibliography}

\end{document}